\documentclass[11pt]{article}
\usepackage[margin=1in]{geometry}
\usepackage[T1]{fontenc}
\usepackage{lmodern,amsmath,amssymb,amsthm,mathtools}
\usepackage{xcolor}
\usepackage[citecolor=red]{hyperref}

\allowdisplaybreaks[2]
\numberwithin{equation}{section}
\newtheorem{theorem}{Theorem}[section]
\newtheorem{lemma}[theorem]{Lemma}
\newtheorem{proposition}[theorem]{Proposition}
\newtheorem{remark}[theorem]{Remark}
\newtheorem{corollary}[theorem]{Corollary}

\newcommand{\eps}{\epsilon}

\newcommand{\T}{\mathbb T}
\newcommand{\R}{\mathbb R}
\newcommand{\pa}{\partial}
\newcommand{\dd}{\,\mathrm d}
\newcommand{\av}[1]{\big\langle #1\big\rangle}
\newcommand{\bY}{\bar Y}
\newcommand{\tY}{\widetilde Y}
\newcommand{\Db}{\partial_{b_0^1}}
\newcommand{\HH}{\mathcal H_a}
\newcommand{\LL}{\mathcal L_a}
\newcommand{\WW}{\mathcal W_0}
\newcommand{\SSS}{\mathcal S_a}
\DeclareMathOperator{\diver}{div}

\newtheorem{thm}{Theorem}[section]

\theoremstyle{plain}

\newtheorem{rem}[thm]{Remark}

\let\p=\partial
\let\f=\frac

\def\na{\nabla}
\def\bR{\mathbb{R}}
\def\bT{\mathbb{T}}
\def\bZ{\mathbb{Z}}
\def\rmd{\mathrm{d}}
\def\rme{\mathrm{e}}

\begin{document}
\title{Global nonlinear stability of the 2D incompressible viscous non-resistive MHD under sheared magnetic field}
\author{Yuan Cai\footnote{School of Mathematical Sciences; LMNS and Shanghai Key Laboratory for Contemporary Applied Math,
Fudan University, Shanghai 200433, P. R. China.  Email: caiy@fudan.edu.cn}
		\and
		Bin Han\footnote{School of Mathematics and Statistics, Donghua  University,
			Shanghai, 201620, P. R. China. Email: hanbin@dhu.edu.cn}
		\and Na Zhao\footnote{School of Mathematics, Shanghai University of Finance and Economics, Shanghai, 200433, P. R. China. Email: zhaona@shufe.edu.cn}
	}
\date{}

\maketitle

	\date{}
\begin{abstract}
We study the two dimensional incompressible viscous non-resistive magnetohydrodynamics on the periodic strip $\mathbb T\times\mathbb R$,
subject to a smooth sheared  magnetic field $(\xi(x_2),0)^{\top}$, where
$\xi(x_2)$ is bounded away from zero.
For sufficiently small smooth perturbations,
we prove global well-posedness and nonlinear asymptotic stability in Lagrangian coordinates.
The spatial inhomogeneity of the shear profile generates persistent linear contributions, most critically a nontrivial pressure term that precludes the desired uniform in time estimates. We straighten the integral curves of the initial magnetic field and construct a volume-preserving corrector. This geometric reduction transforms the intractable linear pressure into a quadratic nonlinearity.
A weighted horizontal mean equation and a corresponding primitive estimate are derived to control the
mean horizontal velocity and the accumulated horizontal mean deformation.
These structures yield the global in time energy bounds and the anisotropic algebraic decay rates for the system.
The method can also be  applied to the two dimensional torus case.
This mechanism appears to provide the first rigorous framework for establishing global asymptotic nonlinear stability for the viscous non-resistive magnetohydrodynamics near the genuinely nonuniform sheared magnetic profile.
\end{abstract}
	


	
\tableofcontents

	\section{Introduction}
The equations of magnetohydrodynamics (MHD), first introduced by the 1970 Nobel laureate Hannes Alfv\'en \cite{Al,Al2}, constitute the fundamental system governing the behavior of electrically conducting fluids in the presence of magnetic fields.
	The velocity field obeys the Navier-Stokes equations with the Lorentz force. 
	The magnetic field satisfies  the 
	Maxwell-Faraday equations, which describe Faraday's law of induction.
	We refer to \cite{Bis,Davi,PF} for further detailed explanations of this system.
	In this article, we consider 
	the following two-dimensional incompressible viscous non-resistive magnetohydrodynamic system
	\begin{equation}\label{A1}
		\begin{cases}
			\p_t u+u\cdot \na_x u-\Delta_x u+\nabla_x p=b\cdot \nabla_x b,\\
			\p_t b+u\cdot\nabla_x b=b\cdot\nabla_x u, \\
			\hbox{div}_x\, u=\hbox{div}_x\, b=0, \\
			(u,b)|_{t=0}=(u_0,b_0),
		\end{cases}
         \, x\in \,  \mathbb \bT\times\bR,
	\end{equation}
	where $u=(u^1,u^2)^\top,\, b=(b^1,b^2)^\top$  represent the velocity field and magnetic field, $p$ is the scalar pressure, and $\bT=\bR/\bZ$.  The unknowns $u,\, b,\, p$ are assumed to be periodic in $x_1$ with period $1$.
	
	System \eqref{A1} models conducting fluids in a regime where viscosity is retained and magnetic
diffusion is neglected\cite{Caba}. For the nonlinear MHD system, a sufficiently strong magnetic field can reduce nonlinear interactions \cite{Kr} and inhibit the formation of strong gradients. This effect was observed in direct numerical simulations for the ideal MHD system (i.e. inviscid and non-resistive)
	with periodic boundary conditions \cite{FPSM}. When the background magnetic field is $(1,0)^\top$, i.e. in the case that the
	velocity field and magnetic field are sufficiently close to the equilibrium state $\big( (0,0)^\top,(1,0)^\top \big)$, there are many interesting results. 
	Lin, Xu and Zhang \cite {LXZ1} established  the global existence of small solutions around the equilibrium
	under a certain admissible condition.
	For more results in two dimensions, we refer to \cite{RWXZ,ZhangT,DPZ}; for the three-dimensional case, see \cite{AZ,CHZ,LZ14,LZ,PZZ,XZ}.


Sheared magnetic fields play an important role in plasma physics. In magnetically confined fusion devices like tokamaks, a deliberately configured sheared magnetic field is a principal stabilizing mechanism for high-temperature plasmas. It operates by rendering the field lines non-aligned, which suppresses large-scale instabilities by disrupting the coherent propagation of perturbations.
In solar physics, sheared magnetic fields are routinely observed in active regions and are important for the storage and subsequent release of magnetic energy.
For studies of (in)stability under a sheared magnetic field for MHD equations, with or without a sheared velocity field, we refer to  \cite{LMZ, RWZ, RZ, ZZZ} for the two-dimensional linear incompressible ideal MHD.

To the best of our knowledge, rigorous results on nonlinear
stability near nonconstant sheared magnetic fields remain limited.
In this work, we consider the incompressible viscous and non-resistive MHD
equations with a sheared background magnetic field
$(\xi(x_2),0)^\top$ on the periodic strip $\T\times\R$.
The profile function $\xi(x_2)$ is assumed to be bounded away from zero
and has bounded derivatives in Sobolev spaces.
No upper bound on $\xi$ is imposed, and no smallness
assumption is made on $\xi'$.
For sufficiently small initial perturbations in suitable
Sobolev spaces,
we prove global existence,
uniqueness, and nonlinear asymptotic stability in Lagrangian coordinates. 
The method can also be applied to the $\T^2$ case.

\subsection{Reformulation of the problem}
Let $b=(\xi(x_2),0)+H$, where $H=(H^1, H^2)$. Then
the system \eqref{A1} becomes
\begin{equation}\label{MHD-shear-Euler}
\left\{\begin{array}{l}
	\displaystyle \pa_t  u+u\cdot\nabla_x u -\Delta_x u+\na_x p=H\cdot\na_x H+\xi\p_{x_1}H+H^2(\xi',0)^\top ,\\
	\displaystyle \pa_t H + u\cdot \na_x H+ u^2(\xi',0)^\top= H\cdot\nabla_x u+\xi\p_{x_1} u.
\end{array}\right.
\end{equation}
Taking the divergence of the first equation in \eqref{MHD-shear-Euler} gives
\begin{align*}
\Delta_x p \;=
2\xi'\pa_{x_1}H^2-\operatorname{div}_x
\bigl(u\cdot\nabla_x u-H\cdot\nabla_x H\bigr).
\end{align*}
The linear terms $H^2(\xi',0)^\top$ and $u^2(\xi',0)^\top$, together with
the pressure gradient $2\nabla_x\Delta_x^{-1}(\xi'\pa_{x_1}H^2)$ in \eqref{MHD-shear-Euler},
are the main obstacles to obtaining uniform-in-time
energy estimates, since $\xi'$ need not be small.
In such a complicated linear system involving nonlocal operators, one might resort to a delicate spectral analysis to study its (in)stability.
We show that an energy method 
suffices to establish the global in time nonlinear asymptotic stability.

Instead of working on \eqref{MHD-shear-Euler} under Eulerian coordinates,
we switch to the Lagrangian formulation, where the troublesome terms $H^2(\xi',0)^\top$ and $u^2(\xi',0)^\top$ are absent.
Following \cite{AZ}, we define the flow map by
\begin{equation}\label{flow}
	\begin{cases}
		\f{\rmd}{\rmd t} X(t,y)=u(t,X(t,y)),\\
		X(0,y)=y.
	\end{cases}
\end{equation}
By using $\textrm{div}_x\, u=0$, we have
 \[\p_t\det(\nabla_y X(t,y))
 =\det(\nabla_y X(t,y))\,(\operatorname{div}_x u)(t,X(t,y))=0.\]
Hence
\begin{align*}
	\det(\nabla_y X(t,y))=\det(\nabla_y X(0,y))=1.
\end{align*}
We denote $A=(\nabla_{y}X)^{-\top}$.
As derived in \cite{AZ, CHZ, LXZ1}, we have
\begin{align}\label{eqbd}
	\frac{\p}{\p t}(A_{ij}b^i(t,X(t,y)))=0,\, j=1,2,
\end{align}
which yields
\begin{equation}\label{eqbt}
	b(t,X(t,y))=b_0(y)\cdot\na_y X(t,y):=\p_{b_0} X(t,y),
	\,\,\textrm{for}\,\, t\geq 0,
\end{equation}
with $\p_{b_0}=b_0(y)\cdot\nabla_y$. Then
the MHD system \eqref{A1} can be written as follows:
\begin{equation}\label{A3}
	\begin{cases}
	 X^i_{tt} - \hbox{div}_y(A^\top A\nabla_y X_t^i ) -  \p_{b_0}^2X^i + (A\nabla_y p)^i=0,\, i=1,2,\\
		\det(\nabla_y X)=1,\\
		X(0,y)=y,\, X_t(0,y)=u_0(y).
	\end{cases}
\end{equation}
Write the particle map as $X(t,y)=y+Y(t,y)$.
Then \eqref{A3} reduces to
\begin{equation}\label{A14}
	\begin{cases}
		Y^i_{tt}  - \hbox{div}_y\big( A^\top A\nabla_y Y^i_t \big)  - \p_{b_0}^2 Y^i-\p_{b_0}b_0^i =-(A\nabla_y p)^i,\, i=1,2,\\
		\det(I+\nabla_y Y)=1,\\
		Y(0,y)=\textbf{0},\quad Y_t(0,y)=u_0(y),
	\end{cases}
\end{equation}
where
\begin{align*}
	\begin{split}
		A=
        \begin{pmatrix}
           1+\p_{y_2}Y^2&-\p_{y_1}Y^2\\ -\p_{y_2}Y^1
			&	1+\p_{y_1}Y^1
        \end{pmatrix}.
	\end{split}
\end{align*}
Under Eulerian coordinates in $\bT\times \bR$, if $u(t,x)$ and $b(t,x)$ are periodic in $x_1$, then under Lagrangian coordinates, $Y(t,y)$ and the pressure $p(t,X(t,y))$ are also periodic in $y_1$. We refer to Appendix \ref{app:B} for details.
\begin{rem}
	An alternative approach to introducing Lagrangian coordinates,
	widely adopted in the literature, see e.g. \cite{LXZ1,XZ,CHZ}, is to replace $y$ in $\eqref{flow}_2$ by a suitably chosen initial map $X_0(y)$.
	By constructing appropriate $X_0(y)$, it is expected that
	\begin{align}\label{b-X0}
		b(t,X(t,y))=\xi(X_0^2(y)) \p_{y_1} X(t,y).
	\end{align}
The advantage of this formulation is that the magnetic
directional derivative becomes a $\partial_{y_1}$ derivative multiplied
by a fixed positive coefficient.
For $\xi=1$, the existence of $X_0(y)$ was established in \cite{LXZ1,XZ}.
For a non-constant background sheared magnetic field,
however, we show in Appendix~\ref{app:A} that a necessary condition for
\eqref{b-X0}, together with volume preservation, is
$$b_0\cdot\na_x\xi=0.$$
Since $\xi=\xi(x_2)$, this condition forces $b_0^2=0$
wherever $\xi'\ne0$, and therefore restricts the admissible magnetic perturbations.
\end{rem}
A direct energy estimate on the unknowns appears infeasible because of the slow temporal
decay. We therefore resort to a frequency-space analysis aligned with the $b_0$ direction to decompose
the unknowns into the high and low frequency parts.
Since $b_0^2$ may not vanish, this curved direction along $b_0$ is inconvenient to work with.
We straighten it through a
change of variables. To this end, we introduce the map $y=y(z)$,
$z\in\bT\times\bR$, as follows:
\begin{equation}\label{mapyz}
	\begin{cases}
		y_1=z_1,\\
		y_2=z_2+\int_{0}^{z_1} \big( \f{b_0^2}{b_0^1}\big) ( z_1',y_2(z_1',z_2)  )\rmd z_1',
	\end{cases}z\in \bT\times\bR.
\end{equation}	
The periodicity of the map $y(z)$, under the divergence free condition  and
\eqref{b01-cond}, is investigated in Appendix \ref{app:B}.
	
Note that the above equations \eqref{mapyz} in integral form are equivalent to
\begin{equation}\label{ode}
	\begin{cases}
		y_1=z_1,\\
		\f{\rmd y_2(z)}{\rmd z_1 } =\big( \f{b_0^2}{b_0^1}\big) ( z_1,y_2(z_1,z_2) ),\quad y_2(z)\big|_{z_1=0}=z_2,
	\end{cases}z\in \bT\times\bR.
\end{equation}
Applying  $\p_{z_2}$ to $\eqref{ode}_2$, and solving the resulting linear equation gives
\begin{equation}\label{pz2y2}
	\p_{z_2}  y_2= \exp(-h(z)),\qquad
	h(z)=-\int_{0}^{z_1} \p_{y_2}\big( \f{b_0^2}{b_0^1}\big) ( z_1',y_2(z_1',z_2) ) \rmd z_1'.
\end{equation}
Then the Jacobian matrix is given by
\begin{align}\label{nazy}
	\begin{split}
		\na_zy=
		\begin{pmatrix}
			\p_{z_1} y_1 & \p_{z_2}  y_1 \\
			\p_{z_1} y_2 & \p_{z_2}  y_2
		\end{pmatrix}
		&=\begin{pmatrix}
			1 & 0 \\
			\f{b_0^2}{b_0^1}(y(z))  & \exp(-h(z))
		\end{pmatrix}.
	\end{split}
\end{align}	
We will prove in Appendix \ref{app:B} that
	$\f12\leq \det (\na_zy)=\exp(-h(z))\leq 2$,
and the mapping defined in \eqref{mapyz} is globally invertible.

For any smooth function $f$, we have
\begin{align*}
(\p_{b_0}f)(y(z))=(b_0 \cdot \nabla_y f)(y(z))
= b_0^1(y(z)) \p_{z_1} f(y(z)):=\p_{b_0^1}f(y(z)).
\end{align*}
Therefore, we have straightened the directional derivative in the direction of $b_0$ onto that of $e_1$.
Let $B=\big( \nabla_z y \big)^{-\top}$. It follows from \eqref{nazy} that
\begin{align}\label{def-B}
\begin{split}
B=\begin{pmatrix}
		1 & -\f{b_0^2}{b_0^1}(y(z))\exp(h(z)) \\
		0	& \exp(h(z))
	\end{pmatrix}.
\end{split}
\end{align}
By Piola's identity, the rows of $\mathrm{e}^{-h}B$ are divergence free:
\begin{align*}
\nabla_z\cdot(\rme^{-h(z)}B)=0.
\end{align*}
The particle map is $X(t,y)=y+Y(t,y)$. Pullbacks by the fixed map $y(z)$
are denoted by the same symbols. In particular,
$$
 X(t,y(z))=y(z)+Y(t,z),\qquad
 A=(I+\nabla_y Y)^{-\mathsf T},\qquad
 AB=(\nabla_z(X(t,y(z))))^{-\mathsf T}.
$$
For a vector-valued map $F$, we use the Jacobian convention
$\nabla_z F=(\partial_jF^i)_{1\le i,j\le2}$.

In the $z$-coordinates, the system \eqref{A14} becomes
\begin{equation}\label{equ-Y}
\begin{cases}
Y_{tt}-\rme^h\diver_z(\rme^{-h}G\nabla_z Y_t)-\Db^2 Y-\Db b_0 =-AB\nabla_z p,\\
 \det(I+\nabla_y Y)=1,\\
 Y(0,z)=0,\qquad Y_t(0,z)=u_0(y(z)),
\end{cases}
\end{equation}
where the symmetric coefficient matrix of the viscous operator in the fixed
$z$ coordinates is
\begin{equation*}
 G=B^\top A^\top AB
   =[\nabla_z(y+Y)]^{-1}[\nabla_z(y+Y)]^{-\mathsf T}.
\end{equation*}
The periodicity of the unknowns in $z_1$ is established in Appendix \ref{app:B}.

The change of variables $y=y(z)$ straightens the integral curves of $b_0$,
so that $z_1$ is aligned with the $b_0$-direction and $z_2$ parametrizes the transverse direction.
In this coordinate system, the domain is foliated by one-dimensional fibers  $\{z_2=\mathrm{const.}\}$ corresponding to streamlines of $b_0$.

Since the term $- \p_{b_0^1}b_0(y(z))$ independent of time, we introduce a corrector to absorb this term.
The goal is threefold: \\
 \quad {\bf{(i)}}  A normalization to guarantee the horizontal periodicity along each streamline of $b_0$ direction. \\
 \quad {\bf{(ii)}} To absorb the no-time decaying term $- \p_{b_0^1}b_0(y(z))$.\\
 \quad {\bf{(iii)}} A volume-preserving structure
 and thereby handle the linear pressure term.

Define the fixed coefficients
\begin{equation}\label{short-coeff}
 \rho=\av{\rme^{-h}},\qquad
 \gamma=\big\langle\frac1{b_0^1(y(z))}\big\rangle^{-1},\qquad
 \Phi=b_0(y(z))-(\gamma,0)^\top,
\end{equation}
where $\av{F}$ denotes the average over one period in $z_1$.
$$
 \av{F}(z_2)=\int_{0}^{1}F(z_1,z_2)\dd z_1,
 \qquad F_L=\av F,\qquad F_H=F-F_L.
$$
	
Since $\gamma$ is independent of $z_1$, it is easy to check that
\begin{align*}
\p_{b_0^1} b_0(y(z))=\p_{b_0^1} \Phi(z).
\end{align*}
Define the corrector $\widetilde Y$ and the corrected
displacement $\bar Y$ as
\begin{align}\label{DefBaY}
\begin{split}
&\widetilde{Y}^1(z):=-\int_0^{z_1} \f{\Phi^1(\bar{z}_1,z_2)}{b_0^1(y(\bar{z}_1,z_2))} \rmd \bar{z}_1,\\&\widetilde{Y}^2(z):=-\int_0^{z_1} \f{b_0^2}{b_0^1} (y(\bar{z}_1,z_2))\rmd \bar{z}_1
+\int_{\mathbb{T}}
\Big(
\int_0^{z_1}
\frac{b_0^2}{b_0^1}(y(\bar{z}_1,z_2))
\,\mathrm{d}\bar{z}_1
\Big)\mathrm{d}{z_1},\\&
\bar{Y}(t,z):=Y(t,z)-\widetilde{Y}(z).
\end{split}
\end{align}	
Since
$$\int_0^{z_1}\frac{b_0^2}{b_0^1}(y(\bar{z}_1,z_2))\,\mathrm d\bar{z}_1
=y_2(z)-z_2,$$
the second component of $\widetilde{Y}$ satisfies
\begin{align}\label{tildeY2}
 \widetilde Y^2(z)
 =\langle y_2\rangle(z_2)-y_2(z),
 \qquad
 \langle\widetilde Y^2\rangle=0.
\end{align}

Under the conditions of Theorem~\ref{th1}, the choice of
$\gamma$ and the mean correction in $\widetilde Y^2$ ensure the
volume-preserving property of the map
$y\mapsto y+\widetilde Y(z(y))$
(see Lemma~\ref{geo-identities}). Precisely, in the $y$-coordinates,
\begin{align*}
\det(I+\nabla_y\widetilde{Y})=1\,.
\end{align*}
Here $\gamma(z_2)$ normalizes the horizontal period along each streamline, since
\begin{align*}
\int_\bT \frac{\gamma(z_2)}{b_0^1(y(z))}\rmd z_1=1,\quad
\int_\bT \frac{\Phi^1(z)}{b_0^1(y(z))}\rmd z_1=0.
\end{align*}
The integral terms in $\widetilde Y$ provide a correction along each
streamline. However, this longitudinal correction does not automatically
guarantee the required transverse Jacobian.
The second term in $\widetilde Y^2$, namely
$\langle y_2\rangle-z_2$, supplies a correction depending only on
$z_2$. It gives
\[
\partial_1(y_2+\widetilde Y^2)=0,\qquad
\partial_2(y_2+\widetilde Y^2)=\rho.
\]
Since $\partial_1(y_1+\widetilde Y^1)=e^{-h}/\rho$, the corrected map
satisfies $\det\nabla_z(y+\widetilde Y)=e^{-h}=\det\nabla_zy$.
Thus, the two normalizations preserve volume in the $y$-coordinates.

The volume-preserving properties of the particle map and the
corrector map provide the geometric structure needed for the
energy estimates.
Writing $Y=\bY+\tY$ and comparing the two determinant identities
eliminates the terms involving only the fixed corrector.
The resulting constraint expresses the divergence of the
corrected displacement in terms of nonlinear products and
terms multiplied by small fixed coefficients.
This reduction transforms the linear pressure contribution into a nonlinear perturbation. 
The same constraint bounds the vertical mean deformation
$\pa_2\bY_L^2$ in terms of the high-frequency component.
The horizontal mean deformation $\pa_2\bY_L^1$ is controlled
separately through the evolution equation for the weighted
horizontal mean velocity.
These mechanisms provide the desired bounds
to close the global energy estimates.

\begin{rem}
It is noteworthy that the expression for
$\gamma$ is analogous
to the effective coefficients or effective parameters in one-dimensional homogenization theory   (cf. Section 1.3 in \cite{BLP}).
\end{rem}

In terms of the corrected displacement $\bar Y$, system \eqref{equ-Y} becomes
\begin{equation}\label{equ-bY}
\begin{cases}
 \bY_{tt}-\rme^h\diver(\rme^{-h}G\nabla \bY_t)-\Db^2\bY_H
 =-AB\nabla p,\\
\det(I+\nabla_y Y)=1,\\
\bar Y(0,z)=-\widetilde{Y}(z),\quad \bar Y_t(0,z)=u_0(y(z)),
\end{cases}
\end{equation}
The periodicity of the unknowns, the corrector and the map will be investigated in Appendix \ref{app:B}.

A direct energy estimate for $\bar{Y}$ does not provide the same control for the horizontal mean and
the oscillatory part.
Based on the structure of the system, we decompose $\bar{Y}$
into its high-frequency and low-frequency parts along the first spatial variable.
The corresponding high frequency part is
$$\bar{Y}_H=\bar{Y}- \bar{Y}_L,\,\, \bar{Y}_L=\av{\bY}\,.$$
Then we can derive the equations for $\bar{Y}_H$
\begin{equation}\label{high-projected-equation}
 \p_t^2\bY_H-\Delta\p_t\bY_H-(\Db^2\bY_H)_H+\nabla q
 =\bigl((\Delta_X-\Delta)\bY_t\bigr)_H
                         -(AB\nabla p-\nabla q)_H.
\end{equation}
Here $\Delta_X=\mathrm e^h\diver_z(\mathrm e^{-h}G\nabla_z\,\cdot\,)$
and $q$ is the linear pressure correction defined in \eqref{def-q}.

The evolution equations for the weighted mean velocity and
$\partial_t\bar Y_L^1$ will be derived in Section~\ref{sec:7}, where we
also establish the estimate for $\partial_2\bar Y_L^1$.

\subsection{Main results and key ideas}
For any integer $a\ge3$, we define the high-frequency energy
$\mathcal H_a(t)$ for $\bar Y_H$, the low-frequency energy
$\mathcal L_a(t)$ for $\partial_t\bar Y_L$,
and the horizontal mean deformation functional
$\mathcal S_a(t)$ as follows:
\begin{align}
\mathcal H_a(t):=&\sup_{0\leq \tau\leq t}\big(\|\p_t\bar Y_H\|_{H^a}^2(\tau)+\|\p_{b_0^1} \bar Y_H\|_{H^a}^2(\tau)
+\|\na \bar Y_H\|_{H^a}^2(\tau)\big)\notag\\&+\int_0^t\big(\|\na \p_t\bar Y_H\|_{H^a}^2(\tau)+\| \p_{b_0^1} \bar Y_H\|^2_{H^a}(\tau)\big)\rmd\tau,\label{def-Ha}\\
\mathcal L_a(t):=&\sup_{0\leq \tau\leq t}\|\p_t\bar Y_L\|^2_{H^a}(\tau)+\int_0^t\|\p_2 \p_t\bar Y_L\|_{H^a}^2(\tau)\rmd\tau,\label{def-La}\\
\SSS(t):=&\sup_{0\le\tau\le t}
                  \|\pa_2\bY_L^1(\tau)\|_{H^a(\R)}^2.\label{Sa}
\end{align}
We also define the time-weighted energy $\mathcal W_0(t)$ by
\begin{align}\label{def-W0}
\begin{split}
	&\mathcal{W}_0(t):=\sup_{0\leq \tau\leq t}\big((4a+\tau)^a\big(\|\p_t\bar Y_H\|_{L^2}^2(\tau)+ \|\p_{b_0^1} \bar Y_H\|_{L^2}^2(\tau)+ \|\na \bar Y_H\|_{L^2}^2(\tau) \big)\big)\\
&\,\, +\int_0^t\Big[(4a+\tau)^a \big(\|\na \p_t\bar Y_H\|_{L^2}^2(\tau)+ \| \p_{b_0^1} \bar Y_H\|^2_{L^2}(\tau)\big)
+a(a-1)(4a+\tau)^{a-2}\|\bar Y_H\|_{L^2}^2(\tau)\Big] \rmd\tau.	
\end{split}
\end{align}
The definition of $\mathcal H_b$, $\mathcal L_b(t)$ and $\mathcal S_b(t)$ for $0\le b\le a$ is obtained from
\eqref{def-Ha}, \eqref{def-La}, \eqref{Sa} respectively by replacing the Sobolev index $a$ with $b$.
The functional $\LL$ controls low frequency of the velocity field. The additional functional $\SSS$ measures the low frequency of $\pa_2\bY^1$.

For the fixed integer $a\ge3$, we impose the following conditions on $\xi$:
\begin{align}\label{xi-cond0}
 \xi\geq m>0,\quad \|\xi'\|_{H^{a+1}}\leq L,
\end{align}
where $m$ and $L$ are positive constants.
For the case that $\xi$ is negative and the condition $\eqref{xi-cond0}_1$ is replaced by
$ 0<m\leq -\xi,$
the estimate is similar.

Now we state the main result as follows.
\begin{thm}\label{th1}
Let $a\ge3$ be an integer.
Assume that $\xi$ satisfies \eqref{xi-cond0}.
Let
$u_0\in H^a(\T\times\R)$ and
$b_0-(\xi,0)^\top\in H^{a+1}(\T\times\R)$ satisfy
$\operatorname{div}u_0=\operatorname{div}b_0=0$. There exists a constant $\epsilon_*>0$, depending only
on $a,m,L$, such that, for every
$\epsilon_0\in(0,\epsilon_*]$, if
\begin{equation}\label{asmp}
 \|u_0\|_{H^a}\le\epsilon_0,
 \qquad
 \|b_0-(\xi,0)^\top\|_{H^{a+1}}\le\epsilon_0,
\end{equation}
and
\begin{equation}\label{eulerian-low-assumption}
 \big\|\pa_{x_2}^{-1}
       \int_\T u_0^1(x_1,x_2)\dd x_1\big\|_{L^2(\R)}
 \le\epsilon_0,
\end{equation}
then \eqref{equ-Y} admits a unique global strong
solution $Y$. For every $T>0$, this solution satisfies
\[
 \p_tY,\ \nabla Y,\ \Db Y\in C([0,T];H^a),
 \qquad
 \nabla\p_tY\in L^2(0,T;H^a).
\]
Moreover, for all $t\ge0$,
\begin{equation}\label{main-bound}
 \HH(t)+\LL(t)+\WW(t)+\SSS(t)
 \le C\epsilon_0^2,
\end{equation}
and
\begin{equation}\label{asympt}
 \begin{aligned}
 \|u(t,X(t,y(z)))-u_L(t,z_2)\|_{L^2}
 &\le C\epsilon_0(4a+t)^{-a/2},\\
 \|b(t,X(t,y(z)))-(\gamma(z_2),0)^\top\|_{L^2}
 &\le C\epsilon_0(4a+t)^{-a/2},
 \end{aligned}
\end{equation}
where
$u_L(t,z_2)=\big\langle u(t,X(t,y(z)))\big\rangle$.
The constant $C$ depends only
on $a,\,m,\,L$.
\end{thm}

Combining \eqref{xi-cond0} with \eqref{asmp}, and noting that $\epsilon_0$ is sufficiently small, we obtain
\begin{align}\label{b01-cond}
	&0<\frac{m}{2}\leq b_0^1 ,\quad	\|\p_{y_1}b_0^1\|_{H^{a}}\leq \epsilon_0, \quad \|\p_{y_2}b_0^1\|_{H^{a}}\leq L+\epsilon_0,\quad \|b_0^2\|_{H^{a+1}}\leq \epsilon_0.
\end{align}
These bounds will be used repeatedly throughout the paper.
\begin{rem}
The initial magnetic perturbation is assumed to have one
more Sobolev derivative than the initial velocity. This additional
regularity is used in the fixed coordinate change \eqref{mapyz}, since
$h(z)$ contains a derivative of $b_0$; see \eqref{pz2y2}.
\end{rem}
\begin{rem}
The low-frequency condition \eqref{eulerian-low-assumption}
is used to control $\pa_2\bY_L^1$ uniformly in time.
To see why this condition is needed for uniform control,
consider
\[
 b_0=\xi(x_2)e_1,
 \qquad
 u_0=f(x_2)e_1.
\]
For these data, $y=z$ and $\tY=0$, and the corresponding
 solution admits an explicit expression
\[
 u(t,x)=\bigl(\rme^{t\pa_{x_2}^2}f\bigr)(x_2)e_1,
 \qquad
 b(t,x)=\xi(x_2)e_1,
 \qquad
 \bY_L^1(t)=\int_0^t\rme^{s\pa_{x_2}^2}f\,\rmd s.
\]
Plancherel's identity and monotone convergence give
\[
 \sup_{t\ge0}\|\pa_2\bY_L^1(t)\|_{L^2}^2
 =\|\pa_2^{-1}f\|_{L^2}^2,
\]
where the right-hand side is understood as $+\infty$
if $f$ has no $L^2(\R)$ primitive.
For example, take
\[
 f(s)=\varepsilon\frac{\rmd}{\rmd s}(1+s^2)^{-1/4},
 \qquad \varepsilon>0.
\]
Then
\[
 \int_\R f(s)\,\rmd s=0,
 \qquad
 \|f\|_{L^1}+\|f\|_{H^N}\le C_N\varepsilon
 \quad\text{for every integer }N\ge0.
\]
However, every primitive of $f$ has the form
$\varepsilon(1+s^2)^{-1/4}+c$, which does not belong
to $L^2(\R)$.
Thus $\|\pa_2\bY_L^1(t)\|_{L^2}\to\infty$ as
$t\to\infty$.
\end{rem}
\begin{rem}
The decay exponent for $\|\pa_2\bY_H(t)\|_{L^2}$ in
$\mathcal W_0$ cannot be improved uniformly for the
constant-coefficient linear equation with initial data
in the stated energy class. Consider
\[
 \p_t^2\bY_H-\Delta\p_t\bY_H-\pa_1^2\bY_H=0.
\]
The characteristic roots of the Fourier modes are
\[
 r_\pm(\lambda)
 =-\frac{(2\pi)^2|\lambda|^2}{2}
 \pm\frac12
 \sqrt{(2\pi)^4|\lambda|^4-4(2\pi)^2\lambda_1^2},
 \qquad \lambda_1\ne0.
\]
For $N\ge2$, set
\[
 \Omega_N
 :=\big\{
 (\lambda_1,\lambda_2)\in\mathbb Z\times\R:
 |\lambda_1|=1,\ N\le|\lambda_2|\le2N
 \big\}.
\]
Rationalizing the expression for $r_+$ gives
$-r_+(\lambda)\simeq N^{-2}$ on $\Omega_N$.
Choose a real divergence-free vector field $F_N$
with $\|F_N\|_{L^2}=1$ and smooth Fourier transform
supported in $\Omega_N$. Prescribe the linear
initial data by
\[
 \bY_H(0)=\epsilon_0N^{-a-1}F_N,
 \qquad
 \widehat{\p_t\bY_H}(0,\lambda)
 =r_+(\lambda)\widehat{\bY_H}(0,\lambda).
\]
They satisfy
\[
 \|\nabla\bY_H(0)\|_{H^a}
 +\|\p_t\bY_H(0)\|_{H^a}
 \le C\epsilon_0
\]
uniformly in $N$. The corresponding solution is
given by
\[
 \widehat{\bY_H}(t,\lambda)
 =\rme^{t r_+(\lambda)}
  \widehat{\bY_H}(0,\lambda).
\]
At $t=N^2$, the exponential is bounded below on
$\Omega_N$ by a positive constant independent of $N$.
Plancherel's identity therefore gives
\[
 \|\pa_2\bY_H(N^2)\|_{L^2}
 \ge cN\|\bY_H(0)\|_{L^2}
 =c\epsilon_0N^{-a}=c\epsilon_0t^{-a/2}.
\]
Thus a uniform bound for $\|\pa_2\bY_H(t)\|_{L^2}$
in this linear energy class cannot have a decay
exponent greater than $a/2$.

For the same solutions, however,
\[
 \|\pa_1\bY_H(N^2)\|_{L^2}
 \simeq\epsilon_0N^{-a-1},
 \qquad
 \|\p_t\bY_H(N^2)\|_{L^2}
 \simeq\epsilon_0N^{-a-3}.
\]
These scalings correspond to the exponents
$(a+1)/2$ and $(a+3)/2$, respectively, but do not
by themselves establish the corresponding nonlinear
decay estimates. In particular, the common exponent
$a/2$ supplied by $\mathcal W_0$ need not be optimal
for $\pa_1\bY_H$ and $\p_t\bY_H$; sharper estimates
for these quantities require separate arguments.
\end{rem}
\begin{rem}
Define the following time weighted energy:
\begin{align*}
\begin{split}
	\mathcal{W}_k(t):&=\sup_{0\leq \tau\leq t}\big((4a+\tau)^{a-k}\big(\|\p_t\bar Y_H\|_{H^k}^2(\tau)+ \|\p_{b_0^1} \bar Y_H\|_{H^k}^2(\tau)+ \|\na \bar Y_H\|_{H^k}^2(\tau) \big)\big)\\
&\,\, +\int_0^t(4a+\tau)^{a-k} \big(\|\na \p_t\bar Y_H\|_{H^k}^2(\tau)+ \| \p_{b_0^1} \bar Y_H\|^2_{H^k}(\tau)\big)
\rmd\tau\\
&\,\, +\int_0^t(a-k)(a-k-1)(4a+\tau)^{a-k-2}\|\bar Y_H\|_{H^k}^2 \rmd\tau,
\end{split}
\end{align*}
for $0\leq k\leq a$. Using the Poincar\'e inequality and the interpolation between $\mathcal H_a(t)$ and $\mathcal{W}_0(t)$, we immediately have the following temporal weighted energy bound:
$$	 \mathcal{W}_k(t)\leq C\epsilon_0^2,\,\,\textrm{for}\,\, 0\leq k\leq a.$$
Correspondingly,
\begin{align*}
\begin{split}
\| u(t,X(t,y(z))) - u_L(t,z_2)\|_{H^k}\lesssim \epsilon_0 (4a+t)^{-\f{a-k}{2}}, \ \text{for}\ 0\leq k\leq a,\\
\| b(t,X(t,y(z))) - \big( \gamma(z_2),0\big)^\top\|_{H^k}\lesssim \epsilon_0 (4a+t)^{-\f{a-k}{2}}, \ \text{for}\ 0\leq k\leq a\,.
\end{split}
\end{align*}
\end{rem}


We now state our main result in Eulerian coordinates as follows.
\begin{thm}\label{MTH}
Let $a\ge3$ be an integer, and assume that $\xi$
satisfies \eqref{xi-cond0}.
Let $u_0\in H^a(\mathbb T\times\mathbb R)$ and
$b_0-(\xi,0)^\top\in H^{a+1}(\mathbb T\times\mathbb R)$
satisfy
$
 \operatorname{div}_xu_0=\operatorname{div}_xb_0=0.
$
There exists a constant $\epsilon_*>0$, depending only
on $a,\,m,\,L$, such that, for every
$\epsilon_0\in(0,\epsilon_*]$, if
\eqref{asmp} and \eqref{eulerian-low-assumption} hold, then
\eqref{A1} admits a unique global strong solution $(u,b)$.
For every $T>0$, this solution satisfies
\[
 u,\ b-(\xi,0)^\top\in C([0,T];H^a),
 \qquad
 \nabla_x u\in L^2(0,T;H^a).
\]
\end{thm}
We also obtain the following result on $\T^2$ under
the additional condition $\int_{\T^2}b_0^2(x)\,\rmd x=0$.
The low-frequency condition
\eqref{eulerian-low-assumption} is not required
in this case.
In the following theorem, spatial norms are taken on
$\T^2$, or on $\T$ for functions depending only on
the second spatial variable.
The energy functionals are understood with these
periodic norms.

\begin{thm}\label{thm:torus-extension}
Let $a\ge3$ be an integer, and let $\xi$ be a
$1$-periodic function satisfying \eqref{xi-cond0}.
Let $u_0\in H^a(\T^2)$ and
$b_0-(\xi,0)^\top\in H^{a+1}(\T^2)$ satisfy
\[
 \operatorname{div}_xu_0=\operatorname{div}_xb_0=0,
 \qquad
 \int_{\T^2}b_0^2(x)\,\rmd x=0.
\]
There exists $\epsilon_*>0$, depending only on
$a,\,m,\,L$, such that, for every
$\epsilon_0\in(0,\epsilon_*]$, if \eqref{asmp} holds,
then the periodic system \eqref{equ-Y} admits a
unique global strong solution $Y$.
For every $T>0$,
\[
 \p_tY,\ \nabla Y,\ \Db Y\in C([0,T];H^a),
 \qquad
 \nabla\p_tY\in L^2(0,T;H^a).
\]
Moreover, \eqref{main-bound} holds for all $t\ge0$.

The corresponding Eulerian fields $(u,b)$ form the
unique global strong solution of \eqref{A1} on $\T^2$
and satisfy, for every $T>0$,
\[
 u,\ b-(\xi,0)^\top\in C([0,T];H^a),
 \qquad
 \nabla_xu\in L^2(0,T;H^a).
\]
The asymptotic estimates \eqref{asympt} also hold,
where
$u_L(t,z_2)=\langle u(t,X(t,y(z)))\rangle$.
The constants in these estimates depend only on
$a,\,m,\,L$.
\end{thm}

We now sketch the main steps and explain the key ideas underlying the energy estimates.
The equation for $\bar{Y}_H$ contains the following linear pressure term:
\begin{align}\label{def-q}
q:=2\Delta^{-1}\p_1(\gamma'(z_2)\p_{b_0^1} \bar{Y}_H^2).
\end{align}
Estimates \eqref{b01-xi-small} and \eqref{gm-xi-small} give
$\|\gamma'-\xi'\|_{H^a}\le C\epsilon_0$.
Since $\xi'$ may be large,
we retain the linear pressure
term $\na q$ on the left-hand side of the high-frequency equation
\eqref{high-projected-equation}.
The right-hand side consists of the variable-coefficient viscous
contribution and the pressure remainder $AB\na p-\na q$,
whose contributions to the energy estimates are treated
perturbatively.

A central difficulty in the high-frequency energy estimates is the contribution
of the linear pressure,
\[
 \Big(
  \na q,\,
  2\pa_t\bY_H+\frac12\bY_H
 \Big)_{L^2}.
\]
Although the coefficient $\gamma'$ need not be small, this
contribution can be controlled by exploiting the incompressibility
constraint and the geometry of the corrector.
Indeed, both the particle map and the corrector map preserve
volume in the $y$-coordinates:
\[
 \det(I+\na_yY)=\det(I+\na_y\tY)=1.
\]
Writing $Y=\bY+\tY$ and comparing these two determinant identities
gives
\begin{equation*}
 \rho\pa_1\bY^1+\pa_2\bY^2
 = \pa_1\bY^2\pa_2Y^1 - \pa_1Y^1\pa_2\bY^2;
\end{equation*}
see Lemma~\ref{constraint-lemma}.
Consequently,
\[
 \na\cdot\bY_H
 =
 -\pa_2\bY_L^2+\mathcal R_H^1,
 \qquad
 \na\cdot\pa_t\bY_H
 =
 -\pa_t\pa_2\bY_L^2+\mathcal R_H^2,
\]
where
\[
 \begin{aligned}
  \mathcal R_H^1
  &:=
  (1-\rho)\pa_1\bY_H^1
  +\pa_1\bY_H^2\pa_2Y^1
  -\pa_1Y^1\pa_2\bY^2,\qquad \mathcal R_H^2:=
  \pa_t\mathcal R_H^1.
 \end{aligned}
\]
These remainders consist of nonlinear products and terms
multiplied by small fixed coefficients.

By its definition, $q$ has zero horizontal mean.
Since $\pa_2\bY_L^2$ and $\pa_t\pa_2\bY_L^2$ depend only on
$z_2$, they are orthogonal to $q$.
Integration by parts therefore yields the exact identity
\[
 \begin{aligned}
 &\Big(
   \na q,\,
   2\pa_t\bY_H+\frac12\bY_H
  \Big)_{L^2}=
 -\Big(
   q,\,
   2\na\cdot\pa_t\bY_H
   +\frac12\na\cdot\bY_H
  \Big)_{L^2}=
 -\Big(
   q,\,
   2\mathcal R_H^2+\frac12\mathcal R_H^1
  \Big)_{L^2}.
 \end{aligned}
\]
Thus, although $q$ is linear in the displacement, its contribution
to the energy identity involves only $L^2$ inner products of $q$
with the above remainders.
The resulting terms can be controlled by the dissipation and the
low-frequency velocity estimates.
More precisely, under the bootstrap assumption
\eqref{bootstrap}, the proof of Proposition~\ref{high-prop} gives
\[
 \int_0^T
 \left|
  \Big(
   \na q,\,
   2\pa_t\bY_H+\frac12\bY_H
  \Big)_{L^2}
 \right|\,\rmd t
 \le
 C(\epsilon_0+\sqrt{\delta})\mathcal H_a(T).
\]
This cancellation is reminiscent of a normal-form argument:
the apparently leading-order pressure contribution is reduced,
at the energy level, to higher-order terms and terms carrying
small fixed coefficients.
Here the improvement follows directly from the determinant
constraint and horizontal orthogonality.

In the higher-order estimates for the linear system, at the
$\dot H^b$ level ($1\le b\le a$), linear commutators arise when
vertical derivatives fall on the background coefficients.
Young's inequality separates their contributions into a
small multiple of the dissipation at order $b$ and a bounded multiple of
lower-order dissipation. An induction on $b$ then closes these estimates
without requiring the background derivatives to be small.

In estimating the nonlinear viscous terms, we must treat
the contributions involving the horizontal mean and
the fixed corrector separately, because the corresponding deformation
coefficients need not decay in time. To handle these terms, we write
\begin{align}
&\sum_{|\alpha|\le a}
 \left|\int_0^t
 \big(\pa^\alpha((\Delta_X-\Delta)\bY_t),
       \pa^\alpha(\p_t\bY_H+\tfrac14\bY_H)\big)_{L^2}\dd\tau\right|
 \notag\\
&\le  \sum_{|\alpha|\le a}
 \left|\int_0^t
 \big(\pa^\alpha \Delta_X \p_t\bY_L,
       \pa^\alpha(\p_t\bY_H+\tfrac14\bY_H)\big)_{L^2}\dd\tau\right|+\text{terms with high frequency}
       \notag.
\end{align}
Noting \(\pa_1\p_t\bY_L=0\), we write
\begin{align*}
\Delta_X \p_t\bY_L
=\rme^h\diver(\rme^{-h}G\nabla \p_t\bY_L)
=\rme^h \Big[ \p_1(\rme^{-h}G_{12}\p_2 \p_t\bY_L)+\p_2(\rme^{-h}G_{22}\p_2 \p_t\bY_L) \Big].
\end{align*}
To estimate the contribution of $\Delta_X\p_t\bar Y_L$,
we first separate out the terms that contains no time decaying factors that involve a purely horizontal mean term of $\p_2 \bY^1$ and  $\tY$, $\rho$.
By the explicit formulas for $G$, we write
\begin{align*}
 G_{12}+\frac{\rme^h}{\rho}(\pa_2\tY^1+\pa_2\bY_L^1)
 =& \text{terms containing derivatives of $\bY_H$},\\
 G_{22}-\rho^{-2}
 =&\text{terms containing $\partial_1\bY_H$} .
\end{align*}
For these bad contribution to
$\Delta_X\p_t\bar Y_L$,
they combination is equal to
\begin{equation*}
 \rme^h\left\{\pa_1\Big[-\frac{\pa_2\tY^1+\pa_2\bY_L^1}{\rho}
                                     \pa_2\p_t\bY_L\Big]
       +\pa_2\Big[\frac{\rme^{-h}}{\rho^2}\pa_2\p_t\bY_L\Big]\right\}
 =\rho^{-1}\pa_2(\rho^{-1}\pa_2\p_t\bY_L).
\end{equation*}
This expression is independent of \(z_1\), although the two extracted coefficients need not be independent of $z_1$.  This bad contribution vanishes when taking the $H^a$ inner product with the high-frequency unknowns.
This structure is crucial for closing the energy estimate.

Now we discuss the estimate of the low frequency part $\p_2 \bar{Y}_L$.
 The determinant constraint $\det(I+\nabla_y{Y})=1$ and
$\det(I+\nabla_y\widetilde{Y})=1$ yield an identity involving $\na\cdot \bar{Y}$.
The quadratic relation of $\na\cdot \bar{Y}$
 yields the control over $\p_2 \bar{Y} ^2$ and its low-frequency part $\p_2 \bar{Y}_L^2$.
While $\pa_2\bY_L^1$ requires a separate argument.
Note that
$
 \frac{\langle \rme^{-h}\pa_t\bY^1\rangle}{\rho}
$
satisfies the one-dimensional weighted parabolic equation with diffusion operator
$\rho^{-1}\pa_2(\rho^{-1}\pa_2\,\cdot\,)$
(see Lemma~\ref{mean-equation-lemma}):
\begin{equation*}
 \p_t v-D^2v=DR,
\end{equation*}
where
$$
v(t,z_2):=\frac{\langle\rme^{-h}\p_t\bY^1\rangle}{\rho},\quad D:=\rho^{-1}\pa_2.
$$
Here $R$ is the nonlinear momentum flux defined in
Section~\ref{sec:7}. The ordinary and weighted horizontal means of the displacement differ
by a term containing a small fixed coefficient and $\bar Y_H$.
Thus, once the high-frequency component is controlled, estimating
$\partial_2\bar Y_L^1$ reduces to estimating the spatial derivative
of the weighted mean displacement.
For this one dimensional weighted heat equation on $\bR$, the $L^\infty$ norm
of $\p_t\bY^1_L$ decays at the rate $t^{-\f12}$. However, an estimate for $\p_2\bY^1_L$ itself is required. Recovering  $\bY^1_L$ by integrating $\p_t \bY^1_L$  in time would instead lead to an undesirable time growth of order $t^{\f12}$.
To overcome this difficulty,  we introduce the weighted primitive
$V$ satisfying $DV=v$ and $V_t-D^2V=R$.
The controls of its initial value in suitable Sobolev spaces
need primitive condition \eqref{eulerian-low-assumption}.
 Integrating $Dv=V_t-R$ in time then expresses the change
in the weighted mean deformation as
$V(t)-V(0)-\int_0^tR(\tau)\,\rmd\tau$ yields the uniform control for $\p_2\bY^1_L$.

We also need the time-weighted high-frequency estimate to obtain
$\nabla\bar Y_{H,t}\in L_t^1L^\infty$, with bounds independent of the
time interval.
The estimates for $\mathcal{H}_a(t)$,
$\mathcal{L}_a(t)$, $\mathcal{W}_0(t)$ and $\mathcal{S}_a(t)$
suffice to close the energy.
Combining the uniform in time estimate with the standard local
well-posedness result stated in Proposition~\ref{local-self-contained},
we obtain the global nonlinear stability result.

Throughout the paper, unless otherwise stated, $a\ge3$
is a fixed integer, $\xi$ satisfies \eqref{xi-cond0},
and the initial data $u_0$ and $b_0$ are divergence-free. Unless otherwise indicated, spatial derivatives and norms
are taken in the $z$-coordinates on $\T\times\R$, and
the subscript $z$ is omitted from differential operators.
Derivatives in other coordinates are specified explicitly.
Norms of the initial data before composition are taken
in the original spatial variables.
For functions depending only on $z_2$, we use the
corresponding norms on $\R$.
We write $e_1=(1,0)^\top$ and $e_2=(0,1)^\top$.
Vector components are denoted by superscripts, and
repeated spatial indices are summed from $1$ to $2$.
For a scalar- or vector-valued function $f$, we denote its usual $L^p$ norm by $\|f\|_{L^p}$.
For a nonnegative integer $s$, we define $(f\mid g)_{L^2} :=\int_{\T\times\R}f\cdot g\,\rmd z$, $(f\mid g)_{H^s} :=\sum_{|\alpha|\le s} (\pa^\alpha f\mid\pa^\alpha g)_{L^2}.$
For a Banach space $X$ and $1\le p\le\infty$, we write
$L_t^p(X):=L^p(0,T;X)$, where $(0,T)$ is the time
interval under consideration.
For nonnegative quantities $X$ and $Y$, we write
$X\lesssim Y$ if $X\le CY$ for some constant $C>0$,
and $X\gtrsim Y$ if $Y\lesssim X$.
We write $X\simeq Y$ if both $X\lesssim Y$ and
$Y\lesssim X$ hold.

The rest of the paper is organized as follows.
Section~\ref{sec:2} establishes the geometric identities
and preliminary estimates.
Section~\ref{sec:3} derives the pressure equation and
the corresponding estimates.
Sections~\ref{high-linear-section}
and~\ref{high-nonlinear-section} establish the linear
and nonlinear estimates for the high-frequency component,
respectively.
Section~\ref{weight-sec} proves the time-weighted estimate,
and Section~\ref{sec:7} estimates the low-frequency components.
These estimates are combined in Section~\ref{sec:8}
to prove the global stability theorem on the strip. Section~\ref{sec:torus-extension} gives the proof on $\T^2$.
The appendices contain auxiliary results on coordinate
changes, periodicity, and Sobolev estimates.

\section{Geometric identities and preliminary estimates}\label{sec:2}

In Sections~\ref{sec:2}--\ref{sec:7}, we establish a priori
estimates for a sufficiently regular solution $(Y,p)$
of \eqref{equ-Y} on $[0,T]$, where $T>0$.
Throughout these sections, we assume that \eqref{asmp}
holds with $\epsilon_0$ sufficiently small. Unless otherwise indicated, the constants $C$ in these
estimates depend only on $a,\,m,\,L$ and are independent
of $T$, $\epsilon_0$, and $\delta$.
We retain the notation $h$, $B$, $\rho$, $\gamma$, and $\Phi$
from \eqref{pz2y2}, \eqref{def-B}, and \eqref{short-coeff},
together with the corrector $\widetilde Y$ and the corrected
displacement $\bar Y$ introduced in \eqref{DefBaY}.

We first establish several intrinsic identities, which involve the
derivative of $\tY$, the expression of the matrix $AB$, the matrix $G$,  the volume-preserving property of the map $Id_y+\tY$, and an identity relating $\rme^{h},\,\rho,\,\gamma,\,b_0^1(y(z))$.
\begin{lemma}\label{geo-identities}
The following identities hold:
\begin{equation}\label{rho-identities}
 \partial_1\bigl(\rme^{-h}b_0^1(y(z))\bigr)=0,
 \qquad
 \frac{\rme^{-h}}{\rho}=\frac{\gamma}{b_0^1(y(z))}.
\end{equation}
The corrector $\widetilde Y$ satisfies
\begin{equation}\label{corrector-identities}
 -\partial_{b_0^1}\widetilde Y=\Phi,
 \qquad
 \partial_2\widetilde Y^2=\rho-\rme^{-h},
\end{equation}
and the map $\operatorname{Id}_y+\widetilde Y$ preserves volume:
\begin{equation}\label{geom:corrector-determinant}
 \det(I+\nabla_y\widetilde Y)=1.
\end{equation}
Furthermore, the matrix $AB$ is given by
\begin{equation}\label{geom:AB}
 AB=\rme^h
 \begin{pmatrix}
  \rho+\partial_2\bar Y^2 & -\partial_1\bar Y_H^2\\
  -\partial_2Y^1
  & \rme^{-h}\rho^{-1}+\partial_1\bar Y_H^1
 \end{pmatrix},
\end{equation}
and it satisfies the Piola identity
\begin{equation}\label{geom:AB-Piola}
 \diver\bigl(\rme^{-h}AB\bigr)=0,
 \qquad\text{i.e.,}\qquad
 \partial_k\bigl(\rme^{-h}(AB)_{ik}\bigr)=0,
 \quad i=1,2.
\end{equation}
Together with the incompressibility constraint,
this yields
\begin{equation}\label{geom:velocity-div}
 \diver\bigl(\rme^{-h}(AB)^\top\p_t\bY\bigr)=0,
 \qquad\text{i.e.,}\qquad
 \partial_k\bigl(
 \rme^{-h}(AB)_{ik}\p_t\bY^i
 \bigr)=0.
\end{equation}
The matrix $G=(AB)^\top AB$ is given by
\begin{equation}\label{matrix-G}
\begin{aligned}
 G_{11}
 &=\rme^{2h}\big[
 (\rho+\pa_2\bY^2)^2
 +(\pa_2\tY^1+\pa_2\bY^1)^2
 \big],\\
 G_{12}=G_{21}
 &=-\frac{\rme^h}{\rho}\pa_2\tY^1
   -\frac{\rme^h}{\rho}\pa_2\bY^1-\rme^{2h}\big[
 \pa_1\bY_H^1(\pa_2\tY^1+\pa_2\bY^1)
 +\pa_1\bY_H^2(\rho+\pa_2\bY^2)
 \big],\\
 G_{22}
 &=\rho^{-2}
 +\frac{2\rme^h}{\rho}\pa_1\bY_H^1
 +\rme^{2h}|\pa_1\bY_H|^2.
\end{aligned}
\end{equation}
\end{lemma}
\begin{proof}
By \eqref{mapyz} and \eqref{pz2y2},
\[
 \partial_1y_2=\frac{b_0^2}{b_0^1}(y(z)),
 \qquad
 \partial_2y_2=\rme^{-h},
 \qquad
 \partial_1h
 =-\partial_{y_2}\big(\frac{b_0^2}{b_0^1}\big)(y(z)).
\]
Hence
\begin{align*}
 \partial_1\bigl(\rme^{-h}b_0^1(y(z))\bigr)
 &=\rme^{-h}\Big[
 b_0^1\partial_{y_2}\big(\frac{b_0^2}{b_0^1}\big)
 +\partial_{y_1}b_0^1
 +\frac{b_0^2}{b_0^1}\partial_{y_2}b_0^1
 \Big](y(z))\\
 &=\rme^{-h}(\operatorname{div}_y b_0)(y(z))=0.
\end{align*}
Since $\rme^{-h}b_0^1(y(z))$ is independent of $z_1$, the definitions
of $\rho$ and $\gamma$ yield
\[
 \rho=\langle \rme^{-h}\rangle
 =\rme^{-h}b_0^1(y(z))
   \big\langle\frac{1}{b_0^1(y(z))}\big\rangle
 =\frac{\rme^{-h}b_0^1(y(z))}{\gamma},
\]
which proves \eqref{rho-identities}.

By \eqref{DefBaY},
\[
 \partial_1\widetilde Y^1
 =\frac{\gamma}{b_0^1(y(z))}-1,
 \qquad
 \partial_1\widetilde Y^2
 =-\frac{b_0^2}{b_0^1}(y(z)).
\]
Since $\partial_{b_0^1}=b_0^1(y(z))\partial_1$, these identities
give $\partial_{b_0^1}\widetilde Y=-\Phi$.
Moreover, differentiating \eqref{tildeY2} gives
\[
 \partial_2\widetilde Y^2
 =\langle\partial_2y_2\rangle-\partial_2y_2
 =\rho-\rme^{-h}.
\]
This proves \eqref{corrector-identities}.
Together with \eqref{rho-identities}, we obtain
\[
 \nabla(y+\widetilde Y)
 =
 \begin{pmatrix}
  \rme^{-h}\rho^{-1} & \partial_2\widetilde Y^1\\
  0 & \rho
 \end{pmatrix},
 \qquad
 \det\nabla(y+\widetilde Y)=\rme^{-h}.
\]
Using $\det\nabla y=\rme^{-h}$ and the chain rule, we conclude that
\[
 \det(I+\nabla_y\widetilde Y)
 =\frac{\det\nabla_z(y+\widetilde Y)}
        {\det\nabla_zy}
 =1,
\]
which is \eqref{geom:corrector-determinant}.

Next, $Y=\widetilde Y+\bar Y$ and
$\partial_1\bar Y=\partial_1\bar Y_H$ imply
\begin{align}\label{nazX}
     \nabla (X\circ y)=\nabla (y+Y)
 =
 \begin{pmatrix}
  \rme^{-h}\rho^{-1}+\partial_1\bar Y_H^1
    & \partial_2Y^1\\
  \partial_1\bar Y_H^2
    & \rho+\partial_2\bar Y^2
 \end{pmatrix}.
\end{align}
By incompressibility,
\begin{align}\label{detnazX}
  \det\nabla (X\circ y)
 =\bigl(\det\nabla_yX\bigr)\circ y\,\det\nabla y
 =\rme^{-h}.
\end{align}
The identity $AB=[\nabla (X\circ y)]^{-\top}$ and the
cofactor formula therefore give \eqref{geom:AB}.
Since
\[
 \rme^{-h}AB=\operatorname{cof}\nabla(X\circ y),
\]
Piola's identity yields \eqref{geom:AB-Piola}.
Since $h$ and $\tY$ are independent of time,
differentiating \eqref{detnazX} and using
$AB=[\nabla(X\circ y)]^{-\top}$ give
\[
 0=\p_t\det\nabla(X\circ y)
 =\rme^{-h}(AB)_{ik}\partial_k\p_t\bY^i.
\]
Together with \eqref{geom:AB-Piola}, this implies
\[
 \begin{aligned}
 \partial_k\bigl(
 \rme^{-h}(AB)_{ik}\p_t\bY^i
 \bigr)
 &=
 \partial_k\bigl(\rme^{-h}(AB)_{ik}\bigr)\p_t\bY^i
 +\rme^{-h}(AB)_{ik}\partial_k\p_t\bY^i=0,
 \end{aligned}
\]
which proves \eqref{geom:velocity-div}. Finally, \eqref{matrix-G} follows from the definition of $G$ and \eqref{geom:AB}.
\end{proof}

For the mapping $y=y(z)$ defined in \eqref{mapyz} and $h=h(z)$ defined in \eqref{pz2y2},
 we collect here estimates for the Jacobian
$\nabla_z y$, $\rme^h$. The estimate for
$B-I,\, \rho^{\pm1}-1,\,\nabla(b_0^1(y(z))),\,\gamma',\,\pa_1(b_0^1(y(z))),\,
\Phi,\, \tY$
are also presented. All of which are used frequently throughout the paper.
Then we show that $\|\pa_1f\|_{H^s}$ can be bounded by $\|\Db f\|_{H^s}$
and a commutator estimate involving $\pa_{z_1}f$ and $b_0^1(y(z))$.
The detailed proof is given in Appendix~\ref{app:C}.
\begin{lemma}\label{coordinate-lemma}
The fixed coordinate map and the corrector satisfy
\begin{align}
 &\|\nabla y-I\|_{H^a}
 +\|\rme^{\pm h}-1\|_{H^a}
 +\|B-I\|_{H^a}
 +\|\rho^{\pm1}-1\|_{H^a}
 \le C\|b_0^2\|_{H^{a+1}},
 \label{static-est3}\\
 &\|\nabla(b_0^1(y(z)))\|_{H^a}
 +\|\gamma'\|_{H^a}
 \le C,
 \label{static-est2}\\
 &\|\pa_1(b_0^1(y(z)))\|_{H^a}
 +\|\Phi\|_{H^{a+1}}
 +\|\tY\|_{H^{a+1}}
 \le C\|b_0-(\xi,0)^\top\|_{H^{a+1}}.
 \label{static-est1}
\end{align}
For every integer $0\le s\le a$ and every $f$ such that
$\Db f\in H^s$, we have
\begin{equation}\label{equiva-1}
 \|\pa_1f\|_{H^s}\le C\|\Db f\|_{H^s},
\end{equation}
and
\begin{equation}\label{equ-b0-norm}
 C^{-1}\|\Db f\|_{H^s}^2
 \le \sum_{|\alpha|\le s}
       \|\Db\partial^\alpha f\|_{L^2}^2
 \le C\|\Db f\|_{H^s}^2.
\end{equation}
Moreover, for every integer $1\le s\le a$, every multi-index
$\alpha$ with $|\alpha|\le s$, and every $f$ such that $\Db f\in H^{s-1}$, we have
\begin{equation}\label{magnetic-commutator-bound}
 \big\|[\partial^\alpha,b_0^1(y(z))]\pa_1f\big\|_{L^2}
 +\big\|[\partial^\alpha,(b_0^1(y(z)))^2]\pa_1f\big\|_{L^2}
 \le C\|\Db f\|_{H^{s-1}}.
\end{equation}
\end{lemma}

Now we derive the inherent identity for $\hbox{div} \bY$
and the estimate for the low frequency $\p_2\bY_L$.
\begin{lemma}\label{constraint-lemma}
The corrected displacement $\bY$ satisfies
\begin{equation}\label{constraint}
 \rho\partial_1\bar Y_H^1+\partial_2\bar Y^2
 =\partial_1\bar Y_H^2\partial_2Y^1
 -\partial_1Y^1\partial_2\bar Y^2,
\end{equation}
and
\begin{align}\label{vertical-average}
\begin{split}
  \partial_2\bar Y_L^2
 &=\big\langle
 \partial_1\bar Y_H^2\partial_2(\widetilde Y^1+\bar Y_H^1)
 -\partial_1(\widetilde Y^1+\bar Y_H^1)\partial_2\bar Y_H^2
 \big\rangle\\&=\partial_2\big\langle
 (\widetilde Y^1+\bar Y_H^1)\partial_1\bar Y_H^2
 \big\rangle.
\end{split}
\end{align}
Consequently,
\begin{align}
 \|\partial_2\bar Y_L^2\|_{H^a}
 &\le C\bigl(\epsilon_0+\|\nabla\bar Y_H\|_{H^a}\bigr)
 \|\nabla\bar Y_H\|_{H^a},\label{vertical-est}\\
 \|\p_t\partial_2\bar Y_{L}^2\|_{H^a}
 &\le C\bigl(\epsilon_0+\|\nabla\bar Y_H\|_{H^a}\bigr)
 \|\p_t\nabla\bar Y_{H}\|_{H^a},\label{vertical-time-est}\\
 \|\p_t\partial_2\bar Y_{L}^2\|_{L^2(\mathbb R)}
 &\le C\bigl(\epsilon_0+\|\nabla\bar Y_H\|_{L^\infty}\bigr)
 \|\p_t\nabla\bar Y_{H}\|_{L^2}.\label{vertical-time-low}
\end{align}
If, in addition, $\|\nabla Y\|_{L^\infty}$ is sufficiently small, then
\begin{equation}\label{vertical-tangent}
 \|\partial_2\bar Y^2\|_{L^2}
 +\|\partial_2\bar Y_L^2\|_{L^2}
 +\|\partial_2\bar Y_H^2\|_{L^2}
 \le C\|\partial_1\bar Y_H\|_{L^2}.
\end{equation}
\end{lemma}
\begin{proof}
Expanding the determinant in \eqref{nazX} and using
\eqref{detnazX}, we obtain
\[
 \rho\partial_1\bar Y_H^1
 +\big(
 \rme^{-h}\rho^{-1}+\partial_1\bar Y_H^1
 \big)\partial_2\bar Y^2
 =\partial_1\bar Y_H^2\partial_2Y^1.
\]
By Lemma~\ref{geo-identities},
$Y=\widetilde Y+\bar Y$, and
$\partial_1\bar Y=\partial_1\bar Y_H$, we have
\[
 \rme^{-h}\rho^{-1}+\partial_1\bar Y_H^1
 =1+\partial_1\widetilde Y^1+\partial_1\bar Y_H^1
 =1+\partial_1Y^1.
\]
Substituting this identity into the preceding equality proves \eqref{constraint}.

Since $\rho$ and $\bar Y_L$ are independent of $z_1$, we have
\[
 \rho\langle\partial_1\bar Y_H^1\rangle=0,
 \qquad
 \langle\partial_1Y^1\rangle=0,
 \qquad
 \langle\partial_2\bar Y_H^2\rangle=0.
\]
Integrating \eqref{constraint} over $\mathbb T$ with respect
to $z_1$
and using $Y=\widetilde Y+\bar Y_H+\bar Y_L$, we obtain
\begin{align}\label{p2Y2-1}
\begin{split}
 \partial_2\bar Y_L^2
 &=\big\langle
 \partial_1\bar Y_H^2\partial_2Y^1
 -\partial_1Y^1\partial_2\bar Y^2
 \big\rangle\\
 &=\big\langle
 \partial_1\bar Y_H^2
 \partial_2(\widetilde Y^1+\bar Y_H^1)
 -\partial_1(\widetilde Y^1+\bar Y_H^1)
 \partial_2\bar Y_H^2
 \big\rangle,
\end{split}
\end{align}
where the terms containing $\partial_2\bar Y_L^1$
and $\partial_2\bar Y_L^2$ vanish by the preceding
identities. Moreover, integration by parts in $z_1$ to the last term on the right hand side of \eqref{p2Y2-1} , one has
\[
 \big\langle
 \partial_1\bar Y_H^2
 \partial_2(\widetilde Y^1+\bar Y_H^1)
 -\partial_1(\widetilde Y^1+\bar Y_H^1)
 \partial_2\bar Y_H^2
 \big\rangle
 =\partial_2\big\langle
 (\widetilde Y^1+\bar Y_H^1)\partial_1\bar Y_H^2
 \big\rangle.
\]
This identity and \eqref{p2Y2-1} prove \eqref{vertical-average}.

By the boundedness of horizontal averaging, the Sobolev
product estimate, and \eqref{static-est1}, we have
\begin{align*}
 \|\partial_2\bar Y_L^2\|_{H^a}
 &\le C\|\nabla\bar Y_H^2\|_{H^a}
 \|\nabla(\widetilde Y^1+\bar Y_H^1)\|_{H^a}\\
 &\le C\big(
 \epsilon_0+\|\nabla\bar Y_H\|_{H^a}
 \big)\|\nabla\bar Y_H\|_{H^a},
\end{align*}
which proves \eqref{vertical-est}.

Since $\widetilde Y$ is independent of time,
differentiating \eqref{vertical-average} in time
and applying the same product estimate gives
\begin{align*}
 \|\p_t\partial_2\bar Y_L^2\|_{H^a}
 &\le C\Big(
 \|\nabla(\widetilde Y^1+\bar Y_H^1)\|_{H^a}
 \|\p_t\nabla\bar Y_H^2\|_{H^a}
 +\|\nabla\bar Y_H^2\|_{H^a}
 \|\p_t\nabla\bar Y_H^1\|_{H^a}
 \Big)\\
 &\le C\big(
 \epsilon_0+\|\nabla\bar Y_H\|_{H^a}
 \big)\|\p_t\nabla\bar Y_H\|_{H^a}.
\end{align*}
This proves \eqref{vertical-time-est}.
Similarly, by H\"older's inequality and \eqref{static-est1}, we obtain
\begin{align*}
 \|\p_t\partial_2\bar Y_L^2\|_{L^2(\mathbb R)}
 &\le C\Big(
 \|\nabla(\widetilde Y^1+\bar Y_H^1)\|_{L^\infty}
 +\|\nabla\bar Y_H^2\|_{L^\infty}
 \Big)\|\p_t\nabla\bar Y_H\|_{L^2}\\
 &\le C\big(
 \epsilon_0+\|\nabla\bar Y_H\|_{L^\infty}
 \big)\|\p_t\nabla\bar Y_H\|_{L^2},
\end{align*}
which proves \eqref{vertical-time-low}.

Finally, rearranging \eqref{constraint} gives
\begin{equation}\label{divided}
 \partial_2\bar Y^2
 =\frac{
 -\rho\partial_1\bar Y_H^1
 +\partial_1\bar Y_H^2\partial_2Y^1
 }{1+\partial_1Y^1}.
\end{equation}
Taking $\|\nabla Y\|_{L^\infty}$ sufficiently small,
we have $1+\partial_1Y^1\ge1/2$.
Using $\|\rho\|_{L^\infty}\le C$, which follows from
\eqref{static-est3}, we deduce from \eqref{divided} that
\begin{align*}
 \|\partial_2\bar Y^2\|_{L^2}
 &\le C\Big(
 \|\rho\|_{L^\infty}\|\partial_1\bar Y_H^1\|_{L^2}
 +\|\partial_2Y^1\|_{L^\infty}
  \|\partial_1\bar Y_H^2\|_{L^2}
 \Big)\le C\|\partial_1\bar Y_H\|_{L^2}.
\end{align*}
Since $\partial_2\bar Y_L^2=\langle\partial_2\bar Y^2\rangle$ and $\partial_2\bar Y_H^2=\partial_2\bar Y^2-\partial_2\bar Y_L^2$, one gets from the Cauchy--Schwarz inequality that
\[
 \|\partial_2\bar Y_L^2\|_{L^2}
 \le C\|\partial_2\bar Y^2\|_{L^2},\quad \|\partial_2\bar Y_H^2\|_{L^2}\le C\|\partial_2\bar Y^2\|_{L^2}.
\]
Combining the preceding two estimates proves
\eqref{vertical-tangent}.
\end{proof}


We shall use the bootstrap condition
\begin{equation}\label{bootstrap}
 \HH(T)+\LL(T)+\WW(T)+\SSS(T)\le\delta,
\end{equation}
where $\delta>0$ will be chosen sufficiently small.

Under \eqref{bootstrap}, we first show some preliminary estimate for $\nabla Y,\, AB-I,\, G-I$ and $G_t $.
They will be used in the estimate of the nonlinear terms and the estimate of the pressure.
\begin{lemma}
Assume that \eqref{bootstrap} holds on $[0,T]$.
If $\delta$ is sufficiently small, then
\begin{equation}\label{geometry-small}
 \|\nabla Y\|_{L^\infty(0,T;H^a)}
 +\|AB-I\|_{L^\infty(0,T;H^a)}
 +\|G-I\|_{L^\infty(0,T;H^a)}
 \le C(\epsilon_0+\sqrt\delta),
\end{equation}
and
\begin{equation}\label{Gt-est}
 \|G_t\|_{L^2(0,T;H^a)}
 \le C\bigl(\mathcal H_a(T)^{1/2}
 +\mathcal L_a(T)^{1/2}\bigr).
\end{equation}
\end{lemma}
\begin{proof}
Since $Y=\widetilde Y+\bar Y_H+\bar Y_L$ and
$\partial_1\bar Y_L=0$, we deduce from
\eqref{static-est1}, \eqref{vertical-est},
\eqref{def-Ha}, and \eqref{Sa} that
\begin{align*}
 \|\nabla Y\|_{L^\infty_t(H^a)}
 &\le \|\nabla\widetilde Y\|_{H^a}
 +\|\nabla\bar Y_H\|_{L^\infty_t(H^a)}
 +\|\partial_2\bar Y_L^1\|_{L^\infty_t(H^a)}
 +\|\partial_2\bar Y_L^2\|_{L^\infty_t(H^a)}\\
 &\le C\epsilon_0
 +C\mathcal H_a(T)^{1/2}
 +C\mathcal S_a(T)^{1/2}
 +C\big(
 \epsilon_0+\mathcal H_a(T)^{1/2}
 \big)\mathcal H_a(T)^{1/2}\\
 &\le C(\epsilon_0+\sqrt\delta),
\end{align*}
where the last inequality follows from \eqref{bootstrap}. By \eqref{corrector-identities}, \eqref{rho-identities},
and $Y=\widetilde Y+\bar Y$, we have
\[
 \begin{aligned}
 \rho+\partial_2\bar Y^2
 &=\rme^{-h}+\partial_2Y^2,\quad
 \partial_1\bar Y_H^2
 &=\frac{b_0^2}{b_0^1}(y(z))+\partial_1Y^2,\quad
 \rme^{-h}\rho^{-1}+\partial_1\bar Y_H^1
 &=1+\partial_1Y^1.
 \end{aligned}
\]
Substituting these identities into \eqref{geom:AB}
and using \eqref{def-B}, we obtain
\[
 AB-I
 =B-I+\rme^h
 \begin{pmatrix}
  \partial_2Y^2 & -\partial_1Y^2\\
  -\partial_2Y^1 & \partial_1Y^1
 \end{pmatrix}.
\]
Writing $\rme^h=1+(\rme^h-1)$ and applying the Sobolev
product estimate, we deduce from \eqref{static-est3}
and the preceding estimate for $\nabla Y$ that
\begin{align*}
 \|AB-I\|_{L^\infty_t(H^a)}
 &\le \|B-I\|_{H^a}
 +C\big(1+\|\rme^h-1\|_{H^a}\big)
       \|\nabla Y\|_{L^\infty_t(H^a)}\le C(\epsilon_0+\sqrt\delta).
\end{align*}
Moreover, $G=(AB)^\top AB$ implies
\[
 G-I=(AB-I)^\top+(AB-I)+(AB-I)^\top(AB-I).
\]
Applying the Sobolev product estimate, we obtain
\begin{align*}
 \|G-I\|_{L^\infty_t(H^a)}
 &\le C\big(
 1+\|AB-I\|_{L^\infty_t(H^a)}
 \big)\|AB-I\|_{L^\infty_t(H^a)}\le C(\epsilon_0+\sqrt\delta).
\end{align*}
This proves \eqref{geometry-small}.

We next estimate $G_t$. Since $y$ is independent of time, differentiating
$[\nabla (y+Y)]^\top AB=I$ gives
\[
 (\nabla Y_t)^\top AB
 +[\nabla (y+Y)]^\top(AB)_t=0.
\]
Multiplying this identity from the left by $AB$
and differentiating $G=(AB)^\top AB$, we obtain
\[
 (AB)_t=-AB(\nabla Y_t)^\top AB,
 \qquad
 G_t=(AB)_t^\top AB+(AB)^\top(AB)_t.
\]
Writing $AB=I+(AB-I)$, we deduce from \eqref{geometry-small}
that, for $0\le t\le T$,
\begin{align*}
 \|(AB)_t\|_{H^a}
 &\le C\big(1+\|AB-I\|_{H^a}\big)^2
       \|\nabla Y_t\|_{H^a}
 \le C\|\nabla Y_t\|_{H^a},\\
 \|G_t\|_{H^a}
 &\le C\big(1+\|AB-I\|_{H^a}\big)
       \|(AB)_t\|_{H^a}
 \le C\|\nabla Y_t\|_{H^a}.
\end{align*}
Finally, $\widetilde Y$ is independent of time and
$\bar Y_L$ is independent of $z_1$. Hence, by
\eqref{def-Ha} and \eqref{def-La},
\begin{align*}
 \|G_t\|_{L^2_t(H^a)}
 &\le C\|\nabla\p_t Y\|_{L^2_t(H^a)}\le C\big(
 \|\nabla \p_t\bar Y_H\|_{L^2_t(H^a)}
 +\|\partial_2\p_t\bar Y_L\|_{L^2_t(H^a)}
 \big)\\
 &\le C\big(
 \mathcal H_a(T)^{1/2}+\mathcal L_a(T)^{1/2}
 \big).
\end{align*}
This proves \eqref{Gt-est} and completes the proof.
\end{proof}
\section{Pressure estimates}\label{sec:3}

In this section, we derive the  pressure equation in the $z$-coordinates and estimate $\nabla p$ and $\nabla(p-q)$.
\begin{lemma}\label{pressure-equation-lemma}
Denote the pressure in the
$z$-coordinates by $p(t,z):=p(t,X(t,y(z)))$.
Then
\begin{align}\label{pressure-det}
\begin{split}
 \Delta p
 ={}&\diver\big((I-\rme^{-h}G)\nabla p\big)\\
 &+2\big(\det\nabla\p_t\bY_H
         -\det\nabla\Db\bY_H\big)\\
 &+2\big(\pa_2\p_t\bY_L^2\pa_1\p_t\bY_H^1
         -\pa_2\p_t\bY_L^1\pa_1\p_t\bY_H^2\big)
 +2\gamma'\pa_1\Db\bY_H^2.
\end{split}
\end{align}
Moreover, for each fixed $t$, the Poisson equation
\begin{equation}\label{eq-q}
 \Delta q=2\pa_1(\gamma'\Db\bY_H^2)
\end{equation}
admits a unique solution
$q\in L^2$, given by
$
 q=2\Delta^{-1}\pa_1(\gamma'\Db\bY_H^2).
$
This solution satisfies
\begin{equation}\label{q-properties}
 \langle q\rangle=0,
 \qquad
 \|q\|_{L^2}\le C\|\Db\bY_H\|_{L^2}.
\end{equation}
\end{lemma}
\begin{proof}
By \eqref{mapyz},
\[
 b_0^1(y(z))\pa_1y(z)=b_0(y(z)).
\]
Thus \eqref{eqbt} and the chain rule give
\begin{align*}
 b(t,X(t,y(z)))
 &=b_0^1(y(z))\pa_1\big[X(t,y(z))\big]=b_0(y(z))+\Db Y(t,z),
\end{align*}
where we have used
$X(t,y(z))=y(z)+Y(t,z)$ and
$\Db=b_0^1(y(z))\pa_1$.
Since $Y=\widetilde Y+\bY_H+\bY_L$ and
$\pa_1\bY_L=0$, \eqref{corrector-identities}
and the definition of $\Phi$ imply
\begin{equation}\label{pressure-magnetic-identity}
 \begin{aligned}
 b(t,X(t,y(z)))
 &=b_0(y(z))+\Db\widetilde Y(z)+\Db\bY_H(t,z)\\
 &=b_0(y(z))-\Phi(z)+\Db\bY_H(t,z)\\
 &=(\gamma(z_2),0)^{\top}+\Db\bY_H(t,z).
 \end{aligned}
\end{equation}
Moreover, \eqref{flow} and the time independence of
$y$ and $\widetilde Y$ give
\begin{equation}\label{pressure-velocity-identity}
 u(t,X(t,y(z)))
 =\p_tY(t,z)=\p_t\bY(t,z).
\end{equation}

Taking the divergence of the Eulerian momentum equation
and using $\diver_xu=\diver_xb=0$, we obtain
\begin{equation}\label{pres-euler}
 \Delta_xp
 =-\partial_{x_i}u^j\partial_{x_j}u^i
  +\partial_{x_i}b^j\partial_{x_j}b^i=2\det\nabla_xu-2\det\nabla_xb.
\end{equation}
Here we have used the fact that  $\operatorname{tr}(M^2)
 =(\operatorname{tr}M)^2-2\det M$ for any $2\times2$ matrix $M$.

By Lemma~\ref{geo-identities},
$AB=[\nabla_z(X\circ y)]^{-\top}$ and
$\det\nabla_z(X\circ y)=\rme^{-h}$.
The chain rule and the Piola identity
\eqref{geom:AB-Piola}, together with
$G=(AB)^\top AB$, therefore yield
\[
 (\Delta_xp)(t,X(t,y(z)))
 =\rme^h\diver(\rme^{-h}G\nabla p).
\]
Furthermore, for $v=u,b$, the chain rule gives
\begin{align*}
 (\det\nabla_xv)(t,X(t,y(z)))
 &=\frac{\det\nabla\big[v(t,X(t,y(z)))\big]}
         {\det\nabla(X(t,y(z)))}=\rme^h\det\nabla\big[v(t,X(t,y(z)))\big].
\end{align*}
Substituting these identities into the pressure equation \eqref{pres-euler}
and using \eqref{pressure-magnetic-identity} and
\eqref{pressure-velocity-identity}, we obtain
\[
 \diver(\rme^{-h}G\nabla p)
 =2\det\nabla\p_t\bY
 -2\det\nabla\big((\gamma,0)^\top
                     +\Db\bY_H\big).
\]
Since $\bY_L$ and $\gamma$ are independent of $z_1$,
\begin{align*}
 \det\nabla\p_t\bY
 &=\det\nabla\p_t\bY_H
 +\pa_2\p_t\bY_L^2\pa_1\p_t\bY_H^1
 -\pa_2\p_t\bY_L^1\pa_1\p_t\bY_H^2,\\
 \det\nabla\big((\gamma,0)^\top+\Db\bY_H\big)
 &=\det\nabla\Db\bY_H
 -\gamma'\pa_1\Db\bY_H^2.
\end{align*}
Plugging these expansions into the preceding identity,
we obtain \eqref{pressure-det}.

We now show the existence, uniqueness, and some
properties of the linear pressure $q$. Fix time $t$. To construct an $L^2$ solution
of \eqref{eq-q}, define
\[
 \widehat q(\lambda)
 =
 \begin{cases}
 \displaystyle
 -\frac{\mathrm i\lambda_1}{\pi|\lambda|^2}
 \widehat{\gamma'\Db\bY_H^2}(\lambda),
 & \lambda_1\ne0,\\[6pt]
 0,
 & \lambda_1=0,
 \end{cases}
 \qquad
 \lambda=(\lambda_1,\lambda_2)\in\mathbb Z\times\mathbb R.
\]
Since $|\lambda_1|\ge1$ whenever $\lambda_1\ne0$,
\[
 \frac{|\lambda_1|}{\pi|\lambda|^2}
 \le \frac1{\pi|\lambda_1|}
 \le \frac1\pi.
\]
Thus Plancherel's identity, \eqref{static-est2}, and
Sobolev embedding show that this formula defines
an $L^2$ function satisfying
\[
 \begin{aligned}
 \|q\|_{L^2}
 &\le \frac1\pi\|\gamma'\Db\bY_H^2\|_{L^2}\le \frac1\pi\|\gamma'\|_{L^\infty}
                   \|\Db\bY_H^2\|_{L^2}
 \le C\|\Db\bY_H\|_{L^2}.
 \end{aligned}
\]
Moreover,
\[
 -4\pi^2|\lambda|^2\widehat q(\lambda)
 =4\pi\mathrm i\lambda_1
   \widehat{\gamma'\Db\bY_H^2}(\lambda)
 =\widehat{2\pa_1(\gamma'\Db\bY_H^2)}(\lambda).
\]
Hence $q$ solves \eqref{eq-q},
which shows existence. By construction,
$
 q=2\Delta^{-1}\pa_1(\gamma'\Db\bY_H^2),
$
where $\Delta^{-1}$ acts on nonzero horizontal frequencies,
with the zero horizontal mode set to zero.
The Uniqueness follows from the fact that every $L^2$ harmonic
function on $\mathbb T\times\mathbb R$ vanishes, as is
immediate from the Fourier transform.

Finally, integrating \eqref{eq-q} over $\mathbb T$ with respect
to $z_1$ and using periodicity, we obtain
$
 \pa_2^2\langle q\rangle=0
$
in the sense of distributions. Consequently,
$ \langle q\rangle(t,z_2)=c_0(t)+c_1(t)z_2.
$
Since
\[
 \|\langle q\rangle\|_{L^2(\mathbb R)}
 \le \|q\|_{L^2(\mathbb T\times\mathbb R)}<\infty,
\]
we must have $c_0(t)=c_1(t)=0$.
Together with the preceding $L^2$ estimate, we finish the proof of
\eqref{q-properties}.
\end{proof}
\begin{lemma}\label{pressure-lemma}
Assume that \eqref{bootstrap} holds on $[0,T]$.
If $\delta$ is sufficiently small, then, for $s\in\{0,a\}$,
\begin{equation}\label{pressure-basic}
 \|\nabla p\|_{H^s}
 \le C\big(
 \|\Db\bY_H\|_{H^s}
 +\|\nabla\p_t\bY_H\|_{H^s}
 \big),
\end{equation}
and
\begin{equation}\label{pressure-remainder}
 \begin{aligned}
 &\|\nabla(p-q)\|_{H^s}
 +\|AB\nabla p-\nabla q\|_{H^s}\le C(\epsilon_0+\sqrt\delta)
 \big(
 \|\Db\bY_H\|_{H^s}
 +\|\nabla\p_t\bY_H\|_{H^s}
 \big).
 \end{aligned}
\end{equation}
\end{lemma}
\begin{proof}
Let us begin with the estimate of $\nabla q$.
Writing
$I-\rme^{-h}G=(1-\rme^{-h})I-\rme^{-h}(G-I)$,
we deduce from \eqref{static-est3},
\eqref{geometry-small}, and the Sobolev product estimate
that
\begin{equation}\label{pressure-coeff-small}
 \|I-\rme^{-h}G\|_{H^a}+\|AB-I\|_{H^a}
 \le C(\epsilon_0+\sqrt\delta).
\end{equation}
By \eqref{def-q}, the boundedness of the Fourier multiplier
$\nabla\Delta^{-1}\pa_1$, and \eqref{static-est2},
\begin{equation}\label{pressure-q-gradient}
 \|\nabla q\|_{H^s}
 \le C\|\gamma'\Db\bY_H^2\|_{H^s}
 \le C\|\Db\bY_H\|_{H^s},
 \qquad s\in\{0,a\}.
\end{equation}

Writing the determinants in \eqref{pressure-det} in
divergence form and subtracting \eqref{eq-q}, we obtain,
since $\gamma$ and $\bY_L$ are independent of $z_1$,
\begin{align*}
\begin{split}
 \Delta(p-q)
 ={}&\diver\big((I-\rme^{-h}G)\nabla p\big)+2\diver\big(
 \p_t\bY_H^1\pa_2\p_t\bY_H^2,
 -\p_t\bY_H^1\pa_1\p_t\bY_H^2\big)\\
 &-2\diver\big(
 \Db\bY_H^1\pa_2\Db\bY_H^2,
 -\Db\bY_H^1\pa_1\Db\bY_H^2\big)+2\pa_1\big(
 \pa_2\p_t\bY_L^2\p_t\bY_H^1
 -\pa_2\p_t\bY_L^1\p_t\bY_H^2\big).
\end{split}
\end{align*}
The $L^2$ boundedness of
$\nabla\Delta^{-1}\diver$, together with
\eqref{pressure-coeff-small}, \eqref{bootstrap},
\eqref{def-Ha}, \eqref{def-La}, Sobolev embedding,
and the Poincar\'e inequality, yields
\begin{align*}
 \|\nabla(p-q)\|_{L^2}
 &\le C\|(I-\rme^{-h}G)\nabla p\|_{L^2}
      +C\|\p_t\bY_H^1\nabla\p_t\bY_H^2\|_{L^2}\\
 &\quad+C\|\Db\bY_H^1\nabla\Db\bY_H^2\|_{L^2}+C\|\pa_2\p_t\bY_L^2\p_t\bY_H^1
             -\pa_2\p_t\bY_L^1\p_t\bY_H^2\|_{L^2}\\
 &\le C\|I-\rme^{-h}G\|_{H^a}\|\nabla p\|_{L^2}
      +C\|\p_t\bY_H\|_{L^\infty}
        \|\nabla\p_t\bY_H\|_{L^2}\\
 &\quad+C\|\Db\bY_H\|_{L^2}
          \|\nabla\Db\bY_H\|_{L^\infty}
      +C\|\pa_2\p_t\bY_L\|_{L^\infty}
        \|\p_t\bY_H\|_{L^2}\\
 &\le C(\epsilon_0+\sqrt\delta)\|\nabla p\|_{L^2}
      +C\sqrt\delta\Big(
        \|\Db\bY_H\|_{L^2}
        +\|\nabla\p_t\bY_H\|_{L^2}\Big).
\end{align*}
Combining this with \eqref{pressure-q-gradient} and taking
$\epsilon_0$ and $\delta$ sufficiently small, we absorb
the term involving $\|\nabla p\|_{L^2}$ and obtain
\eqref{pressure-basic} for $s=0$.
Substituting this bound into the preceding inequality gives
\begin{equation}\label{pressure-difference-L2}
 \|\nabla(p-q)\|_{L^2}
 \le C(\epsilon_0+\sqrt\delta)\Big(
 \|\Db\bY_H\|_{L^2}
 +\|\nabla\p_t\bY_H\|_{L^2}\Big).
\end{equation}

For the higher derivatives, we have
\begin{align}\label{nap-q-Ha}
   \|\nabla(p-q)\|_{H^a}
 \le C\Big(
 \|\nabla(p-q)\|_{L^2}
 +\|\Delta(p-q)\|_{H^{a-1}}\Big).
\end{align}
Since $a-1\ge2$, it follows from
\eqref{pressure-det} and \eqref{eq-q} that
\begin{align}\label{Deltap-q-Ha}
    \begin{split}
    \|\Delta(p-q)\|_{H^{a-1}}\leq& C\|\diver\big((I-\rme^{-h}G)\nabla p\big)\|_{H^{a-1}}+C\|\det\nabla\p_t\bY_H\|_{H^{a-1}}
      \\&+C\|\det\nabla\Db\bY_H\|_{H^{a-1}}+C\|\pa_2\p_t\bY_L^2\pa_1\p_t\bY_H^1
             -\pa_2\p_t\bY_L^1\pa_1\p_t\bY_H^2\|_{H^{a-1}}\\
 \le &C\|I-\rme^{-h}G\|_{H^a}\|\nabla p\|_{H^a}
      +C\|\Db\bY_H\|_{H^a}^2\\&+C\big(
       \|\p_t\bY_H\|_{H^a}
       +\|\p_t\bY_L\|_{H^a}\big)
       \|\nabla\p_t\bY_H\|_{H^a}\\
 \le &C(\epsilon_0+\sqrt\delta)\big(
       \|\nabla p\|_{H^a}
       +\|\Db\bY_H\|_{H^a}
       +\|\nabla\p_t\bY_H\|_{H^a}\big),
    \end{split}
\end{align}
where the last inequality follows from
\eqref{pressure-coeff-small},
\eqref{pressure-difference-L2}, and \eqref{bootstrap}. We get from \eqref{pressure-difference-L2}, \eqref{nap-q-Ha} and \eqref{Deltap-q-Ha} that
\begin{align*}
    \begin{split}
     \|\nabla(p-q)\|_{H^a}\leq &C(\epsilon_0+\sqrt\delta)\big(
       \|\nabla p\|_{H^a}
       +\|\Db\bY_H\|_{H^a}
       +\|\nabla\p_t\bY_H\|_{H^a}\big).
    \end{split}
\end{align*}
Combining this with \eqref{pressure-q-gradient} and
absorbing the term involving $\|\nabla p\|_{H^a}$ proves
\eqref{pressure-basic} for $s=a$.
Substitution into the preceding inequality, together with
\eqref{pressure-difference-L2}, then yields
\begin{equation}\label{pressure-difference-Hs}
 \|\nabla(p-q)\|_{H^s}
 \le C(\epsilon_0+\sqrt\delta)\big(
 \|\Db\bY_H\|_{H^s}
 +\|\nabla\p_t\bY_H\|_{H^s}\big),
 \qquad s\in\{0,a\}.
\end{equation}

Finally, by \eqref{pressure-coeff-small},
\eqref{pressure-basic}, and \eqref{pressure-difference-Hs},
\begin{align*}
 \|AB\nabla p-\nabla q\|_{H^s}
 &\le \|(AB-I)\nabla p\|_{H^s}
      +\|\nabla(p-q)\|_{H^s}\\
 &\le C\|AB-I\|_{H^a}\|\nabla p\|_{H^s}
      +\|\nabla(p-q)\|_{H^s}\\
 &\le C(\epsilon_0+\sqrt\delta)\big(
       \|\Db\bY_H\|_{H^s}
       +\|\nabla\p_t\bY_H\|_{H^s}\big),
 \qquad s\in\{0,a\}.
\end{align*}
We thus finish the proof of
\eqref{pressure-remainder}.
\end{proof}
\section{Estimate of the high-frequency component: the linear estimate}
\label{high-linear-section}
In this section, we derive the energy estimate for the linear system
\begin{equation}\label{high-linear-system}
\begin{cases}
\p_t^2\bar Y_H-\Delta\p_t\bar Y_H
-\big(\p_{b_0^1}^2\bar Y_H\big)_H+\nabla q
=\mathcal F,\\
q=2\Delta^{-1}\p_1\big(\gamma'\p_{b_0^1}\bar Y_H^2\big),
\end{cases}
\end{equation}
where the nonlinear term $\mathcal F$ is given by
\begin{align*}
\mathcal F
={}&\Big(
\diver\big((G-I)\nabla\bar Y_t\big)
-G\nabla h\cdot\nabla\bar Y_t
-\big(AB\nabla p-\nabla q\big)
\Big)_H.
\end{align*}
We first establish the following linear energy estimate.
The estimate of $\mathcal F$ will be given in
Section~\ref{high-nonlinear-section}.

\begin{proposition}\label{high-linear-prop}
The following estimate holds:
\begin{align}\label{high-linear-bound}
\begin{split}
\mathcal H_a(T)
\le{}& C\mathcal H_a(0)
+C\int_0^T
\|\p_{b_0^1}\bar Y_H\|_{L^2}
\big(
\|\mathcal R_H^1\|_{L^2}
+\|\mathcal R_H^2\|_{L^2}
\big)\rmd\tau\\
&+C\sup_{0\le t\le T}
\sum_{|\alpha|\le a}
\left|
\int_0^t
\Big(
\p^\alpha\mathcal F,\,
\p^\alpha\big(\p_\tau\bar Y_H+\tfrac14\bar Y_H\big)
\Big)_{L^2}\rmd\tau
\right|.
\end{split}
\end{align}
Here the remainders $\mathcal R_H^1$ and $\mathcal R_H^2$
are given by
\[
 \begin{aligned}
  \mathcal R_H^1
  &:=
  (1-\rho)\pa_1\bY_H^1
  +\pa_1\bY_H^2\pa_2Y^1
  -\pa_1Y^1\pa_2\bY^2,\qquad
  \mathcal R_H^2
  :=
  \pa_t\mathcal R_H^1.
 \end{aligned}
\] They satisfy
\begin{equation}\label{high-divergence-remainders}
\diver\bar Y_H
=-\p_2\bar Y_L^2+\mathcal R_H^1,
\qquad
\diver\p_t\bar Y_H
=-\p_2\p_t\bar Y_L^2+\mathcal R_H^2.
\end{equation}
\end{proposition}
\begin{proof}
Since the spatial derivatives commute with the horizontal projections and $\bar Y_H,\p_t\bar Y_H$ have zero horizontal mean, the projection on the magnetic term in \eqref{high-linear-system} can be omitted when taking the $L^2$ inner product with
$\partial^\alpha(\p_t\bar Y_H+\frac14\bar Y_H)$.

\medskip
\noindent
\textbf{Step 1: Estimate of $\mathcal H_0(T)$.}
Taking the $L^2$ inner product of
\eqref{high-linear-system} with
$\p_t\bar Y_H+\frac14\bar Y_H$, and integrating by parts,
we obtain
\refstepcounter{equation}\label{high-zero-energy}
$$
\begin{aligned}
&\frac{\rmd}{\rmd t}
\Big(
\frac12\|\p_t\bar Y_H\|_{L^2}^2
+\frac12\|\partial_{b_0^1}\bar Y_H\|_{L^2}^2
+\frac18\|\nabla\bar Y_H\|_{L^2}^2
+\frac14(\p_t\bar Y_H,\bar Y_H)_{L^2}
\Big)
\\
&\quad
+\|\nabla\p_t\bar Y_H\|_{L^2}^2
-\frac14\|\p_t\bar Y_H\|_{L^2}^2
+\frac14\|\partial_{b_0^1}\bar Y_H\|_{L^2}^2
\\
&=
-\Big(
(\partial_1b_0^1(y(z)))\partial_{b_0^1}\bar Y_H,
\p_t\bar Y_H+\frac14\bar Y_H
\Big)_{L^2}
-\Big(\nabla q,\p_t\bar Y_H+\frac14\bar Y_H\Big)_{L^2}
+\Big(
\mathcal F,
\p_t\bar Y_H+\frac14\bar Y_H
\Big)_{L^2}\\
&:=J_1+J_2+J_3.
\end{aligned}
\eqno(\theequation)
$$
Here we used the identity
$$
\partial_{b_0^1}^2\bar Y_H
=
\partial_1\bigl((b_0^1(y(z)))^2\partial_1\bar Y_H\bigr)
-(\partial_1b_0^1(y(z)))\partial_{b_0^1}\bar Y_H.
$$
For the rest of this proof, we write $b_0^1$ for
$b_0^1(y(z))$. The horizontal Poincar\'e inequality gives
$$
\|\bar Y_H\|_{L^2}\le\frac1{2\pi}\|\partial_1\bar Y_H\|_{L^2},
\qquad
\|\p_t\bar Y_H\|_{L^2}
\le\frac1{2\pi}\|\partial_1\p_t\bar Y_H\|_{L^2}.
$$
Consequently, the energy and dissipation on the left-hand side
of \eqref{high-zero-energy} satisfy
$$
\begin{aligned}
&\frac12\|\p_t\bar Y_H\|_{L^2}^2
+\frac12\|\partial_{b_0^1}\bar Y_H\|_{L^2}^2
+\frac18\|\nabla\bar Y_H\|_{L^2}^2
+\frac14(\p_t\bar Y_H,\bar Y_H)_{L^2}
\\
&\qquad\ge
\frac14\|\p_t\bar Y_H\|_{L^2}^2
+\frac12\|\partial_{b_0^1}\bar Y_H\|_{L^2}^2
+\frac1{16}\|\nabla\bar Y_H\|_{L^2}^2,
\\
&\|\nabla\p_t\bar Y_H\|_{L^2}^2
-\frac14\|\p_t\bar Y_H\|_{L^2}^2
+\frac14\|\partial_{b_0^1}\bar Y_H\|_{L^2}^2
\ge
\frac34\|\nabla\p_t\bar Y_H\|_{L^2}^2
+\frac14\|\partial_{b_0^1}\bar Y_H\|_{L^2}^2.
\end{aligned}
$$
We first estimate the magnetic coefficient term.
For $J_1$,
using the static coefficient bounds, Sobolev embedding,
and the horizontal Poincar\'e inequality, we have
$$
\begin{aligned}
|J_1|=\left|
\Big(
(\partial_1b_0^1)\partial_{b_0^1}\bar Y_H,
\p_t\bar Y_H+\frac14\bar Y_H
\Big)_{L^2}
\right|
&\le
C\epsilon_0
\|\partial_{b_0^1}\bar Y_H\|_{L^2}
\Big(
\|\nabla\p_t\bar Y_H\|_{L^2}
+\|\partial_{b_0^1}\bar Y_H\|_{L^2}
\Big)
\\
&\le
C\epsilon_0
\Big(
\|\nabla\p_t\bar Y_H\|_{L^2}^2
+\|\partial_{b_0^1}\bar Y_H\|_{L^2}^2
\Big).
\end{aligned}
$$

For the linear pressure term $J_2$, we use
\eqref{high-divergence-remainders} and the zero horizontal mean of
$q$ in \eqref{q-properties}. Integration by parts gives
$$
\begin{aligned}
-\Big(\nabla q,\p_t\bY_H+\tfrac14\bY_H\Big)_{L^2}
=\Big(q,\diver\p_t\bY_H+\tfrac14\diver\bY_H\Big)_{L^2}
=\Big(q,\mathcal R_H^2+\tfrac14\mathcal R_H^1\Big)_{L^2}.
\end{aligned}
$$
Hence, by \eqref{q-properties},
\refstepcounter{equation}\label{high-linear-pressure}
$$
\left|\Big(\nabla q,\p_t\bY_H+\tfrac14\bY_H\Big)_{L^2}\right|
\le C\|\Db\bY_H\|_{L^2}
\big(\|\mathcal R_H^1\|_{L^2}+\|\mathcal R_H^2\|_{L^2}\big).
\eqno(\theequation)
$$
Although $\gamma'$ need not be small, the linear pressure is paired
only with these remainders. Their estimates will be given in
Section~\ref{high-source-subsection}.

The contribution of $J_3$ is kept in its time-integrated form.
Integrating \eqref{high-zero-energy} over $[0,t]$, using the
preceding estimates, and then taking the
supremum over $0\le t\le T$ and absorbing
$C\epsilon_0\mathcal H_0(T)$,
\refstepcounter{equation}\label{high-zero-bound}
$$
\begin{aligned}
\mathcal H_0(T)
\le{}& C\mathcal H_a(0)
+C\int_0^T\|\Db\bY_H\|_{L^2}
\big(\|\mathcal R_H^1\|_{L^2}+\|\mathcal R_H^2\|_{L^2}\big)\dd t\\
&+C\sup_{0\le t\le T}\left|\int_0^t
\Big(\mathcal F,
\p_\tau\bY_H+\tfrac14\bY_H\Big)_{L^2}\dd\tau\right|.
\end{aligned}
\eqno(\theequation)
$$

\medskip
\noindent
\textbf{Step 2: Estimate of $\mathcal H_b(T)$ for $1\le b\le a$.}

Let $1\le b\le a$ and $|\alpha|\le b$.
We apply $\partial^\alpha$ to the first equation of
\eqref{high-linear-system} and take the $L^2$ inner product with
$\partial^\alpha(\p_t\bar Y_H+\frac14\bar Y_H)$.
We first treat the magnetic term. Using its divergence form, we obtain
$$
\begin{aligned}
-\Big(
\partial^\alpha\partial_{b_0^1}^2\bar Y_H,
\partial^\alpha\Big(\p_t\bar Y_H+\frac14\bar Y_H\Big)
\Big)_{L^2}
&=
\frac12\frac{\rmd}{\rmd t}
\|b_0^1\partial_1\partial^\alpha\bar Y_H\|_{L^2}^2
+\frac14\|b_0^1\partial_1\partial^\alpha\bar Y_H\|_{L^2}^2
\\
&\qquad
+\Big(
[\partial^\alpha,(b_0^1)^2]\partial_1\bar Y_H,
\partial_1\partial^\alpha
\Big(\p_t\bar Y_H+\frac14\bar Y_H\Big)
\Big)_{L^2}
\\
&\qquad
+\Big(
\partial^\alpha\bigl(
b_0^1(\partial_1b_0^1)\partial_1\bar Y_H
\bigr),
\partial^\alpha\Big(\p_t\bar Y_H+\frac14\bar Y_H\Big)
\Big)_{L^2}.
\end{aligned}
$$
Thus
\refstepcounter{equation}\label{high-order-energy}
$$
\begin{aligned}
&\frac{\rmd}{\rmd t}
\Big(
\frac12\|\partial^\alpha\p_t\bar Y_H\|_{L^2}^2
+\frac12\|b_0^1\partial_1\partial^\alpha\bar Y_H\|_{L^2}^2
+\frac18\|\nabla\partial^\alpha\bar Y_H\|_{L^2}^2
+\frac14(\partial^\alpha\p_t\bar Y_H,
\partial^\alpha\bar Y_H)_{L^2}
\Big)
\\
&\quad
+\|\nabla\partial^\alpha\p_t\bar Y_H\|_{L^2}^2
-\frac14\|\partial^\alpha\p_t\bar Y_H\|_{L^2}^2
+\frac14\|b_0^1\partial_1\partial^\alpha\bar Y_H\|_{L^2}^2
\\
&=
-\Big(
[\partial^\alpha,(b_0^1)^2]\partial_1\bar Y_H,
\partial_1\partial^\alpha
\Big(\p_t\bar Y_H+\frac14\bar Y_H\Big)
\Big)_{L^2}
-\Big(
\partial^\alpha\bigl(
b_0^1(\partial_1b_0^1)\partial_1\bar Y_H
\bigr),
\partial^\alpha\Big(\p_t\bar Y_H+\frac14\bar Y_H\Big)
\Big)_{L^2}
\\
&\quad
-\Big(
\partial^\alpha\nabla q,
\partial^\alpha\Big(\p_t\bar Y_H+\frac14\bar Y_H\Big)
\Big)_{L^2}
+\Big(
\partial^\alpha\mathcal F,
\partial^\alpha\Big(\p_t\bar Y_H+\frac14\bar Y_H\Big)
\Big)_{L^2}\\
&:=K_1+K_2+K_3+K_4.
\end{aligned}
\eqno(\theequation)
$$
By \eqref{equ-b0-norm} and the horizontal Poincar\'e
inequality, summing over $|\alpha|\le b$ shows that
the energy and dissipation in \eqref{high-order-energy}
control the quantities defining $\mathcal H_b$.

For $K_1$, \eqref{magnetic-commutator-bound} gives
\[
 \|[\partial^\alpha,(b_0^1)^2]\pa_1\bY_H\|_{L^2}
 \le C\|\Db\bY_H\|_{H^{b-1}},
 \qquad |\alpha|\le b.
\]
By the Cauchy--Schwarz inequality and \eqref{equiva-1},
we obtain, for every
$\eta>0$,
\begin{align*}
 &\sum_{|\alpha|\le b}\int_0^T
 \Big|
 \Big(
 [\partial^\alpha,(b_0^1)^2]\partial_1\bY_H,
 \partial_1\partial^\alpha
 \Big(\p_t\bY_H+\frac14\bY_H\Big)
 \Big)_{L^2}
 \Big|\,\rmd t\\
 &\quad\le
 C\int_0^T
 \|\Db\bY_H\|_{H^{b-1}}
 \Big(
 \|\nabla\p_t\bY_H\|_{H^b}
 +\|\Db\bY_H\|_{H^b}
 \Big)\,\rmd t\\
 &\quad\le
 \eta\int_0^T
 \Big(
 \|\nabla\p_t\bY_H\|_{H^b}^2
 +\|\Db\bY_H\|_{H^b}^2
 \Big)\,\rmd t
 +C_\eta\mathcal H_{b-1}(T).
\end{align*}

We next estimate the magnetic coefficient term $K_2$.
Using
$\|\partial_1b_0^1\|_{H^a}\le C\epsilon_0$,
the Sobolev product estimates yield
$$
\begin{aligned}
\|b_0^1(\partial_1b_0^1)\partial_1\bar Y_H\|_{H^b}
=\|(\partial_1b_0^1)\partial_{b_0^1}\bar Y_H\|_{H^b}
\le C\epsilon_0
\|\partial_{b_0^1}\bar Y_H\|_{H^b}.
\end{aligned}
$$
Using the preceding product estimate, the Cauchy--Schwarz
and horizontal Poincar\'e inequalities, followed by
Young's inequality, we obtain
$$
\begin{aligned}
&\sum_{|\alpha|\le b}\int_0^T
\left|
\Big(
\partial^\alpha\bigl(
b_0^1(\partial_1b_0^1)\partial_1\bar Y_H
\bigr),
\partial^\alpha\Big(\p_t\bar Y_H+\frac14\bar Y_H\Big)
\Big)_{L^2}
\right|\,\rmd t
\\
&\qquad\le
C\epsilon_0\int_0^T
\|\partial_{b_0^1}\bar Y_H\|_{H^b}
\Big(
\|\nabla\p_t\bar Y_H\|_{H^b}
+\|\partial_{b_0^1}\bar Y_H\|_{H^b}
\Big)\,\rmd t
\\
&\qquad\le C\epsilon_0\mathcal H_b(T).
\end{aligned}
$$

For $K_3$,
the case $\alpha=0$ has been treated in
\eqref{high-linear-pressure}.
For $1\le|\alpha|\le b$, the definition of $q$ and
integration by parts in $z_1$ give
$$
\begin{aligned}
\Big(
\partial^\alpha\nabla q,
\partial^\alpha\Big(\p_t\bar Y_H+\frac14\bar Y_H\Big)
\Big)_{L^2}
=
-2\Big(
\partial^\alpha\nabla\Delta^{-1}
(\gamma'\partial_{b_0^1}\bar Y_H^2)_H,
\partial_1\partial^\alpha
\Big(\p_t\bar Y_H+\frac14\bar Y_H\Big)
\Big)_{L^2}.
\end{aligned}
$$
Since $\nabla\Delta^{-1}$ gains one derivative on
nonzero horizontal frequencies,
$$
\begin{aligned}
\big\|
\partial^\alpha\nabla\Delta^{-1}
(\gamma'\partial_{b_0^1}\bar Y_H^2)_H
\big\|_{L^2}
\le
C\|\gamma'\partial_{b_0^1}\bar Y_H^2\|_{H^{b-1}}
\le C\|\partial_{b_0^1}\bar Y_H\|_{H^{b-1}}.
\end{aligned}
$$
Thus Young's inequality implies
$$
\begin{aligned}
&\sum_{1\le|\alpha|\le b}\int_0^T
\left|
\Big(
\partial^\alpha\nabla q,
\partial^\alpha\Big(\p_t\bar Y_H+\frac14\bar Y_H\Big)
\Big)_{L^2}
\right|\,\rmd t
\\
&\qquad\le
C\int_0^T
\|\partial_{b_0^1}\bar Y_H\|_{H^{b-1}}
\Big(
\|\nabla\p_t\bar Y_H\|_{H^b}
+\|\partial_{b_0^1}\bar Y_H\|_{H^b}
\Big)\,\rmd t
\\
&\qquad\le
\eta\int_0^T
\Big(
\|\nabla\p_t\bar Y_H\|_{H^b}^2
+\|\partial_{b_0^1}\bar Y_H\|_{H^b}^2
\Big)\,\rmd t
+C_\eta\mathcal H_{b-1}(T).
\end{aligned}
$$
Together with \eqref{high-linear-pressure}, this estimates
the full pressure contribution at order $b$.

The source term $K_4$ is kept in the form
$$
\sup_{0\le t\le T}\sum_{|\alpha|\le b}
\left|\int_0^t
\Big(\partial^\alpha\mathcal F,
\partial^\alpha(\p_\tau\bar Y_H+\tfrac14\bar Y_H)\Big)_{L^2}
\rmd\tau\right|.
$$
Integrating
\eqref{high-order-energy} over $[0,t]$, summing over $|\alpha|\le b$,
and choosing $\eta$ sufficiently small, we absorb the two
dissipation terms arising from $K_1$ and $K_3$.
Taking the supremum over $0\le t\le T$ and then taking
$\epsilon_0$ sufficiently small, we obtain
\refstepcounter{equation}\label{high-order-recursion}
$$
\begin{aligned}
\mathcal H_b(T)
\le{}& C\mathcal H_a(0)+C\mathcal H_{b-1}(T)
+C\int_0^T\|\Db\bY_H\|_{L^2}
\big(\|\mathcal R_H^1\|_{L^2}+\|\mathcal R_H^2\|_{L^2}\big)\dd t\\
&+C\sup_{0\le t\le T}\sum_{|\alpha|\le b}
\left|\int_0^t
\Big(\pa^\alpha\mathcal F,
\pa^\alpha(\p_\tau\bY_H+\tfrac14\bY_H)\Big)_{L^2}\dd\tau\right|.
\end{aligned}
\eqno(\theequation)
$$
Here the instantaneous energy and the cumulative dissipation
are estimated separately and then added. The constants can be
chosen uniformly for the finitely many orders $0\le b\le a$.

\medskip
\noindent
\textbf{Step 3: Iteration in the derivative order.}

Choose one constant $C\ge1$ for which \eqref{high-zero-bound}
and \eqref{high-order-recursion} hold. Since the sum of the
absolute values of the time integrals at order $b-1$ is bounded
by the corresponding sum at order $b$, induction gives, for $0\le b\le a$,
$$
\begin{aligned}
\mathcal H_b(T)
\le{}&\Big(\sum_{j=1}^{b+1}C^j\Big)
\Bigg\{\mathcal H_a(0)
+\int_0^T\|\Db\bY_H\|_{L^2}
\big(\|\mathcal R_H^1\|_{L^2}+\|\mathcal R_H^2\|_{L^2}\big)\dd t\\
&\quad+\sup_{0\le t\le T}\sum_{|\alpha|\le b}
\left|\int_0^t
\Big(\pa^\alpha\mathcal F,
\pa^\alpha(\p_\tau\bY_H+\tfrac14\bY_H)\Big)_{L^2}\dd\tau\right|
\Bigg\}.
\end{aligned}
$$
Indeed, the case $b=0$ follows from \eqref{high-zero-bound};
substitution of the estimate at order $b-1$ into
\eqref{high-order-recursion} gives the estimate at order $b$.
Taking $b=a$ proves \eqref{high-linear-bound}.
\end{proof}

\section{Estimate of the high-frequency component: the nonlinear estimate}
\label{high-nonlinear-section}
In this section, we estimate the nonlinear terms on the right-hand
side of \eqref{high-linear-bound}. We first treat the viscous terms,
then the pressure remainder, and finally the remainders
$\mathcal R_H^1$ and $\mathcal R_H^2$. Combining these estimates
with Proposition~\ref{high-linear-prop}, we obtain the following result.

\begin{proposition}\label{high-prop}
Assume that \eqref{bootstrap} holds on $[0,T]$.
If $\delta$ is sufficiently small, then
\begin{equation}\label{high-bound}
 \HH(T)\le C\HH(0)+C(\epsilon_0+\sqrt\delta)\HH(T).
\end{equation}
\end{proposition}
We prove Proposition~\ref{high-prop} at the end of this section,
after establishing the estimates in the following three subsections.

\subsection{Estimate of the viscous terms}
We begin by recording two auxiliary estimates for the viscous
energy calculations below.
The first controls the high-frequency component through the
dissipation in $\mathcal H_a(T)$.
The second combines the unweighted and time-weighted energies
to obtain $L^2_t(L^\infty)$ bounds for $\na\bY_H$ and
$L^1_t(L^\infty)$ bounds for $\na\p_t\bY_H$.
These bounds will be used to estimate the variable-coefficient
viscous terms and the associated commutators, including the
interactions with the horizontal mean.

By the horizontal Poincar\'e inequality and the fixed
coefficient bounds, we have
\begin{equation}\label{diss-lower}
 \|\bY_H\|_{L^2_t(H^a)}
 +\|\nabla\bY_H\|_{L^2_t(H^{a-1})}
 +\|\pa_1\bY_H\|_{L^2_t(H^a)}
 +\|\nabla\p_t\bY_H\|_{L^2_t(H^a)}
 \le C\HH(T)^{1/2}.
\end{equation}
Let $w(t)=(4a+t)^a$.
For $a\ge3$, the Sobolev interpolation inequality gives
$$
 \|F\|_{L^\infty}\le C\|F\|_{L^2}^{1-1/a}\|F\|_{H^a}^{1/a}.
$$
Consequently,
\begin{align*}
 \|\nabla\bY_H\|_{L^2_t(L^\infty)}^2
 &\le C\sup_{t\le T}\bigl(w\|\nabla\bY_H\|_{L^2}^2\bigr)^{1-1/a}
          \sup_{t\le T}\|\nabla\bY_H\|_{H^a}^{2/a}
          \int_0^T w^{-1+1/a}\dd t,\\
 \|\nabla\p_t\bY_H\|_{L^1_t(L^\infty)}
 &\le C\int_0^T
    (\sqrt w\|\nabla\p_t\bY_H\|_{L^2})^{1-1/a}
    \|\nabla\p_t\bY_H\|_{H^a}^{1/a}
    w^{-(a-1)/(2a)}\dd t\\
 &\le C\|\sqrt w\nabla\p_t\bY_H\|_{L^2_t(L^2)}^{1-1/a}
    \|\nabla\p_t\bY_H\|_{L^2_t(H^a)}^{1/a}
    \Big(\int_0^T w^{-1+1/a}\dd t\Big)^{1/2}.
\end{align*}
Since
$$
 \int_0^T w^{-1+1/a}\dd t
 =\int_0^T(4a+t)^{-(a-1)}\dd t
 \le\frac{(4a)^{2-a}}{a-2},
$$
we obtain
\begin{equation}\label{GN-time}
 \|\nabla\bY_H\|_{L^2_t(L^\infty)}
 +\|\nabla\p_t\bY_H\|_{L^1_t(L^\infty)}
 \le C\WW(T)^{(a-1)/(2a)}\HH(T)^{1/(2a)}.
\end{equation}

With these auxiliary estimates established, we turn to the
viscous terms.
Recall that the physical Laplacian in the fixed coordinates is
$$
 \Delta_XF=\rme^h\diver(\rme^{-h}G\nabla F)
 =\diver(G\nabla F)-G\nabla h\cdot\nabla F.
$$

\begin{lemma}\label{viscosity-lemma}
Assume that \eqref{bootstrap} holds on $[0,T]$.
If $\delta$ is sufficiently small, then
\begin{equation}\label{viscosity-bound}
\begin{aligned}
&\sup_{0\le t\le T}\sum_{|\alpha|\le a}
\Big|
\int_0^t
\Big(
\partial^\alpha\bigl((\Delta_X-\Delta)\bar Y_t\bigr),
\partial^\alpha\Big(\p_t\bar Y_H+\frac14\bar Y_H\Big)
\Big)_{L^2}\,d\tau
\Big|
\le
C\mathcal H_a(0)
+C\bigl(\epsilon_0+\sqrt{\delta}\bigr)\mathcal H_a(T).
\end{aligned}
\end{equation}
\end{lemma}
\begin{proof}
We first write
\begin{align*}
&\sup_{t\le T}\sum_{|\alpha|\le a}
 \left|\int_0^t
 \big(\pa^\alpha((\Delta_X-\Delta)\bY_t),
       \pa^\alpha(\p_t\bY_H+\tfrac14\bY_H)\big)_{L^2}\dd\tau\right|
 \notag\\
&\le \sup_{t\le T}\sum_{|\alpha|\le a}\left|\int_0^t
 (\pa^\alpha\diver((G-I)\nabla\p_t\bY_H), \pa^\alpha(\p_t\bY_H+\tfrac14\bY_H))_{L^2}\dd\tau\right|\notag\\
&\quad+\sup_{t\le T}\sum_{|\alpha|\le a}
\left| \int_0^t(\pa^\alpha(G\nabla h\cdot\nabla\p_t\bY_H),\pa^\alpha(\p_t\bY_H+\tfrac14\bY_H) )_{L^2}\dd\tau\right| \notag\\
&\quad+ \sup_{t\le T}\sum_{|\alpha|\le a}
 \left|\int_0^t
 \big(\pa^\alpha \Delta_X \p_t\bY_L,
\pa^\alpha(\p_t\bY_H+\tfrac14\bY_H)\big)_{L^2}\dd\tau\right|\notag\\
       &:=I_1+I_2+I_3.
\end{align*}

\textbf{Estimate of $I_1$}.
Firstly, by integration by parts, we write
\begin{align*}
 &I_{11}:=\sum_{|\alpha|\le a}\left|\int_0^t
 (\pa^\alpha((G-I)\nabla\p_t\bY_H),
                         \nabla\pa^\alpha\p_t\bY_H)_{L^2}\dd\tau\right|\\
 &\quad\le C\|G-I\|_{L^\infty_t(H^a)}
                 \|\nabla\p_t\bY_H\|_{L^2_t(H^a)}^2\\
 &\quad\le C(\epsilon_0+\sqrt\delta)\HH(T).
\end{align*}

Next, we estimate
\begin{align*}
 &I_{12}:=\sum_{|\alpha|\le a}\left|\int_0^t
 (\pa^\alpha\diver((G-I)\nabla\p_t\bY_H), \pa^\alpha\bY_H)_{L^2}\dd\tau\right|\,.
\end{align*}
Integration by parts yields
\begin{align*}
 &(\pa^\alpha\diver((G-I)\nabla\p_t\bY_H),\pa^\alpha\bY_H)_{L^2}\\
 &=-\frac12\pa_t
       ((G-I)\nabla\pa^\alpha\bY_H,\nabla\pa^\alpha\bY_H)_{L^2}
   +\frac12(G_t\nabla\pa^\alpha\bY_H,\nabla\pa^\alpha\bY_H)_{L^2}\\
 &\quad-([\pa^\alpha,G]\nabla\p_t\bY_H,\nabla\pa^\alpha\bY_H)_{L^2}.
\end{align*}
We first estimate the time boundary term arising from the preceding
integration by parts. By \eqref{geometry-small}, for $|\alpha|\le a$,
$$
\begin{aligned}
&\sup_{0\le t\le T}
\left|
\Big[
\bigl((G-I)\nabla\partial^\alpha\bar Y_H,
                 \nabla\partial^\alpha\bar Y_H\bigr)_{L^2}
\Big]_0^t
\right|
\le
C(\epsilon_0+\sqrt{\delta})
\bigl(\mathcal H_a(0)+\mathcal H_a(T)\bigr).
\end{aligned}
$$
It remains to estimate the terms involving $G_t$ and the commutator.
For the commutator, the exact identity is
\begin{equation*}
 [\pa^\alpha,G]\nabla\p_t\bY_H
 =\pa_t([\pa^\alpha,G]\nabla\bY_H)
                         -[\pa^\alpha,G_t]\nabla\bY_H.
\end{equation*}
Thus
\begin{align*}
\int_0^t([\pa^\alpha,G]\nabla\p_t\bY_H,\nabla\pa^\alpha\bY_H)_{L^2}\dd\tau
&=\Big[([\pa^\alpha,G]\nabla\bY_H,\nabla\pa^\alpha\bY_H)_{L^2}\Big]_0^t-\int_0^t([\pa^\alpha,G]\nabla\bY_H,\nabla\pa^\alpha\p_t\bY_H)_{L^2}\dd\tau\\
 &\quad-\int_0^t([\pa^\alpha,G_t]\nabla\bY_H,\nabla\pa^\alpha\bY_H)_{L^2}\dd\tau.
\end{align*}
For $\alpha=0$, the commutator vanishes. We therefore consider
$1\le|\alpha|\le a$.
Since $[\pa^\alpha,G]=[\pa^\alpha,G-I]$, Leibniz's formula
and the algebra property of $H^{a-1}$ give
$$
\begin{aligned}
\|[\pa^\alpha,G]\nabla\bY_H\|_{L^2}
&\le
C\sum_{0<\beta\le\alpha}
\|\pa^\beta(G-I)\,
  \pa^{\alpha-\beta}\nabla\bY_H\|_{L^2}
\\
&\le
C\|\nabla G\|_{H^{a-1}}
 \|\nabla\bY_H\|_{H^{a-1}}
\\
&\le
C\|G-I\|_{H^a}\|\nabla\bY_H\|_{H^{a-1}}.
\end{aligned}
$$
Similarly,
$$
\|[\pa^\alpha,G_t]\nabla\bY_H\|_{L^2}
\le
C\|G_t\|_{H^a}\|\nabla\bY_H\|_{H^{a-1}}.
$$
Using
$$
[\pa^\alpha,G]\nabla\p_t\bY_H
=
\p_t\bigl([\pa^\alpha,G]\nabla\bY_H\bigr)
-
[\pa^\alpha,G_t]\nabla\bY_H,
$$
and integrating by parts in time, we obtain
$$
\begin{aligned}
&\int_0^t
\bigl(
[\pa^\alpha,G]\nabla\p_t\bY_H,
\nabla\pa^\alpha\bY_H
\bigr)_{L^2}\dd\tau
=
\Big[
\bigl(
[\pa^\alpha,G]\nabla\bY_H,
\nabla\pa^\alpha\bY_H
\bigr)_{L^2}
\Big]_0^t\\
&\quad
-\int_0^t
\bigl(
[\pa^\alpha,G]\nabla\bY_H,
\nabla\pa^\alpha\p_t\bY_H
\bigr)_{L^2}\dd\tau
-\int_0^t
\bigl(
[\pa^\alpha,G_t]\nabla\bY_H,
\nabla\pa^\alpha\bY_H
\bigr)_{L^2}\dd\tau.
\end{aligned}
$$

For the time boundary term, the preceding commutator estimate
and the definition of $\mathcal H_a$ imply
$$
\begin{aligned}
&\left|
\Big[
\bigl(
[\pa^\alpha,G]\nabla\bY_H,
\nabla\pa^\alpha\bY_H
\bigr)_{L^2}
\Big]_0^t
\right|
\\
&\quad\le
C\|G(t)-I\|_{H^a}
 \|\nabla\bY_H(t)\|_{H^{a-1}}
 \|\nabla\bY_H(t)\|_{H^a}
\\
&\qquad
+C\|G(0)-I\|_{H^a}
 \|\nabla\bY_H(0)\|_{H^{a-1}}
 \|\nabla\bY_H(0)\|_{H^a}
\\
&\quad\le
C(\epsilon_0+\sqrt{\delta})
\bigl(\mathcal H_a(0)+\mathcal H_a(T)\bigr),
\end{aligned}
$$
where we used \eqref{geometry-small} in the last inequality.

For the first integral, H\"older's inequality in time and
\eqref{diss-lower} yield
$$
\begin{aligned}
&\left|
\int_0^t
\bigl(
[\pa^\alpha,G]\nabla\bY_H,
\nabla\pa^\alpha\p_t\bY_H
\bigr)_{L^2}\dd\tau
\right|
\\
&\quad\le
C\int_0^t
\|G-I\|_{H^a}
\|\nabla\bY_H\|_{H^{a-1}}
\|\nabla\p_t\bY_H\|_{H^a}\dd\tau
\\
&\quad\le
C\|G-I\|_{L^\infty_t(H^a)}
\left(
\int_0^T\|\nabla\bY_H\|_{H^{a-1}}^2\dd\tau
\right)^{1/2}
\left(
\int_0^T\|\nabla\p_t\bY_H\|_{H^a}^2\dd\tau
\right)^{1/2}
\\
&\quad\le
C(\epsilon_0+\sqrt{\delta})\mathcal H_a(T).
\end{aligned}
$$

For the second integral, we similarly have
$$
\begin{aligned}
\left|
\int_0^t
\bigl(
[\pa^\alpha,G_t]\nabla\bY_H,
\nabla\pa^\alpha\bY_H
\bigr)_{L^2}\dd\tau
\right|
&\le
C\int_0^t
\|G_t\|_{H^a}
\|\nabla\bY_H\|_{H^{a-1}}
\|\nabla\bY_H\|_{H^a}\dd\tau
\\
&\le
C\bigl(
\mathcal H_a(T)^{1/2}
+\mathcal L_a(T)^{1/2}
\bigr)\mathcal H_a(T)
\end{aligned}
$$
where we used \eqref{Gt-est}, \eqref{diss-lower},
and \eqref{bootstrap} in the last two inequalities.

Combining these estimates, we conclude that, uniformly for
$0\le t\le T$,
$$
\begin{aligned}
\left|
\int_0^t
\bigl(
[\pa^\alpha,G]\nabla\p_t\bY_H,
\nabla\pa^\alpha\bY_H
\bigr)_{L^2}\dd\tau
\right|\le
C(\epsilon_0+\sqrt{\delta})
\bigl(\mathcal H_a(0)+\mathcal H_a(T)\bigr).
\end{aligned}
$$


Now we turn to the estimate of $\int_0^t |(G_t\nabla\pa^\alpha\bY_H,\nabla\pa^\alpha\bY_H)_{L^2}|\dd\tau$.
Differentiating \eqref{matrix-G}, for every scalar \(F\), we obtain
\begin{align*}
 \nabla F\cdot G_t\nabla F
 =2\rme^{2h}\Big\{&
 \big[(\rho+\pa_2\bY^2)\pa_2\bY_t^2
       +(\pa_2\tY^1+\pa_2\bY^1)\pa_2\bY_t^1\big](\pa_1F)^2\\
 &-\big[\pa_1\p_t\bY_H^2(\rho+\pa_2\bY^2)
         +\pa_1\bY_H^2\pa_2\bY_t^2\\
 &\qquad+\pa_1\p_t\bY_H^1(\pa_2\tY^1+\pa_2\bY^1)
         +(\rme^{-h}/\rho+\pa_1\bY_H^1)\pa_2\bY_t^1\big]
                                  \pa_1F\pa_2F\\
 &+\big[\pa_1\bY_H^2\pa_1\p_t\bY_H^2
         +(\rme^{-h}/\rho+\pa_1\bY_H^1)\pa_1\p_t\bY_H^1\big]
                                  (\pa_2F)^2\Big\}.
\end{align*}
Since \(\pa_1\p_t\bY_L=0\), this identity gives
\begin{equation}\label{Gt-low}
 |\nabla F\cdot G_t\nabla F|
 \le C|\nabla\p_t\bY_H|\,|\nabla F|^2
     +C|\pa_2\p_t\bY_L|\,|\pa_1F|\,|\nabla F|.
\end{equation}
Taking \(F=\pa^\alpha\bY_H^i\), summing, and using \eqref{GN-time},
\begin{align*}
 &\sum_{|\alpha|\le a}\int_0^t
 |(G_t\nabla\pa^\alpha\bY_H,\nabla\pa^\alpha\bY_H)_{L^2}|\dd\tau\\
 &\le C\|\nabla\p_t\bY_H\|_{L^1_t(L^\infty)}
\|\nabla\bY_H\|_{L^\infty_t(H^a)}^2+C\|\pa_2\p_t\bY_L\|_{L^2_t(L^\infty)}
             \|\pa_1\bY_H\|_{L^2_t(H^a)}
             \|\nabla\bY_H\|_{L^\infty_t(H^a)}\\
 &\le C\big(\WW(T)^{(a-1)/(2a)}\HH(T)^{1/(2a)}
                         +\LL(T)^{1/2}\big)\HH(T)\\
 &\le C\sqrt\delta\,\HH(T).
\end{align*}
Combining the estimates for the boundary term, the commutator,
and the term involving $G_t$, we obtain
\begin{align*}
I_{12}\le
C(\epsilon_0+\sqrt{\delta})
\bigl(\mathcal H_a(0)+\mathcal H_a(T)\bigr).
\end{align*}
Together with the preceding estimate for the pairing with
$\partial^\alpha\p_t\bar Y_H$, this yields
\begin{align*}
I_1&\le
C(\epsilon_0+\sqrt{\delta})
\bigl(\mathcal H_a(0)+\mathcal H_a(T)\bigr)
\le
C\mathcal H_a(0)
+C(\epsilon_0+\sqrt{\delta})\mathcal H_a(T).
\end{align*}

\textbf{Estimate of $I_2$}.
For $1\le|\alpha|\le a$, choose a multi-index $e\le\alpha$
with $|e|=1$. Since $h$ is independent of time, integration
by parts in space and then in time yields
$$
\begin{aligned}
I_{21}
&:= \int_0^t
\Big(
\partial^\alpha(G\nabla h\cdot\nabla\p_t\bar Y_H),
\partial^\alpha\bar Y_H
\Big)_{L^2}\,d\tau
\\
&= -\Big[
\Big(
\partial^{\alpha-e}(G\nabla h\cdot\nabla\bar Y_H),
\partial^{\alpha+e}\bar Y_H
\Big)_{L^2}
\Big]_0^t
+\int_0^t
\Big(
\partial^{\alpha-e}(G\nabla h\cdot\nabla\bar Y_H),
\partial^{\alpha+e}\p_t\bar Y_H
\Big)_{L^2}\,d\tau
\\
&\quad
+\int_0^t
\Big(
\partial^{\alpha-e}(G_t\nabla h\cdot\nabla\bar Y_H),
\partial^{\alpha+e}\bar Y_H
\Big)_{L^2}\,d\tau.
\end{aligned}
$$
By \eqref{static-est3}, \eqref{geometry-small}, \eqref{Gt-est},
and the algebra property of $H^{a-1}$, we have
$$
\|G\nabla h\|_{L^\infty_t(H^{a-1})}
\le C\epsilon_0,
\qquad
\|G_t\nabla h\|_{L^2_t(H^{a-1})}
\le C\epsilon_0
\bigl(\mathcal H_a(T)^{1/2}+\mathcal L_a(T)^{1/2}\bigr),
$$
Using these bounds, we obtain, uniformly for $0\le t\le T$,
$$
\begin{aligned}
|I_{21}|
&\le
C\epsilon_0\bigl(\mathcal H_a(0)+\mathcal H_a(T)\bigr)
+C\epsilon_0
\|\nabla\bar Y_H\|_{L^2_t(H^{a-1})}
\|\nabla\p_t\bar Y_H\|_{L^2_t(H^a)}
\\
&\quad
+C\|G_t\nabla h\|_{L^2_t(H^{a-1})}
\|\nabla\bar Y_H\|_{L^2_t(H^{a-1})}
\|\nabla\bar Y_H\|_{L^\infty_t(H^a)}
\\
&\le
C\epsilon_0\bigl(\mathcal H_a(0)+\mathcal H_a(T)\bigr)
+C\epsilon_0
\bigl(\mathcal H_a(T)^{1/2}+\mathcal L_a(T)^{1/2}\bigr)
\mathcal H_a(T)
\\
&\le
C\epsilon_0\bigl(\mathcal H_a(0)+\mathcal H_a(T)\bigr)
+C\epsilon_0\sqrt{\delta}\,\mathcal H_a(T),
\end{aligned}
$$
where we used \eqref{diss-lower} and the bootstrap assumption
\eqref{bootstrap} in the last two inequalities.

For $\alpha=0$, H\"older's inequality and the horizontal
Poincar\'e inequality yield
$$
\begin{aligned}
&\left|
\int_0^t
\Big(G\nabla h\cdot\nabla\p_t\bar Y_H,\bar Y_H\Big)_{L^2}
\,d\tau
\right|\le
C\epsilon_0
\|\nabla\p_t\bar Y_H\|_{L^2_t(L^2)}
\|\bar Y_H\|_{L^2_t(L^2)}
\le C\epsilon_0\mathcal H_a(T).
\end{aligned}
$$
Combining this with the estimates for $1\le|\alpha|\le a$
and summing over $\alpha$, we obtain
$$
\sum_{|\alpha|\le a}\sup_{0\le t\le T}|I_{21}|
\le
C\epsilon_0\bigl(\mathcal H_a(0)+\mathcal H_a(T)\bigr)
+C\epsilon_0\sqrt{\delta}\,\mathcal H_a(T).
$$

We next estimate $I_{22}:=\int_0^t
\Big(
\partial^\alpha(G\nabla h\cdot\nabla\p_t\bar Y_H),
\partial^\alpha\p_t\bar Y_H
\Big)_{L^2}\,d\tau$. For $1\le|\alpha|\le a$,
choose $e\le\alpha$ with $|e|=1$. Spatial integration
by parts gives
$$
\begin{aligned}
&\int_0^t
\Big(
\partial^\alpha(G\nabla h\cdot\nabla\p_t\bar Y_H),
\partial^\alpha\p_t\bar Y_H
\Big)_{L^2}\,d\tau
=
-\int_0^t
\Big(
\partial^{\alpha-e}(G\nabla h\cdot\nabla\p_t\bar Y_H),
\partial^{\alpha+e}\p_t\bar Y_H
\Big)_{L^2}\,d\tau.
\end{aligned}
$$
Using the algebra property of $H^{a-1}$, we bound
the absolute value of this integral by
$$
C\|G\nabla h\|_{L^\infty_t(H^{a-1})}
\|\nabla\p_t\bar Y_H\|_{L^2_t(H^{a-1})}
\|\nabla\p_t\bar Y_H\|_{L^2_t(H^a)}.
$$
For $\alpha=0$, the horizontal Poincar\'e inequality gives
$$
\begin{aligned}
&\left|
\int_0^t
\Big(G\nabla h\cdot\nabla\p_t\bar Y_H,\p_t\bar Y_H\Big)_{L^2}
\,d\tau
\right|\le
C\epsilon_0
\|\nabla\p_t\bar Y_H\|_{L^2_t(L^2)}
\|\p_t\bar Y_H\|_{L^2_t(L^2)}
\le C\epsilon_0\mathcal H_a(T).
\end{aligned}
$$
Consequently, summing over $|\alpha|\le a$, we obtain
$$
\begin{aligned}
\sum_{|\alpha|\le a}\sup_{0\le t\le T}|I_{22}|
&\le
C\|G\nabla h\|_{L^\infty_t(H^{a-1})}
\|\nabla\p_t\bar Y_H\|_{L^2_t(H^{a-1})}
\|\nabla\p_t\bar Y_H\|_{L^2_t(H^a)}
\\
&\le C\epsilon_0\mathcal H_a(T).
\end{aligned}
$$
Combining the estimates for $I_{21}$ and $I_{22}$,
we conclude that, uniformly for $0\le t\le T$,
$$
\begin{aligned}
I_2&\le
C\epsilon_0\bigl(\mathcal H_a(0)+\mathcal H_a(T)\bigr)
+C\epsilon_0\sqrt{\delta}\,\mathcal H_a(T)
\\
&\le
C\epsilon_0\bigl(\mathcal H_a(0)+\mathcal H_a(T)\bigr),
\end{aligned}
$$
where the last inequality uses $\delta\le1$.

\textbf{Estimate of $I_3$}.
First, noting \(\pa_1\p_t\bY_L=0\), we calculate
\begin{align*}
\Delta_X \p_t\bY_L
=\rme^h\diver(\rme^{-h}G\nabla \p_t\bY_L)
=\rme^h \Big[ \p_1(\rme^{-h}G_{12}\p_2 \p_t\bY_L)+\p_2(\rme^{-h}G_{22}\p_2 \p_t\bY_L) \Big].
\end{align*}
To estimate the contribution of $\Delta_X\p_t\bar Y_L$,
we first separate out the terms that contains no time decaying factors that involve a purely horizontal-mean term of $\p_2 \bY^1$ and  $\tY$, $\rho$.
By the explicit formulas for $G$, we write
\begin{align}
 G_{12}+\frac{\rme^h}{\rho}(\pa_2\tY^1+\pa_2\bY_L^1)
 =&-\rme^{2h}\big[\pa_1\bY_H^2(\rho+\pa_2\bY^2)
 +\pa_1\bY_H^1(\pa_2\tY^1+\pa_2\bY_L^1)\notag\\
 &\qquad\quad+(\rme^{-h}/\rho+\pa_1\bY_H^1)\pa_2\bY_H^1\big],
 \label{column1}\\
 G_{22}-\rho^{-2}
 =&\frac{2\rme^h}{\rho}\pa_1\bY_H^1 +\rme^{2h}|\pa_1\bY_H|^2.
 \label{column2}
\end{align}
Every term on the right hand sides of the above equalities contains a derivative of $\bar Y_H$.

For the left hand side of \eqref{column1} and \eqref{column2}. We calculate their  combined contribution to
$\Delta_X\p_t\bar Y_L$.
Indeed, using
$\partial_1\p_t\bar Y_L=0$,
$\partial_1\partial_2\bar Y_L^1=0$, and
$\partial_1\partial_2\widetilde Y^1
=\partial_2(\rme^{-h}/\rho)$,
 we obtain
\begin{equation*}
\begin{aligned}
 &\rme^h\left\{\pa_1\Big[-\frac{\pa_2\tY^1+\pa_2\bY_L^1}{\rho}
                                     \pa_2\p_t\bY_L\Big]
       +\pa_2\Big[\frac{\rme^{-h}}{\rho^2}\pa_2\p_t\bY_L\Big]\right\}\\
 &=\rme^h\left\{-\frac1\rho\pa_2\Big(\frac{\rme^{-h}}\rho\Big)
                                      \pa_2\p_t\bY_L
       +\pa_2\Big(\frac{\rme^{-h}}{\rho^2}\pa_2\p_t\bY_L\Big)\right\}\\
 &=\rho^{-1}\pa_2(\rho^{-1}\pa_2\p_t\bY_L).
\end{aligned}
\end{equation*}
This expression is independent of \(z_1\), although the two extracted coefficients need not be independent
of $z_1$.  This expression vanishes when paired with the high-frequency test function, from the
remaining terms, whose coefficients contain derivatives of $\bar Y_H$.

Hence,
\begin{align*}
 \Delta_X\p_t\bY_L
 ={}&\rho^{-1}\pa_2(\rho^{-1}\pa_2\p_t\bY_L)\\
 &+\pa_1\Big[\Big(G_{12}+\frac{\rme^h}{\rho}
          (\pa_2\tY^1+\pa_2\bY_L^1)\Big)\pa_2\p_t\bY_L\Big]
       +\pa_2\big[(G_{22}-\rho^{-2})\pa_2\p_t\bY_L\big]\\
 &-(\pa_1h)\Big(G_{12}+\frac{\rme^h}{\rho}
          (\pa_2\tY^1+\pa_2\bY_L^1)\Big)\pa_2\p_t\bY_L
       -(\pa_2h)(G_{22}-\rho^{-2})\pa_2\p_t\bY_L.
\end{align*}
Thus
\begin{align}
I_3:&= \sup_{t\le T}\sum_{|\alpha|\le a}
 \left|\int_0^t
 \big(\pa^\alpha \Delta_X \p_t\bY_L,
       \pa^\alpha(\p_t\bY_H+\tfrac14\bY_H)\big)_{L^2}\dd\tau\right|\notag\\
&\leq \sup_{t\le T}\sum_{|\alpha|\le a}
 \Big|\int_0^t
 \big(\pa^\alpha \big( \pa_1\big[(G_{12}+\frac{\rme^h}{\rho}
          (\pa_2\tY^1+\pa_2\bY_L^1))\pa_2\p_t\bY_L\big]  \notag\\
&\qquad\qquad      +\pa_2\big[(G_{22}-\rho^{-2})\pa_2\p_t\bY_L\big]\big),
       \pa^\alpha(\p_t\bY_H+\tfrac14\bY_H)\big)_{L^2}\dd\tau\Big|\notag\\
 &\quad+  \sup_{t\le T}\sum_{|\alpha|\le a}
 \Big|\int_0^t
 \big(\pa^\alpha \big( (\pa_1h)(G_{12}+\frac{\rme^h}{\rho}
          (\pa_2\tY^1+\pa_2\bY_L^1))\pa_2\p_t\bY_L \notag\\
 &\qquad\qquad      +(\pa_2h)(G_{22}-\rho^{-2})\pa_2\p_t\bY_L \big),
       \pa^\alpha(\p_t\bY_H+\tfrac14\bY_H)\big)_{L^2}\dd\tau\Big|.
 \label{est-low-nonl}
\end{align}

For \(r=a,a-1\), the algebra estimates applied to
\eqref{column1}--\eqref{column2} yield
\begin{align*}
 \big\|G_{12}+\frac{\rme^h}{\rho}
                   (\pa_2\tY^1+\pa_2\bY_L^1)\big\|_{H^r}
 &\le C\|\nabla\bY_H\|_{H^r},\\
 \|G_{22}-\rho^{-2}\|_{H^r}
 &\le C\|\pa_1\bY_H\|_{H^r}.
\end{align*}
Hence
\begin{align*}
 &\big\|\Big(G_{12}+\frac{\rme^h}{\rho}
          (\pa_2\tY^1+\pa_2\bY_L^1)\Big)\pa_2\p_t\bY_L
       \big\|_{L^2_t(H^a)}\\
 &\qquad\le C\|\nabla\bY_H\|_{L^\infty_t(H^a)}
                    \|\pa_2\p_t\bY_L\|_{L^2_t(H^a)}
       \le C\HH(T)^{1/2}\LL(T)^{1/2},\\
 &\|(G_{22}-\rho^{-2})\pa_2\p_t\bY_L\|_{L^1_t(H^a)}\\
 &\qquad\le C\|\pa_1\bY_H\|_{L^2_t(H^a)}
                    \|\pa_2\p_t\bY_L\|_{L^2_t(H^a)}
       \le C\HH(T)^{1/2}\LL(T)^{1/2},\\
 &\|(G_{22}-\rho^{-2})\pa_2\p_t\bY_L\|_{L^2_t(H^a)}\\
 &\qquad\le C\|\pa_1\bY_H\|_{L^\infty_t(H^a)}
                    \|\pa_2\p_t\bY_L\|_{L^2_t(H^a)}
       \le C\HH(T)^{1/2}\LL(T)^{1/2}.
\end{align*}
Consequently, the second and the third line of \eqref{est-low-nonl} are bounded by
\begin{align*}
&C\big\|\big(G_{12}+\frac{\rme^h}{\rho}
          (\pa_2\tY^1+\pa_2\bY_L^1)\big)\pa_2\p_t\bY_L
       \big\|_{L^2_t(H^a)}  \big( \|\pa_1\bY_H\|_{L^2_t(H^a)} + \|\na\p_t\bY_H\|_{L^2_t(H^a)}\big)\\
&+C \|(G_{22}-\rho^{-2})\pa_2\p_t\bY_L\|_{L^1_t(H^a)} \|\pa_2\bY_H\|_{L^\infty_t(H^a)}\\
& +C \|(G_{22}-\rho^{-2})\pa_2\p_t\bY_L\|_{L^2_t(H^a)} \|\na \p_t\bY_H\|_{L^2_t(H^a)}\\
& \le C\LL(T)^{1/2}\HH(T).
\end{align*}

For the last two lines of \eqref{est-low-nonl}, a similar product estimate gives
\begin{align*}
 &\big\|(\pa_1h)\Big(G_{12}+\frac{\rme^h}{\rho}
          (\pa_2\tY^1+\pa_2\bY_L^1)\Big)\pa_2\p_t\bY_L
          \big\|_{L^1_t(H^{a-1})}+\|(\pa_2h)(G_{22}-\rho^{-2})\pa_2\p_t\bY_L
          \|_{L^1_t(H^{a-1})}\\
 &\qquad\le C\epsilon_0\|\nabla\bY_H\|_{L^2_t(H^{a-1})}
                              \|\pa_2\p_t\bY_L\|_{L^2_t(H^{a-1})}
 \le C\epsilon_0\HH(T)^{1/2}\LL(T)^{1/2},\\
 &\big\|(\pa_1h)\Big(G_{12}+\frac{\rme^h}{\rho}
          (\pa_2\tY^1+\pa_2\bY_L^1)\Big)\pa_2\p_t\bY_L
          \big\|_{L^2_t(H^{a-1})}+\|(\pa_2h)(G_{22}-\rho^{-2})\pa_2\p_t\bY_L
          \|_{L^2_t(H^{a-1})}\\
 &\qquad\le C\epsilon_0\|\nabla\bY_H\|_{L^\infty_t(H^{a-1})}
                              \|\pa_2\p_t\bY_L\|_{L^2_t(H^{a-1})}
 \le C\epsilon_0\HH(T)^{1/2}\LL(T)^{1/2}.
\end{align*}

For $1\le|\alpha|\le a$, choose a multi-index $e\le\alpha$
with $|e|=1$. Integrating by parts in space, we transfer
one derivative to the test function. For $\alpha=0$,
we use the horizontal Poincar\'e inequality.
Thus, H\"older's inequality in time and the preceding
estimates yield
$$
\begin{aligned}
&\sup_{0\le t\le T}\sum_{|\alpha|\le a}
\Bigg|
\int_0^t
\Bigg(
\pa^\alpha\Bigg[
(\pa_1h)
\Big(
G_{12}+\frac{\rme^h}{\rho}
(\pa_2\tY^1+\pa_2\bY_L^1)
\Big)\pa_2\p_t\bY_L
\\
&\qquad\qquad
+(\pa_2h)(G_{22}-\rho^{-2})\pa_2\p_t\bY_L
\Bigg],
\pa^\alpha\Big(\p_t\bY_H+\frac14\bY_H\Big)
\Bigg)_{L^2}\dd\tau
\Bigg|
\\
&\quad\le
C\Bigg(
\Big\|
(\pa_1h)
\Big(
G_{12}+\frac{\rme^h}{\rho}
(\pa_2\tY^1+\pa_2\bY_L^1)
\Big)\pa_2\p_t\bY_L
\Big\|_{L^1_t(H^{a-1})}
\\
&\qquad\qquad
+\big\|
(\pa_2h)(G_{22}-\rho^{-2})\pa_2\p_t\bY_L
\big\|_{L^1_t(H^{a-1})}
\Bigg)
\|\nabla\bY_H\|_{L^\infty_t(H^a)}
\\
&\qquad
+C\Bigg(
\Big\|
(\pa_1h)
\Big(
G_{12}+\frac{\rme^h}{\rho}
(\pa_2\tY^1+\pa_2\bY_L^1)
\Big)\pa_2\p_t\bY_L
\Big\|_{L^2_t(H^{a-1})}
\\
&\qquad\qquad
+\big\|
(\pa_2h)(G_{22}-\rho^{-2})\pa_2\p_t\bY_L
\big\|_{L^2_t(H^{a-1})}
\Bigg)
\|\nabla\p_t\bY_H\|_{L^2_t(H^a)}
\\
&\quad\le
C\epsilon_0
\|\nabla\bY_H\|_{L^2_t(H^{a-1})}
\|\pa_2\p_t\bY_L\|_{L^2_t(H^{a-1})}
\|\nabla\bY_H\|_{L^\infty_t(H^a)}
\\
&\qquad
+C\epsilon_0
\|\nabla\bY_H\|_{L^\infty_t(H^{a-1})}
\|\pa_2\p_t\bY_L\|_{L^2_t(H^{a-1})}
\|\nabla\p_t\bY_H\|_{L^2_t(H^a)}
\\
&\quad\le
C\epsilon_0
\mathcal H_a(T)^{1/2}\mathcal L_a(T)^{1/2}
\Big(
\|\nabla\bY_H\|_{L^\infty_t(H^a)}
+\|\nabla\p_t\bY_H\|_{L^2_t(H^a)}
\Big)
\\
&\quad\le
C\epsilon_0\mathcal L_a(T)^{1/2}\mathcal H_a(T),
\end{aligned}
$$
where we used \eqref{diss-lower} and the definitions of
$\mathcal H_a(T)$ and $\mathcal L_a(T)$ in the last two inequalities.

Combining this with the estimate for the divergence terms gives
$$
I_3
\le C(1+\epsilon_0)\mathcal L_a(T)^{1/2}\mathcal H_a(T)
\le C\sqrt{\delta}\,\mathcal H_a(T).
$$
Together with the bounds for $I_1$ and $I_2$, we conclude that
$$
I_1+I_2+I_3
\le C\mathcal H_a(0)
+C(\epsilon_0+\sqrt{\delta})\mathcal H_a(T),
$$
which proves \eqref{viscosity-bound}.
\end{proof}

\subsection{Estimate of the pressure term}\label{sec5.2}
The estimate of $\nabla(p-q)$ has been established in
Lemma~\ref{pressure-lemma}. We now apply it to the pressure
contribution in \eqref{high-linear-bound}.
For the pressure remainder, \eqref{pressure-remainder}
and the horizontal Poincar\'e inequality yield
$$
\begin{aligned}
&\sum_{|\alpha|\le a}\int_0^T
\left|
\Big(
\partial^\alpha(AB\nabla p-\nabla q),
\partial^\alpha\Big(\p_t\bar Y_H+\frac14\bar Y_H\Big)
\Big)_{L^2}
\right|\,\rmd t
\\
&\qquad\le
C\int_0^T
\|AB\nabla p-\nabla q\|_{H^a}
\Big(
\|\p_t\bar Y_H\|_{H^a}+\|\bar Y_H\|_{H^a}
\Big)\,\rmd t
\\
&\qquad\le
C(\epsilon_0+\sqrt{\delta})
\int_0^T
\Big(
\|\partial_{b_0^1}\bar Y_H\|_{H^a}
+\|\nabla\p_t\bar Y_H\|_{H^a}
\Big)^2\,\rmd t
\\
&\qquad\le
C(\epsilon_0+\sqrt{\delta})\mathcal H_a(T).
\end{aligned}
$$
Here we used  \eqref{equiva-1} and \eqref{equ-b0-norm}.

\subsection{Estimate of the remainders from the constraint}
\label{high-source-subsection}
It remains to estimate the integral containing $\mathcal R_H^1$
and $\mathcal R_H^2$ in \eqref{high-linear-bound}.
By the definitions in Proposition~\ref{high-linear-prop} and \eqref{constraint},
\begin{align*}
 \mathcal R_H^1
 &=(1-\rho)\pa_1\bY_H^1
   +\pa_1\bY_H^2\pa_2Y^1-\pa_1Y^1\pa_2\bY^2,\\
 \mathcal R_H^2
 &=(1-\rho)\pa_1\p_t\bY_H^1
   +\pa_1\p_t\bY_H^2\pa_2Y^1
   +\pa_1\bY_H^2\pa_2\bY_t^1\\
 &\quad-\pa_1\p_t\bY_H^1\pa_2\bY^2
   -\pa_1Y^1\pa_2\bY_t^2.
\end{align*}

Using \eqref{geometry-small} and \eqref{vertical-tangent}, we obtain
$$
 \|\mathcal R_H^1\|_{L^2}
 \le C(\epsilon_0+\sqrt\delta)\|\pa_1\bY_H\|_{L^2}.
$$
For $\mathcal R_H^2$, we use
$\bar Y_t=\p_t\bar Y_H+\p_t\bar Y_L$ to write
$$
\begin{aligned}
\mathcal R_H^2
&=(1-\rho)\partial_1\p_t\bar Y_H^1
+\partial_1\p_t\bar Y_H^2\partial_2Y^1
\\
&\quad
+\partial_1\bar Y_H^2\partial_2\p_t\bar Y_H^1
-\partial_1\p_t\bar Y_H^1\partial_2\bar Y^2
-\partial_1Y^1\partial_2\p_t\bar Y_H^2
\\
&\quad
+\partial_1\bar Y_H^2\partial_2\p_t\bar Y_L^1
-\partial_1Y^1\partial_2\p_t\bar Y_L^2.
\end{aligned}
$$
Thus, by H\"older's inequality,
$$
\begin{aligned}
\|\mathcal R_H^2\|_{L^2}
&\le
\|1-\rho\|_{L^\infty}
\|\partial_1\p_t\bar Y_H^1\|_{L^2}
+
\|\partial_2Y^1\|_{L^\infty}
\|\partial_1\p_t\bar Y_H^2\|_{L^2}
+
\|\partial_1\bar Y_H^2\|_{L^\infty}
\|\partial_2\p_t\bar Y_H^1\|_{L^2}
\\
&\quad+
\|\partial_2\bar Y^2\|_{L^\infty}
\|\partial_1\p_t\bar Y_H^1\|_{L^2}
+
\|\partial_1Y^1\|_{L^\infty}
\|\partial_2\p_t\bar Y_H^2\|_{L^2}
\\
&\quad
+
\|\partial_2\p_t\bar Y_L^1\|_{L^\infty}
\|\partial_1\bar Y_H^2\|_{L^2}
+
\|\partial_1Y^1\|_{L^\infty}
\|\partial_2\p_t\bar Y_L^2\|_{L^2}
\\
&\le
C(\epsilon_0+\sqrt{\delta})
\|\nabla\p_t\bar Y_H\|_{L^2}
+
C\|\partial_2\p_t\bar Y_L\|_{L^\infty}
\|\partial_1\bar Y_H\|_{L^2}
+
\|\partial_1Y^1\|_{L^\infty}
\|\partial_2\p_t\bar Y_L^2\|_{L^2}.
\end{aligned}
$$

To estimate the last term, we apply
\eqref{vertical-time-low} and the bootstrap assumption
\eqref{bootstrap} to obtain
$$
\begin{aligned}
\|\partial_2\p_t\bar Y_L^2\|_{L^2}
&\le
C\bigl(
\epsilon_0+\|\nabla\bar Y_H\|_{L^\infty}
\bigr)
\|\nabla\p_t\bar Y_H\|_{L^2}\le
C(\epsilon_0+\sqrt{\delta})
\|\nabla\p_t\bar Y_H\|_{L^2}.
\end{aligned}
$$
Consequently,
$$
\begin{aligned}
\|\partial_1Y^1\|_{L^\infty}
\|\partial_2\p_t\bar Y_L^2\|_{L^2}
&\le
C(\epsilon_0+\sqrt{\delta})
\|\nabla\p_t\bar Y_H\|_{L^2}.
\end{aligned}
$$
Combining the preceding estimates yields
$$
\|\mathcal R_H^2\|_{L^2}
\le
C(\epsilon_0+\sqrt{\delta})
\|\nabla\p_t\bar Y_H\|_{L^2}
+
C\|\partial_2\p_t\bar Y_L\|_{L^\infty}
\|\partial_1\bar Y_H\|_{L^2}.
$$
Consequently, by \eqref{equiva-1} and H\"older's inequality in time,
\begin{align}\label{high-remainder-bound}
\begin{split}
 &\int_0^T\|\Db\bY_H\|_{L^2}
       \big(\|\mathcal R_H^1\|_{L^2}+\|\mathcal R_H^2\|_{L^2}\big)\dd t\\
 &\quad\le C(\epsilon_0+\sqrt\delta)\mathcal H_a(T)
 +C\|\pa_2\p_t\bY_L\|_{L^2_t(L^\infty)}
    \|\Db\bY_H\|_{L^\infty_t(L^2)}
    \|\Db\bY_H\|_{L^2_t(L^2)}\\
 &\quad\le C\big(\epsilon_0+\sqrt\delta
                    +\mathcal L_a(T)^{1/2}\big)\mathcal H_a(T)
 \le C(\epsilon_0+\sqrt\delta)\mathcal H_a(T).
\end{split}
\end{align}
In particular, \eqref{high-linear-pressure} yields
\begin{equation*}
 \int_0^T\left|\Big(\nabla q, \p_t\bY_H+\tfrac14\bY_H\Big)_{L^2}\right|\dd t
 \le C(\epsilon_0+\sqrt\delta)\mathcal H_a(T).
\end{equation*}

\begin{proof}[{\bf{Proof of Proposition~\ref{high-prop}}}]
By the triangle inequality, the last term in
\eqref{high-linear-bound} is bounded by the sum of the viscous
and pressure contributions estimated above.
Combining \eqref{viscosity-bound}, the estimate in Section~\ref{sec5.2},
and \eqref{high-remainder-bound} with
Proposition~\ref{high-linear-prop}, we obtain
$$
 \mathcal H_a(T)\le C\mathcal H_a(0)
              +C(\epsilon_0+\sqrt\delta)\mathcal H_a(T),
$$
which proves \eqref{high-bound}.
\end{proof}

\section{The time-weighted estimate}\label{weight-sec}
\begin{proposition}\label{weighted-prop}
Assume that \eqref{bootstrap} holds on $[0,T]$.
If $\delta$ is sufficiently small, then
\begin{equation}\label{weighted-bound}
 \WW(T)\le C\big(\HH(0)+\HH(T)\big)
              +C(\epsilon_0+\sqrt\delta)\WW(T).
\end{equation}
\end{proposition}
\begin{proof}
Taking the $L^2$ inner product of $(4a+t)^a(\p_t\bar Y_H+\f14 \bar Y_H)$ with \eqref{high-projected-equation}, we derive that
\begin{align}\label{W0-7}
\begin{split}
&\f{\rmd}{\rmd t}\Big(\f12(4a+t)^a\|\p_t\bar Y_H\|_{L^2}^2+\f12(4a+t)^a\|\p_{b_0^1}\bar Y_H\|_{L^2}^2+\f18(4a+t)^a\|\na \bar Y_H\|_{L^2}^2\\
&\quad -\f{a}{8}(4a+t)^{a-1}\|\bar Y_H\|_{L^2}^2
+\frac14\big(\p_t\bar Y_H |(4a+t)^a \bar Y_H\big)_{L^2}\Big)\\
& +(4a+t)^a\|\na \p_t\bar Y_H\|_{L^2}^2+\frac14(4a+t)^a\|\p_{b_0^1}\bar Y_H\|_{L^2}^2
+\frac{a(a-1)}{8}(4a+t)^{a-2}\|\bar Y_H\|_{L^2}^2 \\
&=\frac14(4a+t)^a\|\p_t\bar Y_H\|_{L^2}^2+\frac{a}{2}(4a+t)^{a-1}\|\p_t\bar Y_H\|_{L^2}^2+\frac{a}{2}(4a+t)^{a-1}\|\p_{b_0^1}\bar Y_H\|_{L^2}^2\\
&\quad+\frac{a}{8}(4a+t)^{a-1}\|\na \bar Y_H\|_{L^2}^2	
-\bigl(\p_1b_0^1(y(z)) \p_{b_0^1}\bar Y_H |(4a+t)^a(\p_t\bar Y_H+\frac14\bar Y_H)\bigr)_{L^2}\\
&\quad-\bigl(\na q |(4a+t)^a(\p_t\bar Y_H+\frac14\bar Y_H)\bigr)_{L^2}
+\bigl(
(\Delta_X-\Delta)\bar Y_t,
(4a+t)^a(\p_t\bar Y_H+\frac14\bar Y_H)
\bigr)_{L^2}
\\
&\quad
-\bigl(
AB\nabla p-\nabla q,
(4a+t)^a(\p_t\bar Y_H+\frac14\bar Y_H)
\bigr)_{L^2}.	
\end{split}
\end{align}
Since $\bar Y_H$ has zero horizontal mean and
$a/(4a+t)\le1/4$, the horizontal Poincar\'e inequality and Young's
inequality show that the energy inside the time derivative
in \eqref{W0-7} is bounded from below by
$$
\frac14(4a+t)^a\|\p_t\bar Y_H\|_{L^2}^2
+\frac12(4a+t)^a\|\partial_{b_0^1}\bar Y_H\|_{L^2}^2
+\frac1{32}(4a+t)^a\|\nabla\bar Y_H\|_{L^2}^2.
$$
Moreover,
$$
\begin{aligned}
&\Big(\frac14(4a+t)^a+\frac a2(4a+t)^{a-1}\Big)
 \|\p_t\bar Y_H\|_{L^2}^2
+\frac a2(4a+t)^{a-1}
 \|\partial_{b_0^1}\bar Y_H\|_{L^2}^2
\\
&\qquad\le
\frac14(4a+t)^a\|\nabla\p_t\bar Y_H\|_{L^2}^2
+\frac18(4a+t)^a
 \|\partial_{b_0^1}\bar Y_H\|_{L^2}^2.
\end{aligned}
$$
For the remaining weighted gradient term, interpolation,
the horizontal Poincar\'e inequality,  \eqref{equiva-1} and \eqref{equ-b0-norm}
give
$$
\begin{aligned}
\frac a8(4a+t)^{a-1}\|\nabla\bar Y_H\|_{L^2}^2
&\le
C(4a+t)^{a-1}
\|\bar Y_H\|_{L^2}^{2(a-1)/a}
\|\bar Y_H\|_{\dot H^a}^{2/a}
\\
&\le
C\Big(
(4a+t)^a\|\partial_{b_0^1}\bar Y_H\|_{L^2}^2
\Big)^{(a-1)/a}
\|\partial_{b_0^1}\bar Y_H\|_{H^a}^{2/a}
\\
&\le
\frac1{16}(4a+t)^a
\|\partial_{b_0^1}\bar Y_H\|_{L^2}^2
+C\|\partial_{b_0^1}\bar Y_H\|_{H^a}^2.
\end{aligned}
$$
Substituting these estimates into \eqref{W0-7}, we obtain

\begin{align}\label{W0-7-1}
\begin{split}
&\frac{\rmd}{\rmd t}
\Bigl(
\frac12(4a+t)^a\|\p_t\bar Y_H\|_{L^2}^2
+\frac12(4a+t)^a\|\partial_{b_0^1}\bar Y_H\|_{L^2}^2
+\frac18(4a+t)^a\|\nabla\bar Y_H\|_{L^2}^2
\\
&\hspace{2cm}
-\frac a8(4a+t)^{a-1}\|\bar Y_H\|_{L^2}^2
+\frac14(4a+t)^a(\p_t\bar Y_H,\bar Y_H)
\Bigr)
\\
&\quad
+\frac34(4a+t)^a\|\nabla\p_t\bar Y_H\|_{L^2}^2
+\frac1{16}(4a+t)^a
 \|\partial_{b_0^1}\bar Y_H\|_{L^2}^2
+\frac{a(a-1)}8(4a+t)^{a-2}\|\bar Y_H\|_{L^2}^2
\\
&\le
C\|\partial_{b_0^1}\bar Y_H\|_{H^a}^2
-\Big(
(\partial_1b_0^1)\partial_{b_0^1}\bar Y_H,
(4a+t)^a\Big(\p_t\bar Y_H+\frac14\bar Y_H\Big)
\Big)
\\
&\quad
-\Big(
\nabla q,
(4a+t)^a\Big(\p_t\bar Y_H+\frac14\bar Y_H\Big)
\Big)
+\Big(
(\Delta_X-\Delta)\bar Y_t,
(4a+t)^a\Big(\p_t\bar Y_H+\frac14\bar Y_H\Big)
\Big)
\\
&\quad
-\Big(
AB\nabla p-\nabla q,
(4a+t)^a\Big(\p_t\bar Y_H+\frac14\bar Y_H\Big)
\Big).
\end{split}
\end{align}

The time integral of the first term on the right-hand side
is bounded by $C\mathcal H_a(T)$. We next estimate the
remaining four terms.
By the Poincar\'e inequality and the H\"older inequality, we obtain
\begin{align*}
\begin{split}
&\big|\bigl(\p_1b_0^1(y(z)) \p_{b_0^1}\bar Y_H |(4a+t)^a(\p_t\bar Y_H+\frac14\bar Y_H)\bigr)_{L^2}\big|\\
&\lesssim \eps_0\big((4a+t)^a\|\p_1\p_t\bar{Y}_H\|_{L^2}^2+(4a+t)^a\| \p_{b_0^1}\bar{Y}_H\|_{L^2}^2\big).    	
\end{split}
\end{align*}
Next, we estimate the linear pressure term as in Step 1
of the proof of Proposition~\ref{high-prop}.
Using \eqref{constraint}, \eqref{vertical-time-low},
\eqref{vertical-tangent}, and \eqref{geometry-small},
together with the zero horizontal mean of $q$, we obtain
$$
\begin{aligned}
\left|
\Big(
\nabla q,
(4a+t)^a\Big(\p_t\bar Y_H+\frac14\bar Y_H\Big)
\Big)
\right|
&\le
C(\epsilon_0+\sqrt{\delta})(4a+t)^a
\Big(
\|\nabla\p_t\bar Y_H\|_{L^2}^2
+\|\partial_{b_0^1}\bar Y_H\|_{L^2}^2
\Big)
\\
&\quad
+C(4a+t)^a\|\partial_2\p_t\bar Y_L\|_{L^\infty}
 \|\partial_{b_0^1}\bar Y_H\|_{L^2}^2.
\end{aligned}
$$
Integrating the \eqref{W0-7-1} over $[0,t]$,
using the lower bound for the weighted energy and the estimates
for the magnetic coefficient and linear pressure terms, we obtain
\begin{equation}\label{W0es-1}
\begin{aligned}
&\frac14(4a+t)^a\|\p_t\bar Y_H(t)\|_{L^2}^2
+\frac12(4a+t)^a
 \|\partial_{b_0^1}\bar Y_H(t)\|_{L^2}^2
+\frac1{32}(4a+t)^a\|\nabla\bar Y_H(t)\|_{L^2}^2
\\
&\quad
+\int_0^t
\Bigl[
\frac12(4a+\tau)^a\|\nabla\p_\tau\bar Y_H\|_{L^2}^2
+\frac1{16}(4a+\tau)^a
 \|\partial_{b_0^1}\bar Y_H\|_{L^2}^2
\\
&\hspace{3cm}
+\frac{a(a-1)}8(4a+\tau)^{a-2}\|\bar Y_H\|_{L^2}^2
\Bigr]\,d\tau
\\
&\qquad\le
C\mathcal W_0(0)+C\mathcal H_a(t)
+C(\epsilon_0+\sqrt{\delta})\mathcal W_0(t)
+\mathcal F_0^w(t),
\end{aligned}
\end{equation}
where
$$
\mathcal F_0^w(t):=
\sup_{0\le t'\le t}
\left|
\int_0^{t'}(4a+\tau)^a
\Big(
(\Delta_X-\Delta)\bar Y_\tau
-(AB\nabla p-\nabla q),
\p_\tau\bar Y_H+\frac14\bar Y_H
\Big)\,d\tau
\right|.
$$
Taking the supremum in time over $[0,t]$ in \eqref{W0es-1}, one has
\begin{align*}
\mathcal{W}_0(t)\lesssim \mathcal{ H}_a(0)+\mathcal{H}_a(t)+(\epsilon_0+\sqrt{\delta}) \mathcal{ W}_0(t)+\mathcal{H}_a^\frac12(t)\mathcal{ W}_0(t)+\mathcal{ F}_0^w(t).
\end{align*}
Here we used $\mathcal W_0(0)\le(4a)^a\mathcal H_a(0)$.
We now estimate $\mathcal F_0^w(t)$.
Set $w(t)=(4a+t)^a$.
As in the proof of Lemma~\ref{viscosity-lemma}, we use
$$
\Delta_X-\Delta
=\operatorname{div}\bigl((G-I)\nabla\,\cdot\,\bigr)
-G\nabla h\cdot\nabla
$$
and decompose $\bar Y_t=\p_t\bar Y_H+\p_t\bar Y_L$.
Since $\Delta\p_t\bar Y_L$ is independent of $z_1$,
its inner product with
$\p_t\bar Y_H+\frac14\bar Y_H$ vanishes.
Thus, integration by parts gives
$$
\begin{aligned}
&w\Big(
(\Delta_X-\Delta)\bar Y_t-(AB\nabla p-\nabla q),
\p_t\bar Y_H+\frac14\bar Y_H
\Big)
\\
&\quad=
-w\Big(
(G-I)\nabla\p_t\bar Y_H,\nabla\p_t\bar Y_H
\Big)
+\frac w4\Big(
\operatorname{div}\bigl((G-I)\nabla\p_t\bar Y_H\bigr),
\bar Y_H
\Big)
\\
&\qquad
-w\Big(
G\nabla h\cdot\nabla\p_t\bar Y_H,
\p_t\bar Y_H+\frac14\bar Y_H
\Big)
+w\Big(
\Delta_X\p_t\bar Y_L,
\p_t\bar Y_H+\frac14\bar Y_H
\Big)
\\
&\qquad
-w\Big(
AB\nabla p-\nabla q,
\p_t\bar Y_H+\frac14\bar Y_H
\Big)\\
&:=N_1+N_2+N_3+N_4+N_5.
\end{aligned}
$$
By the triangle inequality, it suffices to estimate the
time integral of each term on the right-hand side,
with the absolute value taken after integration,
uniformly over all endpoints in $[0,t]$.

For $N_1$,
by \eqref{geometry-small}, we first have
$$
\begin{aligned}
&\sup_{0\le t\le T}
\left|
\int_0^t
w\bigl((G-I)\nabla\p_\tau\bar Y_H,
             \nabla\p_\tau\bar Y_H\bigr)\,d\tau
\right|
\\
&\qquad\le
\|G-I\|_{L^\infty_t(L^\infty)}
\int_0^T w\|\nabla\p_t\bar Y_H\|_{L^2}^2\,\rmd t
\le C(\epsilon_0+\sqrt{\delta})\mathcal W_0(T).
\end{aligned}
$$

For $N_2$, integration
by parts and the symmetry of $G-I$ yield
$$
\begin{aligned}
&\frac14\int_0^t
w\Big(
\operatorname{div}((G-I)\nabla\p_\tau\bar Y_H),
\bar Y_H
\Big)\,d\tau
\\
&\quad=
-\frac18
\Big[
w\bigl((G-I)\nabla\bar Y_H,\nabla\bar Y_H\bigr)
\Big]_0^t
\\
&\qquad
+\frac18\int_0^t
w'\bigl((G-I)\nabla\bar Y_H,\nabla\bar Y_H\bigr)\,d\tau
+\frac18\int_0^t
w\bigl(G_\tau\nabla\bar Y_H,\nabla\bar Y_H\bigr)\,d\tau.
\end{aligned}
$$
The boundary term satisfies
$$
\sup_{0\le t\le T}
\left|
\Big[
w\bigl((G-I)\nabla\bar Y_H,\nabla\bar Y_H\bigr)
\Big]_0^t
\right|
\le C(\epsilon_0+\sqrt{\delta})\mathcal W_0(T).
$$
For the term containing $w'$, the interpolation estimate
established above gives
$$
\begin{aligned}
w'\|\nabla\bar Y_H\|_{L^2}^2
&\le
C\Big(w\|\partial_{b_0^1}\bar Y_H\|_{L^2}^2\Big)^{(a-1)/a}
 \|\partial_{b_0^1}\bar Y_H\|_{H^a}^{2/a}
\le
Cw\|\partial_{b_0^1}\bar Y_H\|_{L^2}^2
+C\|\partial_{b_0^1}\bar Y_H\|_{H^a}^2.
\end{aligned}
$$
Consequently,
$$
\begin{aligned}
\int_0^T
w'\left|
\bigl((G-I)\nabla\bar Y_H,\nabla\bar Y_H\bigr)
\right|\,\rmd t\le
C(\epsilon_0+\sqrt{\delta})
\int_0^T w'\|\nabla\bar Y_H\|_{L^2}^2\,\rmd t\le
C(\epsilon_0+\sqrt{\delta})
\bigl(\mathcal W_0(T)+\mathcal H_a(T)\bigr).
\end{aligned}
$$

Next, the structure of $G_t$ in \eqref{Gt-low} implies
$$
\begin{aligned}
\left|
(G_t\nabla\bar Y_H,\nabla\bar Y_H)
\right|
&\le
C\|\nabla\p_t\bar Y_H\|_{L^2}
 \|\nabla\bar Y_H\|_{L^\infty}
 \|\nabla\bar Y_H\|_{L^2}
+C\|\partial_2\p_t\bar Y_L\|_{L^\infty}
 \|\partial_1\bar Y_H\|_{L^2}
 \|\nabla\bar Y_H\|_{L^2}.
\end{aligned}
$$
Hence, by H\"older's inequality in time,
$$
\begin{aligned}
&\int_0^T
w\left|
(G_t\nabla\bar Y_H,\nabla\bar Y_H)
\right|\,\rmd t
\\
&\quad\le
C\|\sqrt w\,\nabla\bar Y_H\|_{L^\infty_t(L^2)}
\Bigl(
\|\sqrt w\,\nabla\p_t\bar Y_H\|_{L^2_t(L^2)}
\|\nabla\bar Y_H\|_{L^2_t(L^\infty)}
+\|\sqrt w\,\partial_1\bar Y_H\|_{L^2_t(L^2)}
 \|\partial_2\p_t\bar Y_L\|_{L^2_t(L^\infty)}
\Bigr)
\\
&\quad\le
C\mathcal W_0(T)
\Big(
\mathcal W_0(T)^{(a-1)/(2a)}
\mathcal H_a(T)^{1/(2a)}
+\mathcal L_a(T)^{1/2}
\Big)
\\
&\quad\le C\sqrt{\delta}\,\mathcal W_0(T),
\end{aligned}
$$
where we used \eqref{GN-time} and \eqref{bootstrap}.
Combining these estimates, we obtain
$$
\begin{aligned}
&\sup_{0\le t\le T}
\left|
\frac14\int_0^t
w\Big(
\operatorname{div}((G-I)\nabla\p_\tau\bar Y_H),
\bar Y_H
\Big)\,d\tau
\right|\le
C(\epsilon_0+\sqrt{\delta})
\bigl(\mathcal W_0(T)+\mathcal H_a(T)\bigr).
\end{aligned}
$$
Finally, for $N_3$, since
$\|G\nabla h\|_{L^\infty_t(L^\infty)}\le C\epsilon_0$,
the horizontal Poincar\'e inequality and Young's inequality give
$$
\begin{aligned}
&\int_0^T
w\left|
\Big(
G\nabla h\cdot\nabla\p_t\bar Y_H,
\p_t\bar Y_H+\frac14\bar Y_H
\Big)
\right|\,\rmd t
\\
&\qquad\le
C\epsilon_0\int_0^T
w\|\nabla\p_t\bar Y_H\|_{L^2}
\Big(
\|\nabla\p_t\bar Y_H\|_{L^2}
+\|\partial_{b_0^1}\bar Y_H\|_{L^2}
\Big)\,\rmd t
\\
&\qquad\le C\epsilon_0\mathcal W_0(T).
\end{aligned}
$$
Therefore, the total contribution of the three
high-frequency viscous terms is bounded by
$$
C\mathcal H_a(T)
+C(\epsilon_0+\sqrt{\delta})\mathcal W_0(T).
$$

For the estimate of the low-frequency viscous term $N_4$,
we use the decomposition established in the proof of
Lemma~\ref{viscosity-lemma}. The pure mean term
$\rho^{-1}\partial_2(\rho^{-1}\partial_2\p_t\bar Y_L)$
is orthogonal to $\p_t\bar Y_H+\frac14\bar Y_H$.
Moreover, \eqref{column1}--\eqref{column2} and the
bootstrap bounds imply
$$
\left|
G_{12}+\frac{\rme^h}{\rho}
\bigl(\partial_2\widetilde Y^1+\partial_2\bar Y_L^1\bigr)
\right|
\le C|\nabla\bar Y_H|,
\qquad
|G_{22}-\rho^{-2}|
\le C|\partial_1\bar Y_H|.
$$

For the two divergence terms, spatial integration by parts,
followed by H\"older's inequality, gives
$$
\begin{aligned}
&\int_0^T w
\left|
\Big(
\Big[
G_{12}+\frac{\rme^h}{\rho}
\bigl(\partial_2\widetilde Y^1+\partial_2\bar Y_L^1\bigr)
\Big]\partial_2\p_t\bar Y_L,
\partial_1\p_t\bar Y_H+\frac14\partial_1\bar Y_H
\Big)
\right|\,\rmd t
\\
&\quad+
\int_0^T w
\left|
\Big(
(G_{22}-\rho^{-2})\partial_2\p_t\bar Y_L,
\partial_2\p_t\bar Y_H+\frac14\partial_2\bar Y_H
\Big)
\right|\,\rmd t
\\
&\qquad\le
C\int_0^T
w\|\partial_2\p_t\bar Y_L\|_{L^\infty}
\|\nabla\bar Y_H\|_{L^2}
\Big(
\|\nabla\p_t\bar Y_H\|_{L^2}
+\|\partial_1\bar Y_H\|_{L^2}
\Big)\,\rmd t
\\
&\qquad\le
C\|\partial_2\p_t\bar Y_L\|_{L^2_t(L^\infty)}
 \|\sqrt w\,\nabla\bar Y_H\|_{L^\infty_t(L^2)}
\Big(
\|\sqrt w\,\nabla\p_t\bar Y_H\|_{L^2_t(L^2)}
+\|\sqrt w\,\partial_1\bar Y_H\|_{L^2_t(L^2)}
\Big)
\\
&\qquad\le
C\mathcal L_a(T)^{1/2}\mathcal W_0(T).
\end{aligned}
$$
Here the factor $\partial_1\bar Y_H$ in
$G_{22}-\rho^{-2}$ allows us to use the weighted magnetic
dissipation when pairing with $\partial_2\bar Y_H$.

For the terms containing derivatives of $h$, we use
$\|\nabla h\|_{L^\infty}\le C\epsilon_0$ and the horizontal
Poincar\'e inequality to obtain
$$
\begin{aligned}
&\int_0^T w
\left|
\Big(
(\partial_1h)
\Big[
G_{12}+\frac{\rme^h}{\rho}
\bigl(\partial_2\widetilde Y^1+\partial_2\bar Y_L^1\bigr)
\Big]\partial_2\p_t\bar Y_L,
\p_t\bar Y_H+\frac14\bar Y_H
\Big)
\right|\,\rmd t
\\
&\quad+
\int_0^T w
\left|
\Big(
(\partial_2h)(G_{22}-\rho^{-2})\partial_2\p_t\bar Y_L,
\p_t\bar Y_H+\frac14\bar Y_H
\Big)
\right|\,\rmd t
\\
&\qquad\le
C\epsilon_0\int_0^T
w\|\partial_2\p_t\bar Y_L\|_{L^\infty}
\|\nabla\bar Y_H\|_{L^2}
\Big(
\|\nabla\p_t\bar Y_H\|_{L^2}
+\|\partial_1\bar Y_H\|_{L^2}
\Big)\,\rmd t
\\
&\qquad\le
C\epsilon_0\mathcal L_a(T)^{1/2}\mathcal W_0(T).
\end{aligned}
$$
Combining these estimates and using \eqref{bootstrap},
we conclude that
$$
\begin{aligned}
&\sup_{0\le t\le T}
\left|
\int_0^t w(\tau)
\Big(
\Delta_X\p_\tau\bar Y_L,
\p_\tau\bar Y_H+\frac14\bar Y_H
\Big)\,d\tau
\right|
\\
&\qquad\le
C(1+\epsilon_0)\mathcal L_a(T)^{1/2}\mathcal W_0(T)
\le C\sqrt{\delta}\,\mathcal W_0(T).
\end{aligned}
$$

Finally, we estimate  the nonlinear pressure term $N_5$.
By \eqref{pressure-remainder} with $s=0$, we have
$$
\|AB\nabla p-\nabla q\|_{L^2}
\le
C(\epsilon_0+\sqrt{\delta})
\Big(
\|\partial_{b_0^1}\bar Y_H\|_{L^2}
+\|\nabla\p_t\bar Y_H\|_{L^2}
\Big).
$$
Hence, using the Cauchy--Schwarz and horizontal
Poincar\'e inequalities, we obtain
$$
\begin{aligned}
&\int_0^T w
\left|
\Big(
AB\nabla p-\nabla q,
\p_t\bar Y_H+\frac14\bar Y_H
\Big)
\right|\,\rmd t
\\
&\qquad\le
C\int_0^T
w\|AB\nabla p-\nabla q\|_{L^2}
\Big(
\|\p_t\bar Y_H\|_{L^2}+\|\bar Y_H\|_{L^2}
\Big)\,\rmd t
\\
&\qquad\le
C(\epsilon_0+\sqrt{\delta})
\int_0^T w
\Big(
\|\partial_{b_0^1}\bar Y_H\|_{L^2}
+\|\nabla\p_t\bar Y_H\|_{L^2}
\Big)^2\,\rmd t
\\
&\qquad\le
C(\epsilon_0+\sqrt{\delta})\mathcal W_0(T).
\end{aligned}
$$

Combining the estimates for the high-frequency viscous
terms, the low-frequency viscous contribution, and the
nonlinear pressure term, we obtain
$$
\mathcal F_0^w(T)
\le
C\mathcal H_a(T)
+C(\epsilon_0+\sqrt{\delta})\mathcal W_0(T).
$$
Taking the supremum over $0\le t\le T$ in
\eqref{W0es-1}, and estimating the cumulative dissipation
at $t=T$, we deduce that
$$
\begin{aligned}
\mathcal W_0(T)
&\le
C\mathcal W_0(0)+C\mathcal H_a(T)
+C(\epsilon_0+\sqrt{\delta})\mathcal W_0(T)
+C\mathcal F_0^w(T)
\\
&\le
C\mathcal W_0(0)+C\mathcal H_a(T)
+C(\epsilon_0+\sqrt{\delta})\mathcal W_0(T).
\end{aligned}
$$
Since $\mathcal W_0(0)\le(4a)^a\mathcal H_a(0)$,
we conclude that
$$
\mathcal W_0(T)
\le
C\bigl(\mathcal H_a(0)+\mathcal H_a(T)\bigr)
+C(\epsilon_0+\sqrt{\delta})\mathcal W_0(T),
$$
which proves \eqref{weighted-bound}.
\end{proof}

\section{Estimate of the low-frequency components}\label{sec:7}
In this section, we estimate the low-frequency components.
As proved in Lemma \ref{constraint-lemma}, the incompressibility constraint controls $\pa_2\bY_L^2$ by the high-frequency components, but does not directly provide a
bound for $\pa_2\bY_L^1$.
To estimate the latter, we first derive an evolution equation
for the weighted mean velocity.
The Jacobian factor $\rme^{-h}$ leads naturally to a weighted
mean over $\mathbb T$, whose evolution has a one-dimensional
parabolic structure.
Set
\begin{equation*}
 v(t,z_2):=\frac{\langle\rme^{-h}\p_t\bY^1\rangle}{\rho},
 \qquad
 D:=\rho^{-1}\pa_2.
\end{equation*}
Since $\rho$ depends only on $z_2$, the operator $D$ is
independent of time and satisfies
$D^2f=\rho^{-1}\pa_2\big(\rho^{-1}\pa_2f\big).
$
Moreover, for sufficiently small $\epsilon_0$,
\eqref{static-est3} implies that $\rho$ is bounded away
from zero and
\[
 \|\rho-1\|_{H^a}
 +\|\rho^{-1}-1\|_{H^a}
 \le C\epsilon_0.
\]
The Sobolev product estimate therefore gives
\begin{equation}\label{D-norm-equivalence}
 C^{-1}\|\pa_2f\|_{H^a}
 \le\|Df\|_{H^a}
 \le C\|\pa_2f\|_{H^a},
\end{equation}
where we have used $Df=\rho^{-1}\pa_2f$ and
$\pa_2f=\rho Df$.

The following lemma gives the evolution equations for
$v$ and the ordinary mean velocity $\p_t\bY_L^1$.
\begin{lemma}\label{mean-equation-lemma}
The weighted mean $v$ satisfies
\begin{equation}\label{weighted-low}
 \p_t v-D^2v=DR,
\end{equation}
where $R$ is given by
\begin{align}\label{flux-expanded}
\begin{split}
     R=\Big\langle&
 2\pa_1\bY_H^1D\p_t\bY_H^1
 -(D\bY_H^1)\pa_1\p_t\bY_H^1
 +\rme^h|\pa_1\bY_H|^2\pa_2\p_t\bY^1\\
 &-\rme^h\Big[
 \pa_1\bY_H^1(\pa_2\tY^1+\pa_2\bY^1)
 +\pa_1\bY_H^2(\rho+\pa_2\bY^2)
 \Big]\pa_1\p_t\bY_H^1
 +\pa_1\bY_H^2\,p_H
 \Big\rangle.
\end{split}
\end{align}
The ordinary mean velocity $\p_t\bY_L^1$ satisfies
\begin{equation}\label{ordinary-low}
 \p_t^2\bY_L^1-D^2\p_t\bY_L^1
 =DR-(\p_t-D^2)
 \big\langle
 \pa_1\tY^1
 \p_t\bY_H^1
 \big\rangle.
\end{equation}
The corresponding initial data are
\begin{align*}
 v(0)
 &=\frac{\big\langle\rme^{-h}u_0^1(y(z))\big\rangle}{\rho},
 \qquad
 \bY_L^1(0)
 =-\langle\tY^1\rangle,
 \qquad
 \p_t\bY_L^1(0)=\big\langle u_0^1(y(z))\big\rangle.
\end{align*}
\end{lemma}

\begin{proof}
We first derive \eqref{weighted-low}.
By Lemma \ref{geo-identities},
\[
 \pa_k\bigl(\rme^{-h}(AB)_{1k}\bigr)=0,
 \qquad
 \rme^{-h}(AB)_{12}=-\pa_1\bY_H^2,
 \qquad
 \big\langle\rme^{-h}\Db^2\bY_H\big\rangle
 =\rho\gamma\big\langle\pa_1\Db\bY_H\big\rangle=0.
\]
Multiplying the first component of \eqref{equ-bY}
by $\rme^{-h}$ and integrating over $\mathbb T$ in $z_1$,
we therefore obtain
\begin{align}\label{low-ave-int}
 \rho\p_t v
 &=\pa_2\big\langle
 \rme^{-h}G_{2k}\pa_k\p_t\bY^1
 \big\rangle
 -\pa_2\big\langle\rme^{-h}(AB)_{12}p\big\rangle\notag\\
 &=\pa_2\Big\langle
 \rme^{-h}G_{2k}\pa_k\p_t\bY^1
 +(\pa_1\bY_H^2)p
 \Big\rangle,
\end{align}
where we have used the time independence of $h$ and $\rho$.

The second row of \eqref{matrix-G} satisfies
\begin{align*}
 \rme^{-h}G_{21}+D\tY^1
 &=-D\bY^1
 -\rme^h\Big[
 \pa_1\bY_H^1(\pa_2\tY^1+\pa_2\bY^1)
 +\pa_1\bY_H^2(\rho+\pa_2\bY^2)
 \Big],\\
 \rme^{-h}G_{22}-\frac{\rme^{-h}}{\rho^2}
 &=\frac2\rho\pa_1\bY_H^1
 +\rme^h|\pa_1\bY_H|^2.
\end{align*}
Moreover, \eqref{corrector-identities} and
\eqref{rho-identities} imply
\[
 D\pa_1\tY^1
 =D\Big(\frac{\gamma}{b_0^1(y(z))}-1\Big)
 =D\Big(\frac{\rme^{-h}}{\rho}\Big).
\]
Thus, integration by parts in $z_1$ yields
\begin{align*}
 \Big\langle
 -(D\tY^1)\pa_1\p_t\bY^1
 +\frac{\rme^{-h}}{\rho}D\p_t\bY^1
 \Big\rangle
 &=\Big\langle
 D(\pa_1\tY^1)\p_t\bY^1
 +\frac{\rme^{-h}}{\rho}D\p_t\bY^1
 \Big\rangle=Dv,
\end{align*}
since $D$ commutes with $\pa_1$ and with integration
over $\mathbb T$ in the $z_1$ variable.
Using
\[
 \big\langle(D\bY_L^1)\pa_1\p_t\bY_H^1\big\rangle=0,
 \qquad
 \big\langle\pa_1\bY_H^1D\p_t\bY_L^1\big\rangle=0,
 \qquad
 \big\langle(\pa_1\bY_H^2)p\big\rangle
 =\big\langle(\pa_1\bY_H^2)p_H\big\rangle,
\]
we deduce from \eqref{low-ave-int} that
\begin{align*}
 \rho\p_t v
 &=\pa_2\big\langle
 \rme^{-h}G_{2k}\pa_k\p_t\bY^1
 +(\pa_1\bY_H^2)p
 \big\rangle=\pa_2(Dv+R),
\end{align*}
where $R$ is given by \eqref{flux-expanded}.
Dividing by $\rho$ proves \eqref{weighted-low}.

Since $\p_t\bY_L^1$ is independent of $z_1$,
$\rho=\langle\rme^{-h}\rangle$, and
$\langle\p_t\bY_H^1\rangle=0$,
we obtain from \eqref{rho-identities} that
\begin{align}\label{vbarY1}
 v
 &=\p_t\bY_L^1
 +\Big\langle
 \frac{\rme^{-h}}{\rho}\p_t\bY_H^1
 \Big\rangle
 =\p_t\bY_L^1
 +\Big\langle
 \Big(\frac{\gamma}{b_0^1(y(z))}-1\Big)
 \p_t\bY_H^1
 \Big\rangle=\p_t\bY_L^1
 +\big\langle
 \p_1\widetilde{Y}^1
 \p_t\bY_H^1
 \big\rangle.
\end{align}
Applying $\p_t-D^2$ to this identity and using
\eqref{weighted-low} yields \eqref{ordinary-low}.
The initial values follow from
$\bY(0)=-\tY$ and $\p_t\bY(0,z)=u_0(y(z))$.
This completes the proof.
\end{proof}
We next estimate $R$ defined in \eqref{flux-expanded}.
The following bounds in $L^2(0,T;H^a)$ and $L^1(0,T;H^a)$
will be used to estimate $\mathcal{L}_a$ and $\mathcal{S}_a$.

\begin{lemma}\label{flux-lemma}
Assume that \eqref{bootstrap} holds on $[0,T]$.
If $\delta$ is sufficiently small, then
\begin{align}
 \|R\|_{L^2(0,T;H^a)}
 &\le C
 \bigl(\HH(T)+\LL(T)\bigr),
 \label{flux-L2}\\
 \|R\|_{L^1(0,T;H^a)}
 &\le C\big(\HH(T)+\LL(T)+\WW(T)\big).
 \label{flux-L1}
\end{align}
\end{lemma}
\begin{proof}
Let us first estimate $ \|R\|_{H^a}$ at each fixed time. Since integration over $\mathbb T$ in the $z_1$ variable
commutes with $\partial_2$, the Cauchy--Schwarz inequality gives
\begin{equation}\label{flux-average-bound}
 \begin{aligned}
 \big\|\langle F\rangle\big\|_{H^a}^2
 &=\sum_{j=0}^a
 \big\|\langle\partial_2^jF\rangle\big\|_{L^2(\mathbb R)}^2
 \le\sum_{j=0}^a
 \|\partial_2^jF\|_{L^2(\mathbb T\times\mathbb R)}^2
 \le\|F\|_{H^a}^2,
 \end{aligned}
\end{equation}
for any $F$. We apply \eqref{flux-average-bound},
\eqref{D-norm-equivalence}, and the Sobolev product estimate
to \eqref{flux-expanded}.
Using \eqref{static-est3}, we obtain
\begin{align} \label{flux-product-bound}
\begin{split}
     \|R\|_{H^a}
 &\lesssim
 \big\|
 \big\langle
 (D\bar Y_H^1)\partial_1\p_t\bar Y_H^1
 \big\rangle
 \big\|_{H^a}
 +\|\partial_1\bar Y_H\|_{H^a}
       \|\nabla\p_t\bar Y_H\|_{H^a}\\
 &\quad
 +\|\partial_1\bar Y_H\|_{H^a}^2
 \big(
 \|\nabla\p_t\bar Y_H\|_{H^a}
 +\|\partial_2\p_t\bar Y_L^1\|_{H^a}
 \big)\\
 &\quad
 +\|\partial_1\bar Y_H\|_{H^a}
       \|\partial_1\p_t\bar Y_H\|_{H^a}
 \big(
 \|\partial_2\widetilde Y^1+\partial_2\bar Y^1\|_{H^a}
 +1+\|\rho-1\|_{H^a}
 +\|\partial_2\bar Y^2\|_{H^a}
 \big)\\
 &\quad
 +\|\partial_1\bar Y_H\|_{H^a}\|p_H\|_{H^a}.
\end{split}
\end{align}
Here we have split
$\partial_2\p_t\bar Y^1
=\partial_2\p_t\bar Y_H^1+\partial_2\p_t\bar Y_L^1$.
Since $\bar Y=Y-\widetilde Y$ and
$\partial_2\widetilde Y^1+\partial_2\bar Y^1=\partial_2Y^1$, \eqref{static-est3},
\eqref{static-est1}, \eqref{bootstrap} and
\eqref{geometry-small} imply
\begin{align}\label{flux-1}
    \|\partial_1\bar Y_H\|_{H^a}+\|\nabla\bar Y\|_{H^a}
 +\|\partial_2\widetilde Y^1+\partial_2\bar Y^1\|_{H^a}
 +\|\rho-1\|_{H^a}
 \le C(\epsilon_0+\sqrt\delta),
\end{align}
uniformly on $[0,T]$. Moreover, the horizontal Poincar\'e inequality and
\eqref{pressure-basic} with $s=a$ yield
\begin{equation}\label{flux-pressure-bound}
 \begin{aligned}
 \|p_H\|_{H^a}
 &\le C\|\nabla p\|_{H^{a-1}}\le C\big(
 \|\Db\bar Y_H\|_{H^a}
 +\|\nabla\p_t\bar Y_H\|_{H^a}
 \big).
 \end{aligned}
\end{equation}
Plugging \eqref{flux-1} and \eqref{flux-pressure-bound} into \eqref{flux-product-bound}
gives
\begin{align} \label{flux-spatial-bound}
\begin{split}
     \|R\|_{H^a}
 &\lesssim
 \big\|
 \big\langle
 (D\bar Y_H^1)\partial_1\p_t\bar Y_H^1
 \big\rangle
 \big\|_{H^a}
 +\|\partial_1\bar Y_H\|_{H^a}
 \big(
 \|\nabla\p_t\bar Y_H\|_{H^a} +\|\partial_2\p_t\bar Y_L^1\|_{H^a}+ \|\Db\bar Y_H\|_{H^a}\big).
\end{split}
\end{align}

We now prove \eqref{flux-L2}.
By
\eqref{flux-average-bound} and \eqref{D-norm-equivalence},
\[
 \big\|
 \big\langle
 (D\bar Y_H^1)\partial_1\p_t\bar Y_H^1
 \big\rangle
 \big\|_{H^a}
 \lesssim\|\nabla\bar Y_H\|_{H^a}
       \|\nabla\p_t\bar Y_H\|_{H^a}.
\]
Applying H\"older's inequality in time to
\eqref{flux-spatial-bound}, we therefore obtain
\begin{align*}
 \|R\|_{L^2_t(H^a)}
 &\lesssim
 \|\nabla\bar Y_H\|_{L^\infty_t(H^a)}
 \big(
 \|\nabla\p_t\bar Y_H\|_{L^2_t(H^a)}
 +\|\partial_2\p_t\bar Y_L^1\|_{L^2_t(H^a)}+\|\Db\bar Y_H\|_{L^2_t(H^a)}
 \big)\\
 &\lesssim
 \mathcal H_a(T)^{1/2}
 \big(
 \mathcal H_a(T)^{1/2}
 +\mathcal L_a(T)^{1/2}
 \big) \lesssim \mathcal H_a(T)+\mathcal L_a(T).
\end{align*}
This proves \eqref{flux-L2}.

To prove \eqref{flux-L1}, we first use \eqref{GN-time} and Young's inequality to obtain
\begin{equation}\label{flux-mixed-L1}
 \begin{aligned}
 \big\|
 \big\langle
 (D\bar Y_H^1)\partial_1\p_t\bar Y_H^1
 \big\rangle
 \big\|_{L^1_t(H^a)}&\lesssim
 \|\nabla\bar Y_H\|_{L^\infty_t(H^a)}
       \|\nabla\p_t\bar Y_H\|_{L^1_t(L^\infty)}
 +\|\nabla\bar Y_H\|_{L^2_t(L^\infty)}
       \|\nabla\p_t\bar Y_H\|_{L^2_t(H^a)}\\
 &\lesssim
 \mathcal H_a(T)^{\frac{a+1}{2a}}
       \mathcal W_0(T)^{\frac{a-1}{2a}}
 \lesssim \mathcal H_a(T)+\mathcal W_0(T).
 \end{aligned}
\end{equation}
Integrating \eqref{flux-spatial-bound} in time,
applying H\"older's inequality and \eqref{flux-mixed-L1}, we conclude that
\begin{align*}
 \|R\|_{L^1_t(H^a)}
 &\lesssim
 \big\|
 \big\langle
 (D\bar Y_H^1)\partial_1\p_t\bar Y_H^1
 \big\rangle
 \big\|_{L^1_t(H^a)}\\
&\quad +\|\partial_1\bar Y_H\|_{L^2_t(H^a)} \big(
 \|\nabla\p_t\bar Y_H\|_{L^2_t(H^a)}
 +\|\partial_2\p_t\bar Y_L^1\|_{L^2_t(H^a)}+\|\Db\bar Y_H\|_{L^2_t(H^a)}
 \big)\\
 &\lesssim
 \mathcal H_a(T)+\mathcal W_0(T)
 +\mathcal L_a(T),
\end{align*}
We thus finish the proof of \eqref{flux-L1}.

\end{proof}

We now use \eqref{weighted-low} and the $L^2(0,T;H^a)$
estimate for $R$ to control the weighted mean $v$ and
recover the ordinary mean velocity $\p_t\bY_L^1$.
The vertical component $\p_t\bY_L^2$ is controlled by
the incompressibility constraint.
Together, these estimates yield the following bound
for $\LL(T)$.
\begin{proposition}\label{low-energy-proposition}
Assume that \eqref{bootstrap} holds on $[0,T]$.
If $\delta$ is sufficiently small, then
\begin{equation}\label{low-bound}
 \LL(T)
 \le C\epsilon_0^2+C\epsilon_0^2\HH(T)
 +C\bigl(\HH(T)+\LL(T)\bigr)^2.
\end{equation}
\end{proposition}

\begin{proof}
We first estimate $v$.
Since $D$ is independent of time and commutes with $D^2$,
applying $D^j$ to \eqref{weighted-low} and taking the
$L^2(\rho\,\rmd z_2)$ inner product with $D^jv$ yields
\begin{align*}
 &\frac12\frac{\rmd}{\rmd t}\|D^jv\|_{L^2(\rho\,\rmd z_2)}^2
 +\|D^{j+1}v\|_{L^2(\rho\,\rmd z_2)}^2=-\int_{\mathbb R}D^jR\,D^{j+1}v\,\rho\,\rmd z_2\\
 &\qquad
 \le\frac12\|D^jR\|_{L^2(\rho\,\rmd z_2)}^2
 +\frac12\|D^{j+1}v\|_{L^2(\rho\,\rmd z_2)}^2,
 \qquad 0\le j\le a,
\end{align*}
where we used the integration-by-parts identity
\[
 \int_{\mathbb R}(Df)g\,\rho\,\rmd z_2
 =-\int_{\mathbb R}f(Dg)\,\rho\,\rmd z_2.
\]
Summing over $0\le j\le a$ and integrating in time,
we obtain
\begin{align}\label{weighted-mean-energy}
 \begin{split}
 &\sup_{0\le t\le T}
 \sum_{j=0}^a\|D^jv(t)\|_{L^2(\rho\,\rmd z_2)}^2
 +\sum_{j=0}^a\int_0^T
 \|D^{j+1}v\|_{L^2(\rho\,\rmd z_2)}^2\,\rmd t\\
 &\qquad\lesssim
 \sum_{j=0}^a\|D^jv(0)\|_{L^2(\rho\,\rmd z_2)}^2
 +\sum_{j=0}^a\int_0^T
 \|D^jR\|_{L^2(\rho\,\rmd z_2)}^2\,\rmd t.
 \end{split}
\end{align}

We next relate these quantities to the Sobolev norms.
By \eqref{static-est3} and the one-dimensional Sobolev
embedding, $\rho$ is uniformly bounded above and away
from zero, and $\rho$ and $\rho^{-1}$ have bounded
derivatives up to order $a-1$.
For $1\le j\le a$, expanding $D^j$ gives
\[
 D^jf=\rho^{-j}\pa_2^jf
 +\sum_{k=1}^{j-1}c_{jk}\pa_2^kf,
 \qquad
 \|c_{jk}\|_{L^\infty}\le C,
\]
where the coefficients  $c_{jk}$ involve derivatives of $\rho$
of order at most $j-1$.
Consequently,
\[
 \|D^jf\|_{L^2(\rho\,\rmd z_2)}
 \lesssim\|f\|_{H^j},
 \qquad
 \|\pa_2^jf\|_{L^2}
 \lesssim\|D^jf\|_{L^2(\rho\,\rmd z_2)}
 +\|f\|_{H^{j-1}}.
\]
Together with
$\|f\|_{L^2(\rho\,\rmd z_2)}\simeq\|f\|_{L^2}$,
induction on $j$ proves the first equivalence below:
\begin{equation}\label{mean-norm-equivalence}
 \begin{aligned}
 \sum_{j=0}^a\|D^jf\|_{L^2(\rho\,\rmd z_2)}^2
 &\simeq\|f\|_{H^a}^2,\\
 \sum_{j=0}^a\|D^{j+1}f\|_{L^2(\rho\,\rmd z_2)}^2
 &\simeq\|\pa_2f\|_{H^a}^2.
 \end{aligned}
\end{equation}
The second follows by applying the first to $Df$
and using \eqref{D-norm-equivalence}.
Thus, \eqref{weighted-mean-energy} implies
\[
 \|v\|_{L^\infty_t(H^a)}^2
 +\|\pa_2v\|_{L^2_t(H^a)}^2
 \lesssim\|v(0)\|_{H^a}^2+\|R\|_{L^2_t(H^a)}^2.
\]

For the initial value, Lemma~\ref{mean-equation-lemma},
\eqref{flux-average-bound}, \eqref{static-est3},
\eqref{app-composition}, and \eqref{asmp} give
\[
 \|v(0)\|_{H^a}
 =\Big\|
 \frac{\big\langle\rme^{-h}u_0^1(y(z))\big\rangle}{\rho}
 \Big\|_{H^a}
 \lesssim\|u_0(y(z))\|_{H^a}
 \lesssim\|u_0\|_{H^a}
 \le\epsilon_0.
\]
Using \eqref{flux-L2}, we therefore obtain
\begin{equation}\label{horizontal-velocity}
 \|v\|_{L^\infty_t(H^a)}^2
 +\|\pa_2v\|_{L^2_t(H^a)}^2
 \lesssim\epsilon_0^2
 +\bigl(\HH(T)+\LL(T)\bigr)^2.
\end{equation}

We now recover the ordinary mean $\p_t\bY_L^1$.
By \eqref{vbarY1},
\begin{equation}\label{recover-velocity}
 \p_t\bY_L^1
 =v-\big\langle\pa_1\tY^1\,\p_t\bY_H^1\big\rangle.
\end{equation}
Integration by parts in $z_1$ yields
\[
 \pa_2\big\langle\pa_1\tY^1\,\p_t\bY_H^1\big\rangle
 =\big\langle
 \pa_1\tY^1\,\pa_2\p_t\bY_H^1
 -\pa_2\tY^1\,\pa_1\p_t\bY_H^1
 \big\rangle.
\]
By \eqref{recover-velocity},
we get from \eqref{horizontal-velocity} and \eqref{static-est1} that
\begin{align}\label{ptbarY1}
\begin{split}
     &\|\p_t\bY_L^1\|_{L^\infty_t(H^a)}^2
 +\|\pa_2\p_t\bY_L^1\|_{L^2_t(H^a)}^2\\
 &\qquad\lesssim
 \|v\|_{L^\infty_t(H^a)}^2
 +\|\pa_2v\|_{L^2_t(H^a)}^2
 +\|\nabla\tY\|_{H^a}^2
 \big(
 \|\p_t\bY_H\|_{L^\infty_t(H^a)}^2
 +\|\nabla\p_t\bY_H\|_{L^2_t(H^a)}^2
 \big)\\
 &\qquad\lesssim
 \epsilon_0^2+\epsilon_0^2\HH(T)
 +\bigl(\HH(T)+\LL(T)\bigr)^2.
\end{split}
\end{align}

It remains to estimate $\p_t\bY_L^2$.
Integrating \eqref{geom:velocity-div} in $z_1$ gives
\begin{equation}\label{ptbarY2-diff}
 \pa_2\big\langle
 \rme^{-h}(AB)_{i2}\p_t\bY^i
 \big\rangle=0.
\end{equation}
This identity also follows by differentiating
\eqref{vertical-average} in time and integrating
by parts in $z_1$.
By \eqref{static-est3}, \eqref{geometry-small},
\eqref{def-Ha}, and \eqref{def-La},
\[
 \begin{aligned}
 \Big\|
 \big\langle\rme^{-h}(AB)_{i2}\p_t\bY^i\big\rangle
 \Big\|_{L^2(\R)}
 &\le C\|\rme^{-h}AB\|_{L^\infty}
        \|\p_t\bY\|_{L^2(\T\times\R)}<\infty.
 \end{aligned}
\]
Thus this average is independent of $z_2$ and belongs
to $L^2(\R)$, so
\[
 \big\langle\rme^{-h}(AB)_{i2}\p_t\bY^i\big\rangle=0.
\]
Expanding this identity by \eqref{geom:AB} and using
\eqref{corrector-identities}, periodicity in $z_1$,
and $\langle\p_t\bY_H\rangle=0$, we obtain
\begin{equation}\label{vertical-velocity}
 \p_t\bY_L^2
 =\big\langle
 \p_t\bY_H^1\,\pa_1\bY_H^2
 -(\pa_1\tY^1+\pa_1\bY_H^1)\p_t\bY_H^2
 \big\rangle.
\end{equation}
Combining \eqref{vertical-velocity} and
\eqref{vertical-time-est} with
\eqref{flux-average-bound}, the Sobolev product estimate,
and \eqref{static-est1}, we obtain
\begin{align}\label{ptbarY2}
\begin{split}
     &\|\p_t\bY_L^2\|_{L^\infty_t(H^a)}^2
 +\|\pa_2\p_t\bY_L^2\|_{L^2_t(H^a)}^2\\
 &\qquad\lesssim
 \big(
 \epsilon_0^2+\|\nabla\bY_H\|_{L^\infty_t(H^a)}^2
 \big)
 \big(
 \|\p_t\bY_H\|_{L^\infty_t(H^a)}^2
 +\|\nabla\p_t\bY_H\|_{L^2_t(H^a)}^2
 \big)\\
 &\qquad\lesssim
 \bigl(\epsilon_0^2+\HH(T)\bigr)\HH(T).
\end{split}
\end{align}
Consequently, \eqref{low-bound} follows from \eqref{ptbarY1} and \eqref{ptbarY2}.

\end{proof}
We next use \eqref{eulerian-low-assumption} and the $L^1_t(H^a)$
estimate for $R$ to control $\pa_2\bY_L^1$.

\begin{proposition}\label{horizontal-deformation-proposition}
Assume that \eqref{eulerian-low-assumption} holds and that
\eqref{bootstrap} is satisfied on $[0,T]$.
If $\delta$ is sufficiently small, then
\begin{align} \label{strain-bound}
 \SSS(T)
 \leq C\epsilon_0^2+C\epsilon_0^2\HH(T)
 +C\big(\HH(T)+\LL(T)+\WW(T)\big)^2.
\end{align}
\end{proposition}

\begin{proof}
We first construct a primitive $V(t,z_2)$ of the weighted mean
$\rho v=\langle\rme^{-h}\p_t\bY^1\rangle$
with respect to $z_2$, satisfying
\[
 \pa_2V=\rho v
 =\big\langle\rme^{-h}\p_t\bY^1\big\rangle,
 \qquad\text{equivalently,}\qquad DV=v.
\]
We write $V(t):=V(t,\cdot)$ and first prescribe its
initial value $V(0,z_2)$.
By \eqref{eulerian-low-assumption}, there exists
$F_0\in L^2(\R)$ such that
\[
 F_0'(x_2)=\int_\T u_0^1(x_1,x_2)\,\rmd x_1,
 \qquad
 \|F_0\|_{L^2(\R)}\le\epsilon_0.
\]
Define $V(0)$ by
\[
 V(0,z_2)
 :=F_0(z_2)
 +\int_\T\int_{z_2}^{y_2(z_1,z_2)}
       u_0^1(z_1,r)\,\rmd r\,\rmd z_1.
\]
Since $y_1=z_1$ and $\pa_2y_2=\rme^{-h}$ by
\eqref{mapyz} and \eqref{pz2y2}, differentiation gives
\[
 \pa_2V(0,z_2)
 =\int_\T \rme^{-h(z)}u_0^1(y(z))\,\rmd z_1
 =\big\langle\rme^{-h}u_0^1(y(z))\big\rangle.
\]
Now let us estimate $V(0)$.
Changing variables in the definition of
$V(0)$, we obtain
\begin{align}\label{V0-F0}
     V(0,z_2)-F_0(z_2)
 =\int_\T\int_0^{y_2(z_1,z_2)-z_2}
       u_0^1(z_1,z_2+r)\,\rmd r\,\rmd z_1.
\end{align}
By \eqref{mapyz}, \eqref{b01-cond}, and Sobolev embedding,
\[
 \|y_2-z_2\|_{L^\infty}
 \le \big\|b_0^2/b_0^1\big\|_{L^\infty}
 \lesssim \|b_0^2\|_{H^{a+1}}
 \lesssim \epsilon_0.
\]
Since $|y_2(z_1,z_2)-z_2|\le\|y_2-z_2\|_{L^\infty}$,
taking absolute values and enlarging the interval
of integration in \eqref{V0-F0} give
\begin{align*}
 |V(0,z_2)-F_0(z_2)|
 &\le \int_\T
 \Big|\int_0^{y_2(z_1,z_2)-z_2}
       u_0^1(z_1,z_2+r)\,\rmd r\Big|\,\rmd z_1\\
 &\le
 \int_{-\|y_2-z_2\|_{L^\infty}}^{\|y_2-z_2\|_{L^\infty}}
 \int_\T |u_0^1(z_1,z_2+r)|\,\rmd z_1\,\rmd r.
\end{align*}
Applying Minkowski's inequality in $L^2(\R)$ with respect
to $z_2$, followed by translation invariance and
the Cauchy--Schwarz inequality in $z_1$, we deduce that
\begin{align*}
 \|V(0)-F_0\|_{L^2(\R)}
 &\le
 \int_{-\|y_2-z_2\|_{L^\infty}}^{\|y_2-z_2\|_{L^\infty}}
 \int_\T
 \|u_0^1(z_1,\cdot+r)\|_{L^2(\R)}
 \,\rmd z_1\,\rmd r\\
 &\le 2\|y_2-z_2\|_{L^\infty}
       \|u_0^1\|_{L^2(\T\times\R)}
 \lesssim\epsilon_0^2,
\end{align*}
where we used $|\T|=1$, the bound
$\|y_2-z_2\|_{L^\infty}\lesssim\epsilon_0$,
and \eqref{asmp}.
Consequently,
\[\|V(0)\|_{L^2(\R)}\leq \|V(0)-F_0\|_{L^2(\R)}+\|F_0\|_{L^2(\R)}\le\epsilon_0+C\epsilon_0^2.\]
Using \eqref{flux-average-bound}, \eqref{static-est3},
\eqref{app-composition}, and \eqref{asmp}, we obtain
\begin{equation}\label{primitive-initial-bound}
 \begin{aligned}
 \|V(0)\|_{H^{a+1}}
 &\lesssim \|V(0)\|_{L^2}
 +\big\|\langle\rme^{-h}u_0^1(y(z))\rangle\big\|_{H^a}\lesssim \epsilon_0+\epsilon_0^2+\|u_0\|_{H^a}
 \lesssim \epsilon_0.
 \end{aligned}
\end{equation}
We now define
\begin{equation}\label{def-V}
 V(t):=V(0)+\int_0^t(Dv+R)(\tau)\,\rmd\tau.
\end{equation}
By \eqref{horizontal-velocity},
\eqref{D-norm-equivalence}, and
Lemma~\ref{flux-lemma}, we have
$ Dv\in L^2_t(H^a)$ and $R\in L^1_t(H^a)\cap L^2_t(H^a).$
Hence
\[
 \int_0^T\|(Dv+R)(\tau)\|_{H^a}\,\rmd\tau
 \le T^{1/2}\|Dv\|_{L^2_t(H^a)}
      +\|R\|_{L^1_t(H^a)}<\infty.
\]
Thus, the time integral in \eqref{def-V} is well defined
in $H^a$, and $V\in C([0,T];H^a)$.

Since $D$ is independent of time and $DV(0)=v(0)$,
\eqref{weighted-low} implies
\[
 DV(t)
 =v(0)+\int_0^t(D^2v+DR)(\tau)\,\rmd\tau
 =v(t),
\]
which gives $\pa_2V=\rho v
 =\big\langle\rme^{-h}\p_t\bY^1\big\rangle.$
Differentiating \eqref{def-V} in time and substituting $v=DV$, we obtain
\begin{equation}\label{primitive}
 \p_tV-D^2V=R, \quad V(0)= \pa_2^{-1}
 \big\langle\rme^{-h}u_0^1(y(z))\big\rangle.
\end{equation}

For $0\le j\le a$, applying $D^j$ to
\eqref{primitive}, taking the $L^2(\rho\,\rmd z_2)$
inner product with $D^jV$, and integrating by parts
in $z_2$, we obtain
\[
 \frac12\frac{\rmd}{\rmd t}
 \|D^jV\|_{L^2(\rho\,\rmd z_2)}^2
 +\|D^{j+1}V\|_{L^2(\rho\,\rmd z_2)}^2
 =\int_{\mathbb R}D^jR\,D^jV\,\rho\,\rmd z_2.
\]
Summing over $0\le j\le a$, applying the Cauchy--Schwarz inequality, and integrating the resulting
inequality for the square root of the sum yield
\[
 \begin{aligned}
 &\Big(
 \sum_{j=0}^a\|D^jV(t)\|_{L^2(\rho\,\rmd z_2)}^2
 \Big)^{1/2}\le
 \Big(
 \sum_{j=0}^a\|D^jV(0)\|_{L^2(\rho\,\rmd z_2)}^2
 \Big)^{1/2}
 +\int_0^t
 \Big(
 \sum_{j=0}^a\|D^jR(\tau)\|_{L^2(\rho\,\rmd z_2)}^2
 \Big)^{1/2}\,\rmd\tau.
 \end{aligned}
\]
Taking the supremum over $0\le t\le T$ and using
\eqref{mean-norm-equivalence} and
\eqref{primitive-initial-bound}, we conclude that
\begin{equation}\label{primitive-bound}
 \|V\|_{L^\infty_t(H^a)}
 \lesssim
 \|V(0)\|_{H^a}+\|R\|_{L^1_t(H^a)}
 \lesssim
 \epsilon_0+\|R\|_{L^1_t(H^a)}.
\end{equation}

We now relate $V$ to $\pa_2\bY_L^1$.
Since $h$ and $\rho$ are independent of time,
the definition of $v$ and \eqref{primitive} give
\[
 \p_tD\Big(
 \frac{\langle\rme^{-h}\bY^1\rangle}{\rho}
 \Big)
 =Dv=D^2V=\p_tV-R.
\]
Integrating over $[0,t]$, we obtain
\begin{equation}\label{strain-identity}
 \begin{aligned}
 D\Big(
 \frac{\langle\rme^{-h}\bY^1(t)\rangle}{\rho}
 \Big)
 &=
 D\Big(
 \frac{\langle\rme^{-h}\bY^1(0)\rangle}{\rho}
 \Big)
 +V(t)-V(0)-\int_0^tR(\tau)\,\rmd\tau.
 \end{aligned}
\end{equation}
By \eqref{rho-identities} and
\eqref{corrector-identities},
\[
 \frac{\langle\rme^{-h}\bY^1\rangle}{\rho}
 =\langle(1+\p_1\widetilde Y^1)\bY^1\rangle=\bY_L^1
 +\big\langle\pa_1\tY^1\,\bY_H^1\big\rangle.
\]
Moreover, $\bY(0)=-\tY$ and periodicity in $z_1$ imply
\[
 \frac{\langle\rme^{-h}\bY^1(0)\rangle}{\rho}
 =-\big\langle(1+\pa_1\tY^1)\tY^1\big\rangle
 =-\langle\tY^1\rangle.
\]
Multiplying \eqref{strain-identity} by $\rho$ and
using these identities, we deduce that
\begin{equation}\label{recover-strain}
 \begin{aligned}
 \pa_2\bY_L^1(t)
 &=-\langle\pa_2\tY^1\rangle
 +\rho\Big(
 V(t)-V(0)-\int_0^tR(\tau)\,\rmd\tau
 \Big)
 -\pa_2\big\langle
 \pa_1\tY^1\,\bY_H^1(t)
 \big\rangle.
 \end{aligned}
\end{equation}
For the last term in \eqref{recover-strain}, integration by parts in $z_1$ yields
\[
 \pa_2\big\langle\pa_1\tY^1\,\bY_H^1\big\rangle
 =\big\langle
 \pa_1\tY^1\,\pa_2\bY_H^1
 -\pa_2\tY^1\,\pa_1\bY_H^1
 \big\rangle.
\]
Using this identity,
the Sobolev product estimate, and
\eqref{static-est3}, we obtain from
\eqref{recover-strain} that
\[
 \begin{aligned}
\|\pa_2\bY_L^1\|_{L^\infty_t(H^a)}
 &\lesssim
 \|\tY\|_{H^{a+1}}
 +\|V\|_{L^\infty_t(H^a)}
 +\|V(0)\|_{H^a}
 +\|R\|_{L^1_t(H^a)}\\
 &\quad
 +\|\nabla\tY\|_{H^a}
  \|\nabla\bY_H\|_{L^\infty_t(H^a)}\\
 &\lesssim
 \epsilon_0+\epsilon_0\HH(T)^{1/2}
 +\|R\|_{L^1_t(H^a)},
 \end{aligned}
\]
where the last inequality follows from
\eqref{static-est1}, \eqref{primitive-initial-bound},
\eqref{primitive-bound}, and \eqref{def-Ha}.
Squaring this estimate and using \eqref{flux-L1}
and \eqref{Sa} proves \eqref{strain-bound}.
\end{proof}

\section{Proof of the main theorems}\label{sec:8}

In this section, we establish local existence and
continuation in Eulerian and Lagrangian coordinates.
We then combine the energy estimates to prove
Theorems~\ref{th1} and~\ref{MTH}.

\subsection{Local existence and continuation}\label{app-local}

We use $H=b-\xi(x_2)e_1$, as in
\eqref{MHD-shear-Euler}.
Norms of Eulerian quantities are taken in $x$,
whereas norms of $\nabla_yX-I$ and $Y$ are taken
in $y$ and $z$, respectively.

\begin{proposition}\label{local-self-contained}
Let $u_0,H_0\in H^a(\T\times\R)$ be divergence-free,
and set $b_0=\xi(x_2)e_1+H_0$.
There exist constants $c,\eta_0>0$, depending only on
$a,m,L$, such that the following conclusions hold.
\begin{enumerate}
 \item[(i)]
 The Eulerian system \eqref{MHD-shear-Euler}
 admits a unique solution on $[0,\tau]$, where
 \[
  \tau=c\bigl(
  1+\|u_0\|_{H^a}+\|H_0\|_{H^a}
  \bigr)^{-2},
 \]
 satisfying
 \[
  u,H\in C([0,\tau];H^a),
  \qquad
  \nabla_xu\in L^2(0,\tau;H^a).
 \]
 Moreover, a solution on $[0,T^*)$, with
 $0<T^*<\infty$, extends beyond $T^*$ if
 \begin{equation}\label{eulerian-continuation}
  \sup_{t<T^*}
  \bigl(\|u(t)\|_{H^a}+\|H(t)\|_{H^a}\bigr)<\infty.
 \end{equation}

 \item[(ii)]
 If \eqref{asmp} holds with $\epsilon_0$ sufficiently
 small, then the Lagrangian system \eqref{equ-Y}
 admits a unique solution on the same interval,
 satisfying
 \[
  \p_tY,\nabla Y,\Db Y\in C([0,\tau];H^a),
  \qquad
  \nabla\p_tY\in L^2(0,\tau;H^a).
 \]
 Moreover, a solution in this class on $[0,T^*)$,
 with $0<T^*<\infty$, extends beyond $T^*$ if
 \begin{equation}\label{lagrangian-continuation}
  \sup_{t<T^*}\|\p_tY(t)\|_{H^a}<\infty,
  \qquad
  \sup_{t<T^*}\|\nabla Y(t)\|_{H^a}\le\eta_0.
 \end{equation}
\end{enumerate}
\end{proposition}

\begin{proof}
We first prove (i).
To construct solutions when $\xi$ is unbounded,
choose a smooth nonnegative function $\vartheta$
on $[0,\infty)$ such that $\vartheta(s)=s$ for
$0\le s\le1$ and $\vartheta$ is constant for $s\ge2$.
Set
\[
 \xi_n=m+n\vartheta\bigl((\xi-m)/n\bigr),
 \qquad n\ge1.
\]
Each $\xi_n$ is bounded, satisfies $\xi_n\ge m$,
and agrees with $\xi$ on every fixed compact interval
for all sufficiently large $n$.
The chain rule, \eqref{xi-cond0}, and the
one-dimensional Sobolev embedding give
\[
 \sup_{n\ge1}\|\xi_n'\|_{H^{a+1}}\le C(a,L).
\]
We consider the systems with backgrounds $\xi_n$
and the same perturbation data $(u_0,H_0)$.
The index $n$ is omitted in the energy calculation.

Applying $\pa_x^\alpha$ to
\eqref{MHD-shear-Euler}, taking the $L^2$ inner
products with $\pa_x^\alpha u$ and
$\pa_x^\alpha H$, and summing over $|\alpha|\le a$,
we obtain
\begin{equation}\label{local-energy}
 \begin{aligned}
 &\frac{\rmd}{\rmd t}
 \bigl(\|u\|_{H^a}^2+\|H\|_{H^a}^2\bigr)
 +2\|\nabla_xu\|_{H^a}^2\le
 C\bigl(1+\|u\|_{H^a}\bigr)
 \bigl(\|u\|_{H^a}^2+\|H\|_{H^a}^2\bigr).
 \end{aligned}
\end{equation}
Indeed, incompressibility eliminates the principal
transport and pressure terms, while the highest-order
magnetic terms cancel:
\[
 \bigl(
 b\cdot\nabla_x\pa_x^\alpha H
 \mid\pa_x^\alpha u
 \bigr)_{L^2}
 +
 \bigl(
 b\cdot\nabla_x\pa_x^\alpha u
 \mid\pa_x^\alpha H
 \bigr)_{L^2}
 =0.
\]
The Sobolev product and commutator estimates,
together with $H^a\hookrightarrow W^{1,\infty}$,
bound the nonlinear contributions by $C\|u\|_{H^a}^3
 +C\|u\|_{H^a}\|H\|_{H^a}^2$.
For the terms involving $\xi$, the bound
$\|\xi'\|_{W^{a,\infty}}\le C$ gives
\[
 \big\|[\pa_x^\alpha,\xi]\pa_{x_1}f\big\|_{L^2}
 +\|\pa_x^\alpha(\xi'f)\|_{L^2}
 \le C\|f\|_{H^a},
 \qquad f\in\{u,H\}.
\]
These terms contribute at most
$C(\|u\|_{H^a}^2+\|H\|_{H^a}^2)$. This proves \eqref{local-energy} with a constant $C$
depending only on $a$ and $L$, uniformly in $n$.
For the stated choice of $\tau$ and sufficiently small $c$,
a continuity argument and Gronwall's inequality yield
\begin{equation}\label{local-bound}
 \begin{aligned}
 &\sup_{0\le t\le\tau}
 \bigl(\|u(t)\|_{H^a}^2+\|H(t)\|_{H^a}^2\bigr)
 +\int_0^\tau\|\nabla_xu(t)\|_{H^a}^2\,\rmd t\le
 C\bigl(\|u_0\|_{H^a}^2+\|H_0\|_{H^a}^2\bigr),
 \end{aligned}
\end{equation}
where $C$ depends only on $a$ and $L$. For each bounded profile $\xi_n$, local existence and \eqref{local-bound} give a solution $(u_n,H_n)$ with data
$(u_0,H_0)$
on the same interval $[0,\tau]$.

Taking the divergence of the velocity equation and using
the boundedness of $\Delta_x^{-1}\pa_{x_i}\pa_{x_j}$
and $\Delta_x^{-1}\pa_{x_1}$ on $\T\times\R$, we obtain
\[
 \|p_n\|_{H^a}
 \le C\bigl(\|u_n\|_{H^a}^2+\|H_n\|_{H^a}^2
            +\|H_n\|_{H^a}\bigr).
\]
By \eqref{local-bound}, the equations give uniform
$L^2(0,\tau;H^{a-1})$ bounds for $\p_tu_n$ and $\p_tH_n$
on every bounded strip. Local compactness then yields,
along a subsequence, convergence in
$C([0,\tau];H^{a-1}_{\mathrm{loc}})$ to a solution
$(u,H)$ of \eqref{MHD-shear-Euler} satisfying
\eqref{local-bound}.

It remains to prove strong continuity in $H^a$.
Spatial regularization and cutoffs depending only on $x_2$
justify the $H^a$ energy identity for $(u,H)$. The principal
terms containing $\xi\pa_{x_1}$ cancel, and the cutoff
errors vanish by \eqref{local-bound} and the pressure estimate. In particular, \eqref{local-energy}
holds for the limiting solution.
Thus $\|u(t)\|_{H^a}^2+\|H(t)\|_{H^a}^2$ is continuous in
time, which, together with weak continuity in $H^a$, yields
$u,H\in C([0,\tau];H^a)$.
The uniqueness follows from the same localization applied to
the $L^2$ energy estimate for the difference of two solutions.

Under \eqref{eulerian-continuation}, the local
existence time for the data $(u(t_0),H(t_0))$
has a positive lower bound independent of $t_0<T^*$.
Choosing $t_0$ sufficiently close to $T^*$ and
using uniqueness on the common interval extends
the solution beyond $T^*$.
This proves (i).

We next prove (ii).
Assume \eqref{asmp}, and let $(u,H)$ be the solution
obtained in (i), with particle flow $X$ defined by
\eqref{flow}.
By \eqref{local-bound},
$u\in L^1(0,\tau;H^{a+1})$.
The Sobolev estimates for the flow give
\[
 X-\operatorname{Id}\in C([0,\tau];H^{a+1}),
 \qquad
 \det\nabla_yX=1.
\]
In particular, $X(t,\cdot)$ is a volume-preserving
diffeomorphism on $[0,\tau]$.
Reducing $c$ if necessary, the flow equation and
\eqref{local-bound} also give
\[
 \sup_{0\le t\le\tau}\|\nabla_yX(t)-I\|_{H^a}
 \le C\int_0^\tau\|\nabla_xu(s)\|_{H^a}\,\rmd s
 \le C\sqrt c.
\]
The flow estimates, \eqref{static-est3}, and
\eqref{app-composition} therefore apply uniformly
on this interval.

By the equivalence of the Eulerian and Lagrangian
formulations, the displacement
\[
 Y(t,z)=X(t,y(z))-y(z)
\]
solves \eqref{equ-Y} and satisfies
\[
 \p_tY,\nabla Y\in C([0,\tau];H^a),
 \qquad
 \nabla\p_tY\in L^2(0,\tau;H^a).
\]
To verify the regularity of $\Db Y$, we use
the frozen-field identity \eqref{eqbt}, which gives
\begin{equation}\label{eulerian-recovery}
 \begin{aligned}
 H(t,X(t,y(z)))
 ={}&H_0(y(z))+\Db Y(t,z)\\
 &+\bigl[
 \xi(y_2(z))-\xi(y_2(z)+Y^2(t,z))
 \bigr]e_1.
 \end{aligned}
\end{equation}
Since $Y(0)=0$,
\begin{equation}\label{local-shear-difference}
 \xi(y_2+Y^2(t))-\xi(y_2)
 =\int_0^t
 \xi'(y_2+Y^2(s))\p_tY^2(s)\,\rmd s.
\end{equation}
The Sobolev product and composition estimates,
\eqref{xi-cond0}, and the flow regularity show that
the right-hand side belongs to $C([0,\tau];H^a)$.
Thus \eqref{eulerian-recovery} implies
$ \Db Y\in C([0,\tau];H^a).$
Uniqueness follows from (i) through the inverse
change of variables.

It remains to prove the Lagrangian continuation
criterion.
Suppose that \eqref{lagrangian-continuation} holds
on $[0,T^*)$, where $T^*<\infty$.
By \eqref{static-est3},
\[
 \|\nabla(X(t)\circ y)-I\|_{H^a}
 \le\|\nabla y-I\|_{H^a}+\|\nabla Y(t)\|_{H^a}
 \le C\epsilon_0+\eta_0.
\]
For sufficiently small $\epsilon_0$ and $\eta_0$,
the map $X(t)\circ y$ is a diffeomorphism, and
Lemma~\ref{lemA} gives uniform bounds for composition
with this map and its inverse.
The recovered Eulerian velocity therefore satisfies
\[
 \sup_{t<T^*}\|u(t)\|_{H^a}
 \le C\sup_{t<T^*}\|\p_tY(t)\|_{H^a}<\infty.
\]
On every compact subinterval of $[0,T^*)$,
\eqref{eulerian-recovery} and
\eqref{local-shear-difference} place the recovered
fields in the solution class of (i).
Applying Gronwall's inequality to
\eqref{local-energy}, we obtain
\[
 \begin{aligned}
 \|u(t)\|_{H^a}^2+\|H(t)\|_{H^a}^2
 &\le
 \bigl(\|u_0\|_{H^a}^2+\|H_0\|_{H^a}^2\bigr)
 \exp\Big(
 C\int_0^t(1+\|u(s)\|_{H^a})\,\rmd s
 \Big).
 \end{aligned}
\]
Since $T^*<\infty$, this proves
\eqref{eulerian-continuation}.
Part~(i) consequently extends the Eulerian
solution beyond $T^*$.
The flow of the extended velocity agrees with
$X$ on $[0,T^*)$ by uniqueness.
Keeping the fixed map $y=y(z)$, we therefore extend
$Y(t,z)=X(t,y(z))-y(z)$ beyond $T^*$.
The flow estimates and
\eqref{eulerian-recovery}--\eqref{local-shear-difference}
give the regularity stated in (ii).
This completes the proof.
\end{proof}

\subsection{Global existence and stability}

We now combine the estimates for $\HH$, $\LL$,
$\WW$, and $\SSS$ to close the bootstrap argument.
The Lagrangian continuation criterion then gives
global existence, and the bound for $\WW$
yields the decay estimates.

\begin{proof}[Proof of Theorem~\ref{th1}]
Let $T^*\in(0,\infty]$ denote the maximal existence
time of the Lagrangian solution provided by
Proposition~\ref{local-self-contained}(ii).
For every finite $T<T^*$, the solution satisfies
\[
 \p_tY,\nabla Y,\Db Y\in C([0,T];H^a),
 \qquad
 \nabla\p_tY\in L^2(0,T;H^a).
\]
Since $\Db\bY_L=0$, \eqref{corrector-identities}
gives
\[
 \Db\bY_H=\Db Y+\Phi.
\]
Thus the four energy functionals are finite and
continuous on compact time intervals.

The a priori estimates apply to this solution
by spatial regularization and localization in $z_2$.
For the magnetic terms, the localized energy
identities are integrated in time before the
cutoffs are removed.
Since the cutoffs are independent of $z_1$,
they produce no additional terms under $\Db$.

The initial data in \eqref{equ-bY} give
$\bY(0)=-\tY$ and
$\p_t\bY(0)=u_0(y(z))$.
In particular,
\[
 \Db\bY_H(0)=-\Db\tY=\Phi.
\]
Using \eqref{static-est1},
\eqref{app-composition}, and \eqref{asmp},
we obtain
\begin{equation}\label{initial-energy}
 \HH(0)+\LL(0)+\WW(0)+\SSS(0)
 \le C_0\epsilon_0^2.
\end{equation}

Choose $\delta,\epsilon_*\in(0,1]$ sufficiently small
for \eqref{high-bound}, \eqref{weighted-bound},
\eqref{low-bound}, and \eqref{strain-bound} to hold
under \eqref{bootstrap}, and assume that
$0<\epsilon_0\le\epsilon_*$.
Fix a finite $T<T^*$ satisfying \eqref{bootstrap}.
Adding these four estimates and using
\eqref{initial-energy}, we obtain
\begin{equation}\label{ineq-1}
 \begin{aligned}
 &\HH(T)+\LL(T)+\WW(T)+\SSS(T)\\
 &\quad\le
 C\epsilon_0^2+C\HH(T)
 +C(\epsilon_0+\sqrt\delta+\delta)
 \bigl(\HH(T)+\LL(T)+\WW(T)+\SSS(T)\bigr).
 \end{aligned}
\end{equation}
Here $C\HH(T)$ comes from \eqref{weighted-bound}.
The terms containing $\epsilon_0^2\HH(T)$
are controlled using $\epsilon_0^2\le\epsilon_0$,
and the quadratic terms are bounded by
\[
 \bigl(\HH(T)+\LL(T)+\WW(T)\bigr)^2
 \le\delta\bigl(\HH(T)+\LL(T)+\WW(T)\bigr).
\]
Moreover, \eqref{high-bound} and
\eqref{initial-energy} imply
\begin{equation}\label{ineq-2}
 \HH(T)
 \le C\epsilon_0^2
 +C(\epsilon_0+\sqrt\delta)\HH(T).
\end{equation}

Fix a common constant $C$ for
\eqref{ineq-1} and \eqref{ineq-2}.
Reduce $\delta$ and then $\epsilon_*$ so that
\[
 C(\epsilon_*+\sqrt\delta+\delta)\le\frac12.
\]
Absorbing the last term in \eqref{ineq-2} gives
$\HH(T)\le2C\epsilon_0^2$.
Substitution into \eqref{ineq-1}, followed by
absorption of the term containing the total energy,
yields
\begin{equation}\label{bootstrap-improvement}
 \HH(T)+\LL(T)+\WW(T)+\SSS(T)
 \le K\epsilon_0^2,
\end{equation}
where $K>C_0$ depends only on $a,m,L$.

Decrease $\epsilon_*$ further so that
\[
 K\epsilon_*^2\le\frac{\delta}{2}.
\]
By \eqref{initial-energy} and continuity of the
energy functionals, \eqref{bootstrap} holds for all
sufficiently small positive $T$.
If the sum of the four energies first reached
$\delta$ at some finite $T<T^*$,
\eqref{bootstrap-improvement} would give an upper
bound of $\delta/2$, a contradiction.
Hence \eqref{bootstrap-improvement} holds for every
finite $T<T^*$.

We next verify the continuation criterion.
Since $Y=\tY+\bY_H+\bY_L$, we obtain from
\eqref{static-est1}, \eqref{vertical-est},
\eqref{def-Ha}, \eqref{def-La}, and \eqref{Sa} that
\begin{equation}\label{global-lagrangian-bound}
 \begin{aligned}
 &\|\p_tY\|_{L_t^\infty(H^a)}
 +\|\nabla Y\|_{L_t^\infty(H^a)}\\
 &\quad\le C\bigl(
 \epsilon_0+\HH(T)^{1/2}+\LL(T)^{1/2}
 +\SSS(T)^{1/2}+\HH(T)
 \bigr)\\
 &\quad\le C\epsilon_0,
 \qquad 0<T<T^*,
 \end{aligned}
\end{equation}
where the last inequality follows from
\eqref{bootstrap-improvement}.
Reduce $\epsilon_*$ further so that
$C\epsilon_*\le\eta_0$, with $C$ as in
\eqref{global-lagrangian-bound} and $\eta_0$ as in
Proposition~\ref{local-self-contained}.

Suppose that $T^*<\infty$.
Letting $T\uparrow T^*$ in
\eqref{global-lagrangian-bound}, we obtain
\[
 \sup_{t<T^*}\|\p_tY(t)\|_{H^a}<\infty,
 \qquad
 \sup_{t<T^*}\|\nabla Y(t)\|_{H^a}
 \le C\epsilon_0\le\eta_0.
\]
Thus \eqref{lagrangian-continuation} holds, and
Proposition~\ref{local-self-contained}(ii) extends
the solution beyond $T^*$, contradicting maximality.
Therefore $T^*=\infty$.
Estimate \eqref{bootstrap-improvement} proves
\eqref{main-bound}, and uniqueness follows from
Proposition~\ref{local-self-contained}(ii).

Finally, we prove \eqref{asympt}.
Since $\tY$ is independent of time,
\[
 u_L(t,z_2)
 =\big\langle u(t,X(t,y(z)))\big\rangle
 =\p_t\bY_L(t,z_2).
\]
The frozen-field identity \eqref{eqbt} and
\eqref{corrector-identities} give
\[
 \begin{aligned}
 u(t,X(t,y(z)))-u_L(t,z_2)
 &=\p_t\bY_H(t,z),\\
 b(t,X(t,y(z)))
 &=b_0(y(z))+\Db Y(t,z)\\
 &=\gamma(z_2)e_1+\Db\bY_H(t,z).
 \end{aligned}
\]
By \eqref{def-W0} and \eqref{main-bound},
\[
 \begin{aligned}
 \|\p_t\bY_H(t)\|_{L^2}
 +\|\Db\bY_H(t)\|_{L^2}
 &\le C(4a+t)^{-a/2}\WW(t)^{1/2}\le C\epsilon_0(4a+t)^{-a/2}.
 \end{aligned}
\]
This proves \eqref{asympt} and completes the proof.
\end{proof}

We now return to Eulerian coordinates to prove
Theorem~\ref{MTH}.

\begin{proof}[Proof of Theorem~\ref{MTH}]
Let $Y$ be the global solution given by
Theorem~\ref{th1}, and define $X$ by
\[
 X(t,y(z))=y(z)+Y(t,z).
\]
By \eqref{static-est3} and
\eqref{global-lagrangian-bound},
\[
 \|\nabla(X(t)\circ y)-I\|_{H^a}
 \le\|\nabla y-I\|_{H^a}+\|\nabla Y(t)\|_{H^a}
 \le C\epsilon_0,
 \qquad t\ge0.
\]
For sufficiently small $\epsilon_0$,
$X(t)\circ y$ is a global diffeomorphism.
We define $u$ and $b$ by
\[
 \begin{aligned}
 u(t,X(t,y(z)))&=\p_tY(t,z),\\
 b(t,X(t,y(z)))&=b_0(y(z))+\Db Y(t,z).
 \end{aligned}
\]
The equivalence of the Eulerian and Lagrangian
formulations shows that these fields solve
\eqref{A1} with initial data $(u_0,b_0)$.

Fix $T>0$.
Proposition~\ref{local-self-contained}(ii)
and the continuation argument give
\[
 \p_tY,\nabla Y,\Db Y\in C([0,T];H^a),
 \qquad
 \nabla\p_tY\in L^2(0,T;H^a).
\]
Since $Y(0)=0$, we also have
$Y\in C([0,T];H^{a+1})$.
By \eqref{local-shear-difference},
\[
 \xi(y_2+Y^2)-\xi(y_2)\in C([0,T];H^a).
\]
Applying the composition estimates to the formula
for $u$ and to \eqref{eulerian-recovery}, we obtain
\[
 u,\ b-\xi(x_2)e_1\in C([0,T];H^a).
\]

The chain rule gives
\[
 (\nabla_xu)(t,X(t,y(z)))
 =\nabla\p_tY(t,z)
  \bigl(\nabla y(z)+\nabla Y(t,z)\bigr)^{-1}.
\]
By \eqref{static-est3},
\eqref{global-lagrangian-bound}, the Sobolev product
estimate, and Lemma~\ref{lemA},
\[
 \|\nabla_xu\|_{L^2(0,T;H^a)}
 \le C\|\nabla\p_tY\|_{L^2(0,T;H^a)}
 <\infty.
\]
Since $T>0$ is arbitrary, this proves the asserted
global regularity.
The uniqueness follows from
Proposition~\ref{local-self-contained}(i).
\end{proof}

\section{The two dimensional torus case}\label{sec:torus-extension}

In this section, we prove Theorem~\ref{thm:torus-extension}
by adapting the arguments on $\T\times\R$.
We retain the notation and energy functionals introduced
in Sections~\ref{sec:2}--\ref{sec:7}, with spatial norms
taken on $\T^2$, or on $\T$ for functions depending
only on $z_2$.

We first use a Galilean transformation to reduce the
proof to initial velocities with zero mean.
We then show that the condition
$\int_{\T^2}b_0^2\,\rmd x=0$, together with
$\operatorname{div}b_0=0$, ensures that the fixed
coordinate map $y(z)$ and the corrector $\widetilde Y$ are compatible
with periodicity in both variables.
The pressure and low-frequency identities require
additional care, since constant functions now belong
to $L^2$.
We fix the pressure by a zero-mean normalization
and use conservation of momentum to determine
the constants in the low-frequency equations.
In particular, the weighted initial velocity admits
a periodic primitive controlled by the Poincar\'e
inequality, so \eqref{eulerian-low-assumption}
is unnecessary.

With these points established, the energy estimates
of Sections~\ref{sec:2}--\ref{sec:7} and the
continuation argument of Section~\ref{sec:8}
give global existence and the asserted bounds.
Finally, we undo the Galilean transformation
to recover the result for the original initial data.
\begin{proof}[Proof of Theorem~\ref{thm:torus-extension}]

\textbf{\textbf{Reduction to zero mean by a Galilean transformation.}}

We first reduce the proof to initial velocities with
zero mean. Set
\[
 u_{\mathrm{av}}:=\int_{\T^2}u_0\,\rmd x.
\]
Since $|\T^2|=1$, the Cauchy--Schwarz inequality and
orthogonality of the mean and mean-zero parts give
\[
 |u_{\mathrm{av}}|\le\|u_0\|_{L^2}\le\epsilon_0,
 \qquad
 \|u_0-u_{\mathrm{av}}\|_{H^a}^2
 =\|u_0\|_{H^a}^2-|u_{\mathrm{av}}|^2
 \le\epsilon_0^2.
\]
By Galilean invariance of \eqref{A1}, the transformation
\[
 u(t,x)\mapsto u(t,x+t u_{\mathrm{av}})-u_{\mathrm{av}},
 \qquad
 b(t,x)\mapsto b(t,x+t u_{\mathrm{av}})
\]
changes the initial data to $(u_0-u_{\mathrm{av}},b_0)$.
The fixed map $y=y(z)$ and the corrector $\tY$
remain unchanged, since they depend only on $b_0$.
The corresponding transformations are
\[
 \begin{aligned}
 X(t,y)&\mapsto X(t,y)-t u_{\mathrm{av}},
 &\qquad
 Y(t,z)&\mapsto Y(t,z)-t u_{\mathrm{av}},\\
 \bY_H(t,z)&\mapsto\bY_H(t,z),
 &\qquad
 \bY_L(t,z_2)&\mapsto\bY_L(t,z_2)-t u_{\mathrm{av}}.
 \end{aligned}
\]
Thus $\HH$, $\WW$, $\SSS$, and the dissipation term
in $\LL$ are unchanged. The original low-frequency
velocity satisfies
\[
 \|\p_t\bY_L\|_{L_t^\infty(H^a)}^2
 \le
 2\|\p_t\bY_L-u_{\mathrm{av}}\|_{L_t^\infty(H^a)}^2
 +2|u_{\mathrm{av}}|^2,
\]
where the first norm on the right is that of the
transformed solution.

Consequently, a global solution for the transformed data
gives a global solution for the original data by the
inverse transformation. The energy bound is preserved
after enlarging $C$, and the differences in
\eqref{asympt} are unchanged.
We therefore use the same notation for the transformed
data and solution and assume from now on that
\[
 \int_{\T^2}u_0\,\rmd x=0.
\]

\textbf{Periodic coordinate construction and pressure estimates.}

The divergence-free condition gives
\[
 \pa_{y_2}\int_\T b_0^2(y_1,y_2)\,\rmd y_1
 =-\int_\T\pa_{y_1}b_0^1(y_1,y_2)\,\rmd y_1=0.
\]
Together with $\int_{\T^2}b_0^2\,\rmd y=0$, this implies
\begin{equation}\label{torus-zero-flux}
 \int_\T b_0^2(y_1,y_2)\,\rmd y_1=0
 \qquad\text{for every }y_2\in\T.
\end{equation}
Extend $b_0$ periodically to $\R^2$ and define $y$
by \eqref{mapyz}.
Using \eqref{torus-zero-flux} in place of the
$L^2(\R)$ argument, the proof of
Proposition~\ref{prop:coordinate-periodicity} gives
\[
 y_2(z_1+1,z_2)=y_2(z_1,z_2),
 \qquad
 \pa_2y_2
 =\frac{b_0^1(0,z_2)}{b_0^1(y(z))}>0.
\]
Moreover, periodicity of $b_0$ and uniqueness for
\eqref{ode} imply
\[
 y_2(z_1,z_2+1)=y_2(z_1,z_2)+1.
\]
Since $y_1=z_1$, we obtain
\[
 y(z+k)=y(z)+k,
 \qquad k\in\mathbb Z^2.
\]
The positivity of $\pa_2y_2$ then shows that $y$
induces a diffeomorphism of $\T^2$.

Consequently, the coefficients in \eqref{short-coeff} are periodic.
The corrector is also periodic: for $\tY^1$, this
follows from
$\gamma\langle1/b_0^1(y(z))\rangle=1$;
for $\tY^2$, it follows from \eqref{tildeY2}.
Hence the identities in
Lemmas~\ref{geo-identities} and~\ref{constraint-lemma}
remain valid.

The Sobolev product, commutator, and composition
estimates used in the appendices hold in periodic
Sobolev spaces.
In particular, the proofs of Lemmas~\ref{lemA}
and~\ref{lem-g1g2} apply on $\T^2$, with changes of
variables performed on the torus. Thus \eqref{app-composition} and the coefficient bounds
\eqref{static-est3}, \eqref{static-est2}, and \eqref{static-est1}
hold on $\T^2$.
The proof of Lemma~\ref{coordinate-lemma} also
gives \eqref{equiva-1}, \eqref{equ-b0-norm}, and \eqref{magnetic-commutator-bound}, with constants depending only
on $a,m,L$.
The estimates \eqref{vertical-est},
\eqref{vertical-time-est},
\eqref{vertical-time-low}, and \eqref{vertical-tangent} follow as on $\T\times\R$.

For the a priori estimates, we consider a sufficiently
regular periodic solution on a finite interval $[0,T]$
and assume that \eqref{bootstrap} holds.
The associated particle flow, viewed on $\R^2$,
satisfies
\[
 X(t,y+k)=X(t,y)+k,
 \qquad k\in\mathbb Z^2.
\]

For the pressure, impose the normalization $\int_{\T^2}p\,\rmd z
 =0$ and define $q$ by \eqref{def-q}, using the inverse
Laplacian on mean-zero periodic functions.
Since the right-hand side of \eqref{eq-q} has zero
mean, this defines its unique mean-zero solution.
Integrating \eqref{eq-q} in $z_1$ gives
$
 \pa_2^2\langle q\rangle=0.
$
Periodicity implies that $\langle q\rangle$ is
constant, and the normalization
$\int_{\T^2}q\,\rmd z=0$ then gives
$\langle q\rangle=0$.

The Fourier multiplier bounds used in
Lemmas~\ref{pressure-equation-lemma}
and~\ref{pressure-lemma} remain valid for
frequencies in $\mathbb Z^2$.
Together with the periodic Sobolev product
estimates, they yield \eqref{q-properties},
\eqref{pressure-basic}, and
\eqref{pressure-remainder} by the same arguments.

\textbf{Estimates for $\LL$ and $\SSS$.}

We first estimate $\LL$ and establish \eqref{low-bound}.
The equation \eqref{weighted-low} and the norm equivalences
\eqref{D-norm-equivalence} and \eqref{mean-norm-equivalence}
remain valid on the torus.
The estimates \eqref{flux-L2} and \eqref{flux-L1}
also follow from the proof of Lemma~\ref{flux-lemma},
using the periodic Sobolev product estimates and
\[
 \|F\|_{L^\infty(\T^2)}
 \le C\|F\|_{L^2(\T^2)}^{1-1/a}
       \|F\|_{H^a(\T^2)}^{1/a}.
\]
Indeed, this interpolation inequality gives
\eqref{GN-time} by the same integration in time.
Moreover, periodic integration by parts yields
\[
 \int_\T(Df)g\,\rho\,\rmd z_2
 =-\int_\T f(Dg)\,\rho\,\rmd z_2.
\]
Thus the estimates for $v$ and $\p_t\bY_L^1$
in the proof of Proposition~\ref{low-energy-proposition}
apply without change and give
\eqref{horizontal-velocity} and \eqref{ptbarY1}.

To estimate $\p_t\bY_L^2$, we verify
\eqref{vertical-velocity} by determining its integration
constant through conservation of momentum.
Since the initial velocity has zero mean,
the change of variables and conservation of momentum give
\begin{equation}\label{torus-momentum}
 \int_{\T^2}\rme^{-h}\p_t\bY\,\rmd z
 =\int_{\T^2}u(t,x)\,\rmd x
 =\int_{\T^2}u_0(x)\,\rmd x=0.
\end{equation}
Integrating \eqref{geom:velocity-div} in $z_1$
gives \eqref{ptbarY2-diff} on $\T^2$.
Hence
$\big\langle\rme^{-h}(AB)_{i2}\p_t\bY^i\big\rangle$
is independent of $z_2$.
To determine its value, note that $X^2\circ y-z_2$
is periodic and
\[
 AB\nabla(X^2\circ y)
 =[\nabla(X\circ y)]^{-\top}
  [\nabla(X\circ y)]^\top e_2=e_2.
\]
Periodic integration by parts, together with
\eqref{geom:velocity-div} and \eqref{torus-momentum},
therefore gives
\begin{align*}
 &\int_{\T^2}\rme^{-h}(AB)_{i2}\p_t\bY^i\,\rmd z=
 \int_{\T^2}\rme^{-h}(AB)^\top\p_t\bY
 \cdot e_2\,\rmd z\\
 &=
 \int_{\T^2}\rme^{-h}(AB)^\top\p_t\bY
 \cdot\nabla(X^2\circ y)\,\rmd z-
 \int_{\T^2}\rme^{-h}(AB)^\top\p_t\bY
 \cdot\nabla(X^2\circ y-z_2)\,\rmd z\\
 &=
 \int_{\T^2}\rme^{-h}\p_t\bY
 \cdot\bigl(AB\nabla(X^2\circ y)\bigr)\,\rmd z+
 \int_{\T^2}
 \diver\bigl(\rme^{-h}(AB)^\top\p_t\bY\bigr)
 (X^2\circ y-z_2)\,\rmd z\\
 &=\int_{\T^2}\rme^{-h}\p_t\bY^2\,\rmd z=\int_{\T^2}u^2(t,x)\,\rmd x
 =\int_{\T^2}u_0^2(x)\,\rmd x=0.
\end{align*}
Consequently,
\[
 \big\langle
 \rme^{-h}(AB)_{i2}\p_t\bY^i
 \big\rangle=0.
\]
Expanding this identity by \eqref{geom:AB} and
\eqref{corrector-identities}, and using periodicity
in $z_1$, we obtain
\[
 \p_t\bY_L^2
 =\big\langle
 \p_t\bY_H^1\,\pa_1\bY_H^2
 -(\pa_1\tY^1+\pa_1\bY_H^1)\p_t\bY_H^2
 \big\rangle.
\]
This proves \eqref{vertical-velocity} on $\T^2$.
Together with \eqref{vertical-time-est}, it gives
\eqref{ptbarY2} by the same Sobolev product estimates
as in Proposition~\ref{low-energy-proposition}.
Combining \eqref{ptbarY1} and \eqref{ptbarY2}
proves \eqref{low-bound}.

We next estimate $\SSS$ and establish \eqref{strain-bound}.
The only change needed in the construction of $V$
is the choice of its initial value.
By \eqref{torus-momentum} at $t=0$,
\[
 \int_\T
 \big\langle\rme^{-h}u_0^1(y(z))\big\rangle
 \,\rmd z_2=0.
\]
We may therefore prescribe $V(0)$ as the unique
periodic function satisfying
\[
 \pa_2V(0)
 =\big\langle\rme^{-h}u_0^1(y(z))\big\rangle,
 \qquad
 \int_\T V(0)\,\rmd z_2=0.
\]
The periodic Poincar\'e inequality in $z_2$, \eqref{flux-average-bound},
\eqref{static-est3}, and \eqref{app-composition} give
\[
 \begin{aligned}
 \|V(0)\|_{H^{a+1}}
 &\le C\big\|
 \big\langle\rme^{-h}u_0^1(y(z))\big\rangle
 \big\|_{H^a}\le C\|u_0\|_{H^a}
 \le C\epsilon_0.
 \end{aligned}
\]
Thus \eqref{primitive-initial-bound} holds without
the assumption \eqref{eulerian-low-assumption}.

Define $V$ by \eqref{def-V}.
The estimates \eqref{horizontal-velocity},
\eqref{D-norm-equivalence}, and \eqref{flux-L1}
ensure that the time integral is well defined in $H^a$.
Since $DV(0)=v(0)$, applying $D$ to \eqref{def-V}
and using \eqref{weighted-low} gives
\[
 DV=v,
 \qquad
 \p_tV-D^2V=R.
\]
The energy estimate used to prove
\eqref{primitive-bound} applies on $\T$ and yields
\[
 \|V\|_{L_t^\infty(H^a)}
 \le C\bigl(
 \|V(0)\|_{H^a}+\|R\|_{L_t^1(H^a)}
 \bigr)
 \le C\bigl(
 \epsilon_0+\|R\|_{L_t^1(H^a)}
 \bigr).
\]

The identities \eqref{strain-identity} and
\eqref{recover-strain} remain valid.
Following the estimate of \eqref{recover-strain}
in Proposition~\ref{horizontal-deformation-proposition},
we obtain
\[
 \begin{aligned}
 \|\pa_2\bY_L^1\|_{L_t^\infty(H^a)}
 &\le C\Big(
 \|\tY\|_{H^{a+1}}
 +\|V\|_{L_t^\infty(H^a)}
 +\|V(0)\|_{H^a}
 +\|R\|_{L_t^1(H^a)}\\
 &\qquad\quad
 +\|\nabla\tY\|_{H^a}
  \|\nabla\bY_H\|_{L_t^\infty(H^a)}
 \Big)\\
 &\le C\bigl(
 \epsilon_0+\epsilon_0\HH(T)^{1/2}
 +\|R\|_{L_t^1(H^a)}
 \bigr),
 \end{aligned}
\]
where we used \eqref{static-est1} and \eqref{def-Ha}.
Squaring this estimate and applying \eqref{flux-L1}
and \eqref{Sa} proves \eqref{strain-bound}.

\textbf{Global existence and asymptotic behavior.}

The horizontal Poincar\'e inequality and the
$L^2$-orthogonality of the low- and high-frequency
parts remain valid on $\T^2$.
Together with the periodic Sobolev estimates,
\eqref{geometry-small}, \eqref{pressure-basic},
\eqref{pressure-remainder}, and \eqref{GN-time},
they allow us to follow the proofs of
Propositions~\ref{high-linear-prop}, \ref{high-prop},
and~\ref{weighted-prop}.
Hence \eqref{high-bound} and \eqref{weighted-bound},
as well as \eqref{low-bound} and \eqref{strain-bound},
hold whenever \eqref{bootstrap} is satisfied.

The proof of Proposition~\ref{local-self-contained}
also applies on $\T^2$, using the corresponding
periodic Sobolev estimates.
It gives local existence and the continuation
criterion \eqref{lagrangian-continuation}.
Let $T^*\in(0,\infty]$ denote the maximal existence
time of this Lagrangian solution.
The initial data satisfy \eqref{initial-energy},
namely,
\[
 \HH(0)+\LL(0)+\WW(0)+\SSS(0)
 \le C_0\epsilon_0^2.
\]
Choosing $\delta$ and $\epsilon_*$ sufficiently small,
the absorption and continuity argument in the proof
of Theorem~\ref{th1} gives
\[
 \HH(T)+\LL(T)+\WW(T)+\SSS(T)
 \le C\epsilon_0^2
\]
for every finite $T<T^*$, with $C$ independent of $T$.
The derivation of \eqref{global-lagrangian-bound}
therefore applies and yields
\[
 \|\p_tY\|_{L_t^\infty(H^a)}
 +\|\nabla Y\|_{L_t^\infty(H^a)}
 \le C\epsilon_0,
 \qquad 0<T<T^*.
\]
Taking $\epsilon_*$ smaller if necessary, these
uniform bounds imply \eqref{lagrangian-continuation}.
The periodic version of
Proposition~\ref{local-self-contained}(ii)
then excludes a finite maximal existence time.
Thus $T^*=\infty$, and \eqref{main-bound} holds
for every finite $T>0$.

The inverse change of variables gives the asserted
Eulerian regularity by the same argument as in the
proof of Theorem~\ref{MTH}.
Uniqueness follows from the periodic version of
Proposition~\ref{local-self-contained}.
Moreover, \eqref{eqbt} and
\eqref{corrector-identities} give
\[
 \begin{aligned}
 u(t,X(t,y(z)))-u_L(t,z_2)
 &=\p_t\bY_H(t,z),\\
 b(t,X(t,y(z)))-\gamma(z_2)e_1
 &=\Db\bY_H(t,z).
 \end{aligned}
\]
Combining these identities with \eqref{def-W0}
and \eqref{main-bound} proves \eqref{asympt}.

Finally, applying the inverse Galilean transformation
gives a global solution with the original initial
data $(u_0,b_0)$.
As shown at the beginning of the proof,
\eqref{main-bound} remains valid after enlarging $C$,
and the decay estimates \eqref{asympt} are unchanged.
This completes the proof.

\end{proof}

\section*{Acknowledgment}
The authors would like to thank Professor Fang-Hua Lin for his invaluable insights and stimulating discussions
that helped improve the earlier version of this paper.
Cai was supported by NSFC grant (no. 12571247). Han was supported by the startup fund of Donghua University.  Zhao was supported in part by NSFC grants (no. 12301256 and no. 12471215).

%
%
%


\appendix

\refstepcounter{section}
\section*{Appendix \thesection}
\addcontentsline{toc}{section}{Appendix \thesection}
\label{app:A}

In this appendix, we discuss an alternative approach to introducing Lagrangian coordinates, see e.g. \cite{LXZ1,XZ,CHZ}.
To illustrate the main ideas, we only focus on the two dimensional case. Let $(\xi(x_2),0)^\top$ be the background magnetic field bounded below by a positive constant.  Let $b_0=(b_0^1,b_0^2)^\top$ be close to $(\xi(x_2),0)^\top$. Define the trajectory $X(t,y)$ by
\begin{equation*}
\begin{cases}
\f{\rmd}{\rmd t} X(t,y)=u(t,X(t,y)),\\
X(0,y)=X_0(y),
\end{cases} y\in \bR^2.
\end{equation*}
Let $\phi(y)=\xi(X_0^2(y))$.  For $A=(\nabla_{y}X)^{-\top}$, one may choose $X_0(y)$ such that
\begin{align}\label{eq-com}
A^\top_0(y)b_0(X_0(y))=\phi(y) e_1.
\end{align}
Then by \eqref{eqbd}, we have
\begin{equation}\label{appd-b-exp}
b(t,X(t,y))=\phi(y)\nabla_y Xe_1=\phi(y)\p_{y_1} X.
\end{equation}
The function $\phi(y)$ should be regarded as the perturbation of $\xi(y_2)$ since $X_0^2(y)$ is assumed to be close to $y_2$. In this way, it seems that one can  express $b$ in terms of the directional derivative of $X$ with respect to $y_1$. 
 However, the following lemma shows that
\eqref{eq-com}, together with volume preservation,
requires the compatibility condition
\eqref{rest-xi}.

The issue is the existence of an initial map $X_0(y)$ that satisfies \eqref{eq-com} together with the volume-preserving condition
$$\det(\nabla_y X_0(y))=1.$$
Inspired by \cite{LXZ1}, we define
\begin{align*}
\begin{split}
U_0(x)=\Big(\begin{matrix}
\xi^{-1}b_0^1 & f^1\\\xi^{-1}b_0^2  & f^2
\end{matrix}\Big),
\end{split}
\end{align*}
with $f(x)=( f^1(x), f^2(x))^\top$ being some smooth function defined in $\bR^2$. The initial map $X_0(y)$ and $U_0$ should satisfy
\begin{align}\label{app0}
U_0\circ X_0(y)=\na_y X_0(y),\quad \det U_0=1.
\end{align}
\begin{lemma}
Let $X_0:\R^2\to\R^2$ be a smooth diffeomorphism,
let $b_0$ and $f$ be smooth vector fields, and let
$\xi>0$ be a smooth function.
Assume that \eqref{app0} holds and that
$\operatorname{div}_x b_0=0$. Then
\begin{equation}\label{rest-xi}
 b_0\cdot\nabla_x\xi=0,\quad \operatorname{div}_x f=0.
\end{equation}
\end{lemma}
\begin{proof}
By \eqref{app0},
\[
 (U_0e_j)\circ X_0=\partial_{y_j}X_0,
 \qquad j=1,2,
 \qquad
 \det\nabla_yX_0=1.
\] Differentiating
$(U_0e_j)\circ X_0=\partial_{y_j}X_0$
and commuting mixed derivatives, we obtain
\[
 \bigl((\nabla_x(U_0e_j))\circ X_0\bigr)\nabla_yX_0
 =\partial_{y_j}\nabla_yX_0.
\]
Taking the trace after multiplication by
$(\nabla_yX_0)^{-1}$ and using Jacobi's formula, we get
\[
 \begin{aligned}
 \bigl(\operatorname{div}_x(U_0e_j)\bigr)\circ X_0
 &=\operatorname{tr}\bigl(
   (\partial_{y_j}\nabla_yX_0)
   (\nabla_yX_0)^{-1}
   \bigr)\\
 &=\partial_{y_j}\log\det\nabla_yX_0
 =0,
 \qquad j=1,2.
 \end{aligned}
\]
Composing with $X_0^{-1}$, we conclude that
both columns of $U_0$ are divergence-free. In particular, the first column
and $\operatorname{div}_x b_0=0$ yield
\[
 0=\operatorname{div}_x\Big(\frac{b_0}{\xi}\Big)
 =\frac{\operatorname{div}_x b_0}{\xi}
  -\frac{b_0\cdot\nabla_x\xi}{\xi^2} \,.
\]
This proves $\eqref{rest-xi}_1$.
The second column yields $\eqref{rest-xi}_2$.
\end{proof}
\begin{remark}
The lemma also holds in any dimension $n\ge2$.
Let $X_0:\R^n\to\R^n$ be a smooth diffeomorphism,
and write
\[
 U_0=\bigl(b_0/\xi,f_2,\ldots,f_n\bigr),
\]
where $\xi>0$ and the vector fields are smooth.
Assume that \eqref{app0} holds and that
$\operatorname{div}_x b_0=0$.
The same argument shows that every column of $U_0$
is divergence-free. Consequently,
\[
 b_0\cdot\nabla_x\xi=0,
 \qquad
 \operatorname{div}_x f_j=0,
 \quad j=2,\ldots,n,
\]
which generalizes both identities in
\eqref{rest-xi}.
\end{remark}
\refstepcounter{section}
\section*{Appendix \thesection}
\addcontentsline{toc}{section}{Appendix \thesection}
\label{app:B}

In this appendix, we justify the global coordinate
construction in \eqref{mapyz} and prove the periodicity
of the coefficients and transformed unknowns.

\begin{proposition}\label{prop:coordinate-periodicity}
Let $b_0=(b_0^1,b_0^2)^\top$ be a smooth vector field
on $\T\times\R$ satisfying
\[
 \operatorname{div}_y b_0=0,
 \qquad
 0<m\le b_0^1,
 \qquad
 b_0^2\in L^2(\T\times\R).
\]
Then the initial value problem \eqref{ode}, considered
on $\R^2$, defines a unique smooth map $y:\R^2\to\R^2$
satisfying
\begin{equation}\label{eq:map-periodicity}
 y_2(z_1+1,z_2)=y_2(z_1,z_2),
 \qquad
 y(z_1+1,z_2)=y(z_1,z_2)+e_1.
\end{equation}
Moreover,
\begin{equation}\label{eq:jacobian}
 \det\nabla y=\pa_2y_2
 =\frac{b_0^1(0,z_2)}{b_0^1(y(z))}>0.
\end{equation}
The map $y$ is a global diffeomorphism and induces
a diffeomorphism of $\T\times\R$.
If $f$ is periodic in $y_1$, then $f\circ y$
is periodic in $z_1$.
\end{proposition}

\begin{proof}
The divergence-free condition and periodicity give
\[
 \frac{\rmd}{\rmd y_2}
 \int_\T b_0^2(y_1,y_2)\,\rmd y_1
 =-\int_\T\partial_{y_1}b_0^1(y_1,y_2)\,\rmd y_1=0.
\]
By the Cauchy--Schwarz inequality,
\[
 \Big\|
 \int_\T b_0^2(y_1,\cdot)\,\rmd y_1
 \Big\|_{L^2(\R)}
 \le\|b_0^2\|_{L^2(\T\times\R)}<\infty.
\]
Thus this integral is constant in $y_2$ and belongs
to $L^2(\R)$, so
\begin{equation}\label{eq:zero-flux}
 \int_\T b_0^2(y_1,y_2)\,\rmd y_1=0,
 \qquad y_2\in\R.
\end{equation}

Define the stream function
\begin{equation*}
 \Psi(y_1,y_2)
 :=\int_0^{y_2}b_0^1(y_1,s)\,\rmd s
   -\int_{0}^{y_1}b_0^2(r,0)\,\rmd r.
\end{equation*}
Using $\operatorname{div}_y b_0=0$, we obtain
\begin{equation}\label{eq:stream-second}
 \partial_{y_1}\Psi=-b_0^2,
 \qquad
 \partial_{y_2}\Psi=b_0^1\ge m.
\end{equation}
Periodicity of $b_0$ and \eqref{eq:zero-flux} imply
\begin{equation}\label{eq:stream-periodicity}
 \Psi(y_1+1,y_2)-\Psi(y_1,y_2)
 =-\int_{y_1}^{y_1+1}b_0^2(r,0)\,\rmd r=0.
\end{equation}

For each fixed $y_1$, the bound
$\partial_{y_2}\Psi\ge m$ shows that
$s\mapsto\Psi(y_1,s)$ is a strictly increasing
bijection from $\R$ onto $\R$.
Hence, for every $(z_1,z_2)\in\R^2$, there is a unique
$y_2(z_1,z_2)$ satisfying
\begin{equation}\label{eq:implicit-coordinate}
 \Psi\bigl(z_1,y_2(z_1,z_2)\bigr)=\Psi(0,z_2).
\end{equation}
The implicit function theorem gives the smoothness
of $y_2$.
Setting $y_1=z_1$ and differentiating
\eqref{eq:implicit-coordinate} in $z_1$, we find
\[
 -b_0^2(y(z))
 +b_0^1(y(z))\pa_1y_2(z)=0.
\]
At $z_1=0$, strict monotonicity gives
$y_2(0,z_2)=z_2$.
Thus $y$ solves \eqref{ode} globally.
Conversely, \eqref{ode} and \eqref{eq:stream-second}
show that $\Psi(z_1,y_2(z_1,z_2))$ is constant
along any solution of \eqref{ode}.
Every such solution therefore satisfies
\eqref{eq:implicit-coordinate}, which proves uniqueness.

By \eqref{eq:stream-periodicity} and
\eqref{eq:implicit-coordinate},
\[
 \begin{aligned}
 \Psi\bigl(z_1,y_2(z_1+1,z_2)\bigr)
 &=\Psi\bigl(z_1+1,y_2(z_1+1,z_2)\bigr)\\
 &=\Psi(0,z_2)
 =\Psi\bigl(z_1,y_2(z_1,z_2)\bigr).
 \end{aligned}
\]
Strict monotonicity in the second variable yields
$y_2(z_1+1,z_2)=y_2(z_1,z_2)$.
Together with $y_1=z_1$, this proves
\eqref{eq:map-periodicity}.

Differentiating \eqref{eq:implicit-coordinate}
in $z_2$ gives
\[
 b_0^1(y(z))\pa_2y_2(z)=b_0^1(0,z_2),
\]
which proves \eqref{eq:jacobian}.
For each $(y_1,y_2)\in\R^2$, the inverse is determined by
$z_1=y_1$ and
\[
 \Psi(0,z_2)=\Psi(y_1,y_2).
\]
By \eqref{eq:stream-second}, the map
$s\mapsto\Psi(0,s)$ is a smooth increasing bijection
of $\R$ onto itself, with derivative bounded below by $m$.
Thus the inverse is globally defined and smooth,
and $y$ is a global diffeomorphism.
The translation identity \eqref{eq:map-periodicity}
shows that it induces a diffeomorphism of $\T\times\R$.
Finally, if $f$ is periodic in $y_1$, then
\[
 f(y(z_1+1,z_2))
 =f(y(z_1,z_2)+e_1)
 =f(y(z_1,z_2)).
\]
This completes the proof.
\end{proof}

\begin{corollary}\label{cor:corrector}
Under the assumptions of
Proposition~\ref{prop:coordinate-periodicity},
$b_0\circ y$, the coefficients $h$, $B$, $\rho$,
$\gamma$, and $\Phi$ defined in \eqref{pz2y2},
\eqref{def-B}, and \eqref{short-coeff}, and the
corrector $\tY$ defined in \eqref{DefBaY}
are periodic in $z_1$.
\end{corollary}

\begin{proof}
The periodicity of $b_0\circ y$ follows from
Proposition~\ref{prop:coordinate-periodicity}.
By \eqref{pz2y2},
\[
 \rme^{-h}=\pa_2y_2>0.
\]
Since $\pa_2y_2$ is periodic in $z_1$ by
\eqref{eq:map-periodicity}, so is $h$.
The periodicity of $B$ follows from
$B=(\nabla y)^{-\top}$ and
\eqref{eq:map-periodicity}.
The functions $\rho$ and $\gamma$ depend only on
$z_2$, and \eqref{short-coeff} gives the periodicity
of $\Phi$.

For the first component of the corrector,
\eqref{DefBaY} and the definition of $\gamma$ give
\[
 \begin{aligned}
 \tY^1(z_1+1,z_2)-\tY^1(z_1,z_2)
 &=\gamma(z_2)\int_{z_1}^{z_1+1}
 \frac{1}{b_0^1(y(s,z_2))}\,\rmd s-1=0.
 \end{aligned}
\]
For the second component, \eqref{tildeY2} gives
\[
 \tY^2(z)=\langle y_2\rangle(z_2)-y_2(z),
\]
whose periodicity follows directly from
\eqref{eq:map-periodicity}.
\end{proof}

\begin{corollary} 
Under the assumptions of
Proposition~\ref{prop:coordinate-periodicity},
let $u$ be a smooth velocity field periodic in $x_1$,
and suppose that its flow \eqref{flow} exists
on $[0,T]$.
Then the transformed displacements $Y(t,z)$
and $\bY(t,z)$ are periodic in $z_1$.
Moreover, if $f(t,x)$ is periodic in $x_1$, then
$f(t,X(t,y(z)))$ is periodic in $z_1$.
\end{corollary}

\begin{proof}
By periodicity of $u$, the maps
$X(t,y+e_1)-e_1$ and $X(t,y)$ solve the same
initial value problem.
Uniqueness therefore gives
\[
 X(t,y+e_1)=X(t,y)+e_1.
\]
In particular, $X(t,y)-y$ is periodic in $y_1$.
Combining this identity with
\eqref{eq:map-periodicity}, we obtain
\[
 \begin{aligned}
 X(t,y(z_1+1,z_2))
 &=X(t,y(z_1,z_2)+e_1)\\
 &=X(t,y(z_1,z_2))+e_1.
 \end{aligned}
\]
Since $Y(t,z)=X(t,y(z))-y(z)$, the displacement
$Y(t,z)$ is periodic in $z_1$.
Corollary~\ref{cor:corrector} then gives the
periodicity of $\bY=Y-\tY$.
For periodic $f$, the same translation identity yields
\[
 f(t,X(t,y(z_1+1,z_2)))
 =f(t,X(t,y(z_1,z_2))),
\]
as required.
\end{proof}

\begin{remark}
The assumption $b_0^2\in L^2(\T\times\R)$ is used
only to prove \eqref{eq:zero-flux} and may be replaced
by that identity.
Without \eqref{eq:zero-flux},
\eqref{eq:map-periodicity} may fail.
For example, if $b_0=(1,1)^\top$, then
\[
 y_2(z_1,z_2)=z_2+z_1,
 \qquad
 y_2(z_1+1,z_2)-y_2(z_1,z_2)=1.
\]
\end{remark}

\refstepcounter{section}
\section*{Appendix \thesection}
\addcontentsline{toc}{section}{Appendix \thesection}
\label{app:C}
In this appendix, we prove several useful lemmas. The first lemma is the estimate of the composite map which is
a generalized version of Lemma A.1 in \cite{LXZ1}.
\begin{lemma}\label{lemA}
Let $\Omega_1,\, \Omega_2\subseteq\bR^2$ be smooth Sobolev extension domains, with all embedding constants absorbed into the implicit constant. Let $$\Gamma:\Omega_2\rightarrow\Omega_1,
	\quad
	y\mapsto x=\Gamma(y),
$$
be a smooth diffeomorphism satisfying $\det(\nabla_y \Gamma)\neq 0$, $ \nabla_y \Gamma\in L^\infty$ and $\frac{1}{ \det(\nabla_y \Gamma)}\in L^\infty$. Denote by $\Gamma^{-1}$  its inverse mapping.
Then for any smooth functions $f:\Omega_1 \rightarrow \bR$ and $g:\Omega_2 \rightarrow \bR$,
the norms $\|f\circ\Gamma\|_{H^s(\Omega_2)}$ and $\|g\circ\Gamma^{-1}\|_{H^s(\Omega_1)}$ satisfy
\begin{itemize}
\item[(i)] if $s=0$,
\begin{align}\label{lemA-0}
\begin{split}
\|f\circ\Gamma\|_{L^2(\Omega_2)}&\leq \|\frac{1}{\det(\nabla_y\Gamma)}\|_{L^\infty(\Omega_2)}^{\frac12}\|f\|_{L^2(\Omega_1)},\\ \|g\circ\Gamma^{-1}\|_{L^2(\Omega_1)}&\leq \|{\det(\nabla_y\Gamma)}\|_{L^\infty(\Omega_2)}^\frac12\|g\|_{L^2(\Omega_2)}.
\end{split}
\end{align}
\item[(ii)] if $s=1$,
\begin{align}\label{lemA-1}
\begin{split}
\|f\circ\Gamma\|_{{H}^1(\Omega_2)}&\lesssim \big(1+\|\nabla_y\Gamma\|_{L^\infty(\Omega_2)}\big)\|\frac{1}{\det(\nabla_y\Gamma)}\|_{L^\infty(\Omega_2)}^\frac12 \|f\|_{{H}^1(\Omega_1)},\\
\|g\circ\Gamma^{-1}\|_{{H}^1(\Omega_1)}&\lesssim \big(1+\|\frac{1}{\det(\nabla_y\Gamma)}\|_{L^\infty(\Omega_2)}\|\nabla_y\Gamma\|_{L^\infty(\Omega_2)}\big)\\&\qquad\times\|\det(\nabla_y\Gamma)\|_{L^\infty(\Omega_2)}^\frac12 \|g\|_{{H}^1(\Omega_2)}.	
\end{split}
\end{align}
\item[(iii)] if $s=2$,
\begin{align}\label{lemA-2}
\begin{split}
\|f\circ\Gamma\|_{{H}^2(\Omega_2)}\lesssim&\big(1+\|\nabla_y\Gamma\|_{{L}^{\infty}(\Omega_2)}^2+\|\nabla_y^2\Gamma\|_{{H}^{1}(\Omega_2)}^2\big)\\&\times\|\frac{1}{\det(\nabla_y\Gamma)}\|_{L^\infty(\Omega_2)}^\frac12 \|f\|_{{H}^2(\Omega_1)},	\\
\|g\circ\Gamma^{-1		}\|_{{H}^2(\Omega_1)}\lesssim&\big(1+\|\frac{1}{\det(\nabla_y\Gamma)}\|_{L^\infty(\Omega_2)}^3\big)\big(1+\|\nabla_y\Gamma\|_{{L}^{\infty}(\Omega_2)}^4+\|\nabla_y^2\Gamma\|_{{H}^{1}(\Omega_2)}^4\big)\\&\times\|\det(\nabla_y\Gamma)\|_{L^\infty(\Omega_2)}^\frac12\|g\|_{{H}^2(\Omega_2)}.
\end{split}
\end{align}
\item[(iv)] if $s\geq 3$ is an integer,
\begin{align}\label{lemA-3}
\begin{split}
\|f\circ\Gamma\|_{{H}^s(\Omega_2)}\lesssim&\big(1+\|\nabla_y\Gamma\|_{{L}^{\infty}(\Omega_2)}^s+\|\nabla_y^2\Gamma\|_{{H}^{s-2}(\Omega_2)}^s\big)\\&\times\|\frac{1}{\det(\nabla_y\Gamma)}\|_{L^\infty(\Omega_2)}^\frac12 \|f\|_{{H}^s(\Omega_1)},\\
\|g\circ\Gamma^{-1}\|_{{H}^s(\Omega_1)}\lesssim& \big(1+\|\nabla_y\Gamma\|_{{L}^{\infty}(\Omega_2)}^{5s^3+4s^2+s}+\|\nabla_y^2\Gamma\|_{{H}^{s-2}(\Omega_2)}^{5s^3+4s^2+s}\big)\\&\times\big(1+\|\det{\nabla_y\Gamma}\|_{L^{\infty}(\Omega_2)}^{\frac{s}{2}}\big) \big(1+\|\frac{1}{\det(\nabla_y\Gamma)}\|_{L^\infty(\Omega_2)}^{5s^3+4s^2+s}\big)\\&\times \|\det(\nabla_y\Gamma)\|_{L^\infty(\Omega_2)}^\frac12 \|g\|_{{H}^s(\Omega_2)}.	
\end{split}
\end{align}
\end{itemize}
\end{lemma}
\begin{proof}
\textbf{Case (i): $s=0$.}
It is easy to check that
\begin{align}\label{lem-eq1}
\begin{split}
\nabla_x\Gamma^{-1}=(\nabla_y\Gamma)^{-1}\circ\Gamma^{-1},\ \det(\nabla_x\Gamma^{-1})=\frac{1}{\det\Big((\nabla_y\Gamma)\circ\Gamma^{-1}\Big)}.
\end{split}
\end{align}
We get from \eqref{lem-eq1} that
\begin{align*}
&\|f\circ\Gamma\|_{L^2(\Omega_2)}^2
=\int_{\Omega_2}|f(\Gamma(y))|^2\rmd y=\int_{\Omega_1}|f(x)|^2|\det(\nabla_x\Gamma^{-1})|\rmd x\\
&=\int_{\Omega_1}|f(x)|^2\frac{1}{|\det\Big((\nabla_y\Gamma)\circ\Gamma^{-1}\Big)|}\rmd x
\leq \|\frac{1}{\det(\nabla_y\Gamma)}\|_{L^\infty(\Omega_2)}\|f\|_{L^2(\Omega_1)}^2.
\end{align*}
Similarly,
\begin{align*}
&\|g\circ\Gamma^{-1}\|_{L^2(\Omega_1)}^2=\int_{\Omega_1}|g(\Gamma^{-1}(x))|^2\rmd x\\
&=\int_{\Omega_2}|g(y)|^2|\det(\nabla_y\Gamma)|\rmd y\leq \|{\det(\nabla_y\Gamma)}\|_{L^\infty(\Omega_2)}\|g\|_{L^2(\Omega_2)}^2.
\end{align*}
This proves \eqref{lemA-0}.

\textbf{Case (ii): $s=1$.}
Applying \eqref{lem-eq1} and using the definition of the inverse matrix, we have
\begin{align}\label{lem-eq4}
\|\nabla_x\Gamma^{-1}\|_{L^\infty(\Omega_1)}= \|(\nabla_y\Gamma)^{-1}\|_{L^\infty(\Omega_2)}\leq \|\frac{1}{\det(\nabla_y\Gamma)}\|_{L^\infty(\Omega_2)}\|\nabla_y\Gamma\|_{L^\infty(\Omega_2)}.
\end{align}
%
When $s=1$,
we denote
\begin{align*}
\mathbf{A}=(\mathbf{A}_{ij})_{i,j=1,2}=\nabla_y\Gamma,\quad \mathbf{B}=(\mathbf{B}_{ij})_{i,j=1,2}=(\nabla_y\Gamma)^{-\top}.
\end{align*}
By the chain rule, we have
\begin{align}\label{lem-eq6}
\begin{split}
&\p_{y_j}(f\circ\Gamma)=\sum_{i=1}^{2}\mathbf{A}_{ij} (\p_{x_i}f)\circ\Gamma,\\& \p_{x_j}(g\circ\Gamma^{-1})=\sum_{i=1}^{2}\mathbf{B}_{ji}\circ\Gamma^{-1}(\p_{y_i}g)\circ\Gamma^{-1}=\sum_{i=1}^{2}\big(\mathbf{B}_{ji}\p_{y_i}g\big)\circ\Gamma^{-1}.
\end{split}
\end{align}
We deduce from \eqref{lem-eq6} and (i) that
\begin{align*}
\|f\circ\Gamma\|_{\dot{H}^1(\Omega_2)}&\lesssim \|\mathbf{A}\|_{L^\infty(\Omega_2)}\|(\nabla_xf)\circ\Gamma\|_{L^2(\Omega_2)}\\&\lesssim \|\nabla_y\Gamma\|_{L^\infty(\Omega_2)}\|\frac{1}{\det(\nabla_y\Gamma)}\|_{L^\infty(\Omega_2)}^\frac12 \|f\|_{\dot{H}^1(\Omega_1)}.
\end{align*}
Similarly, by \eqref{lem-eq4} and (i), we get
\begin{align*}
\|g\circ\Gamma^{-1}\|_{\dot{H}^1(\Omega_1)}\lesssim& \|\mathbf{B}\|_{L^\infty(\Omega_2)}\|(\nabla_yg)\circ\Gamma^{-1}\|_{L^2(\Omega_1)}\\\lesssim& \|\frac{1}{\det(\nabla_y\Gamma)}\|_{L^\infty(\Omega_2)}\|\nabla_y\Gamma\|_{L^\infty(\Omega_2)}\|{\det(\nabla_y\Gamma)}\|_{L^\infty(\Omega_2)}^\frac12 \|g\|_{\dot{H}^1(\Omega_2)}.
\end{align*}
Hence, we complete the proof of \eqref{lemA-1}.

\textbf{Case (iii): $s=2$.}
Applying \eqref{lem-eq6}, we obtain
\begin{align*}
\begin{split}
\|f\circ\Gamma\|_{{H}^2(\Omega_2)}&\lesssim\|f\circ\Gamma\|_{{L}^2(\Omega_2)} +\|\mathbf{A}^\top(\nabla_xf)\circ\Gamma\|_{{H}^{1}(\Omega_2)}\\&\lesssim \|f\circ\Gamma\|_{{L}^2(\Omega_2)}+\big(\|\nabla_y\Gamma\|_{{L}^{\infty}(\Omega_2)}+\|\nabla_y^2\Gamma\|_{{H}^{1}(\Omega_2)}\big)\|(\nabla_xf)\circ\Gamma\|_{H^1(\Omega_2)}.
\end{split}
\end{align*}
We then use \eqref{lemA-0} and \eqref{lemA-1} to deduce
\begin{align*}
\begin{split}
\|f\circ\Gamma\|_{{H}^2(\Omega_2)}\lesssim&\big(1+\|\nabla_y\Gamma\|_{{L}^{\infty}(\Omega_2)}^2+\|\nabla_y^2\Gamma\|_{{H}^{1}(\Omega_2)}^2\big)\|\frac{1}{\det(\nabla_y\Gamma)}\|_{L^\infty(\Omega_2)}^\frac12 \|f\|_{{H}^2(\Omega_1)}.
\end{split}
\end{align*}
Similarly, by \eqref{lem-eq6}, \eqref{lemA-0} and \eqref{lemA-1}, we have
\begin{align}\label{lem-eq8-2}
\begin{split}
\|g\circ\Gamma^{-1			}\|_{{H}^2(\Omega_1)}\lesssim&\|g\circ\Gamma^{-1}\|_{{L}^2(\Omega_1)} +\|\big(\mathbf{B}\nabla_yg\big)\circ\Gamma^{-1}\|_{{H}^{1}(\Omega_1)}\\\lesssim&\big(1+\|\frac{1}{\det(\nabla_y\Gamma)}\|_{L^\infty(\Omega_2)}\|\nabla_y\Gamma\|_{L^\infty(\Omega_2)}\big)\|\det(\nabla_y\Gamma)\|_{L^\infty(\Omega_2)}^\frac12\\&\times\big(\|g\|_{{L}^2(\Omega_2)}+\|\mathbf{B}\nabla_yg\|_{{H}^1(\Omega_2)}\big) .
\end{split}
\end{align}
The cofactor formula and differentiation of the inverse transpose give
\[
 \mathbf B=(\det\mathbf A)^{-1}\operatorname{cof}\mathbf A,
 \qquad
 \partial_{y_j}\mathbf B
 =-\mathbf B(\partial_{y_j}\mathbf A)^\top\mathbf B.
\]
Consequently,
\begin{align*}
&\|\mathbf{B}\|_{L^\infty(\Omega_2)}\lesssim \|\frac{1}{\det (\nabla_y\Gamma)}\|_{L^\infty(\Omega_2)}\|\nabla_y\Gamma\|_{L^\infty(\Omega_2)},\\&\|\nabla_y\mathbf{B}\|_{L^4(\Omega_2)}\lesssim\big(\|\frac{1}{\det (\nabla_y\Gamma)}\|_{L^\infty(\Omega_2)}+ \|\frac{1}{\det (\nabla_y\Gamma)}\|_{L^\infty(\Omega_2)}^2\|\nabla_y\Gamma\|_{L^\infty(\Omega_2)}^2\big)\|\nabla_y^2\Gamma\|_{H^1(\Omega_2)}.
\end{align*}
Hence,
\begin{align}\label{lem-eq8-3}
\begin{split}
\|\mathbf{B}\nabla_y g\|_{H^1(\Omega_2)}
\lesssim{}&
\|\mathbf{B}\|_{L^\infty(\Omega_2)}
\|\nabla_y g\|_{H^1(\Omega_2)}
+
\|\nabla_y\mathbf{B}\|_{L^4(\Omega_2)}
\|\nabla_y g\|_{L^4(\Omega_2)}\\
\lesssim{}&
\big(
\|\frac{1}{\det(\nabla_y\Gamma)}\|_{L^\infty(\Omega_2)}
\|\nabla_y\Gamma\|_{L^\infty(\Omega_2)}\\
&\quad+
\|\frac{1}{\det(\nabla_y\Gamma)}\|_{L^\infty(\Omega_2)}^2
\|\nabla_y\Gamma\|_{L^\infty(\Omega_2)}^2
\|\nabla_y^2\Gamma\|_{H^1(\Omega_2)}
\big)
\|g\|_{H^2(\Omega_2)}.
\end{split}
\end{align}
Plugging \eqref{lem-eq8-3} into \eqref{lem-eq8-2}, we have
\begin{align*}
\begin{split}
&\|g\circ\Gamma^{-1			}\|_{{H}^2(\Omega_1)}\lesssim\|g\circ\Gamma^{-1}\|_{{L}^2(\Omega_1)} +\|\big(\mathbf{B}\nabla_yg\big)\circ\Gamma^{-1}\|_{{H}^{1}(\Omega_1)}\\&\lesssim\big(1+\|\frac{1}{\det(\nabla_y\Gamma)}\|_{L^\infty(\Omega_2)}^3\big)\big(1+\|\nabla_y\Gamma\|_{L^\infty(\Omega_2)}^4+\|\nabla_y^2\Gamma\|_{H^1(\Omega_2)}^4\big)\|\det(\nabla_y\Gamma)\|_{L^\infty(\Omega_2)}^\frac12\|g\|_{{H}^2(\Omega_2)} .
\end{split}
\end{align*}
This completes the proof of \eqref{lemA-2}.

\textbf{Case (iv): $s\geq 3$.}
We first prove
\begin{align}\label{fHs}
\|f\circ\Gamma\|_{{H}^s(\Omega_2)}\lesssim&\big(1+\|\nabla_y\Gamma\|_{{L}^{\infty}(\Omega_2)}^s+\|\nabla_y^2\Gamma\|_{{H}^{s-2}(\Omega_2)}^s\big)\|\frac{1}{\det(\nabla_y\Gamma)}\|_{L^\infty(\Omega_2)}^\frac12 \|f\|_{{H}^s(\Omega_1)}.
\end{align}
In fact, when $s\geq 3$, by \eqref{lem-eq6}, we compute
\begin{align*}
\|f\circ\Gamma\|_{{H}^s(\Omega_2)}&\lesssim\|f\circ\Gamma\|_{{L}^2(\Omega_2)}+\|\mathbf{A}^\top((\nabla_x f)\circ\Gamma)\|_{{H}^{s-1}(\Omega_2)}\\&\lesssim\|f\circ\Gamma\|_{{L}^2(\Omega_2)}+\big(\|\mathbf{A}\|_{{L}^{\infty}(\Omega_2)}+\|\nabla_y\mathbf{A}\|_{{H}^{s-2}(\Omega_2)}\big)\|(\nabla_x f)\circ\Gamma\|_{{H}^{s-1}(\Omega_2)}.
\end{align*}
Using the result in (i) and (iii), we infer by induction that
\begin{align*}
\|f\circ\Gamma\|_{{H}^s(\Omega_2)}&\lesssim \|\frac{1}{\det(\nabla_y\Gamma)}\|_{L^\infty(\Omega_2)}^\frac12\|f\|_{L^2(\Omega_1)}+\big(\|\nabla_y\Gamma\|_{L^{\infty}(\Omega_2)}+\|\nabla_y^2\Gamma\|_{H^{s-2}(\Omega_2)}\big)\\&\qquad\times\big(1+\|\nabla_y\Gamma\|_{L^{\infty}(\Omega_2)}^{s-1}+\|\nabla_y^2\Gamma\|_{H^{s-2}(\Omega_2)}^{s-1}\big)\|\frac{1}{\det(\nabla_y\Gamma)}\|_{L^\infty(\Omega_2)}^\frac12 \|\nabla_x f\|_{{H}^{s-1}(\Omega_1)}\\&\lesssim\big(1+\|\nabla_y\Gamma\|_{L^{\infty}(\Omega_2)}^{s}+\|\nabla_y^2\Gamma\|_{H^{s-2}(\Omega_2)}^s\big)\|\frac{1}{\det(\nabla_y\Gamma)}\|_{L^\infty(\Omega_2)}^\frac12 \|f\|_{{H}^s(\Omega_1)}.
\end{align*}
This completes the proof of \eqref{fHs}.

For $\|g\circ\Gamma^{-1}\|_{{H}^s(\Omega_1)}$, we first apply \eqref{fHs} and \eqref{lem-eq1} to get
\begin{align}\label{lem-ghs}
	\begin{split}
	\|g\circ\Gamma^{-1}\|_{{H}^s(\Omega_1)}\lesssim& \big(1+\|\mathbf{B}\|_{L^{\infty}(\Omega_2)}^s+\|\nabla_x(\mathbf{B}\circ \Gamma^{-1})\|_{H^{s-2}(\Omega_1)}^s\big)\\&\times\|\det(\nabla_y\Gamma)\|_{L^\infty(\Omega_2)}^\frac12 \|g\|_{{H}^s(\Omega_2)}.	
	\end{split}
\end{align}
Next we estimate $\|\nabla_x(\mathbf{B}\circ \Gamma^{-1})\|_{H^{s-2}(\Omega_1)}$. For $s\geq 3$, we have
\begin{align*}
\begin{split}
&\|\nabla_x(\mathbf{B}\circ \Gamma^{-1})\|_{H^{s-2}(\Omega_1)}\lesssim \|\mathbf{B}\circ\Gamma^{-1}(\nabla_y\mathbf{B})\circ\Gamma^{-1}\|_{{H}^{s-2}(\Omega_1)}\\
&\lesssim \sum_{j=1}^{s-3}\|\mathbf{B}\|_{L^\infty(\Omega_2)}^{j}\|(\nabla_y^j\mathbf{B})\circ\Gamma^{-1}\|_{L^{2}(\Omega_1)}\\&\quad+ \sum_{j=1}^{s-3}\|\mathbf{B}\|_{L^\infty(\Omega_2)}^{j-1}\|\nabla_y^j\mathbf{B}\|_{L^\infty(\Omega_2)}\|\nabla_x(\mathbf{B}\circ\Gamma^{-1})\|_{H^{s-2-j}(\Omega_1)}\\&\quad+\|\mathbf{B}\|_{L^\infty(\Omega_2)}^{s-2}\|(\nabla_y^{s-2}\mathbf{B})\circ\Gamma^{-1}\|_{H^1(\Omega_1)}\big(1+\|(\nabla_y\mathbf{B})\circ\Gamma^{-1}\|_{H^{1}(\Omega_1)}\big).
\end{split}
\end{align*}
Hence, we can prove by induction that
\begin{align}\label{lem-Bphi}
\begin{split}
&\|\nabla_x(\mathbf{B}\circ \Gamma^{-1})\|_{H^{s-2}(\Omega_1)}\\&\lesssim\|\nabla_y\mathbf{B}\|_{H^{s-2}}\big(\|\mathbf{B}\|_{L^{\infty}(\Omega_2)}+\|\nabla_y\mathbf{B}\|_{H^{s-2}(\Omega_2)}\big)\Big(1+\big(\|\mathbf{B}\|_{L^{\infty}(\Omega_2)}+\|\nabla_y\mathbf{B}\|_{H^{s-2}(\Omega_2)}\big)^{s}\big)\\&\quad\times \big(1+\|{\det(\nabla_y\Gamma)}\|_{L^\infty(\Omega_2)}^{\frac{3}{2}(s-2)}\Big)\big(1+\|\frac{1}{\det(\nabla_y\Gamma)}\|_{L^\infty(\Omega_2)}^{(s-1)(s-2)}\|{\det(\nabla_y\Gamma)}\|_{L^\infty(\Omega_2)}^{(s-1)(s-2)}\big).
\end{split}
\end{align}
It remains to estimate $\|\nabla_y\mathbf{B}\|_{{H}^{s-2}(\Omega_2)}$ for $s\geq 3$. In fact, by computing $\nabla_y^j\mathbf{B}$ with $1\leq j\leq s-1$, we can obtain that
\begin{align}\label{lem-BHs}
\begin{split}
\|\nabla_y\mathbf{B}\|_{{H}^{s-2}(\Omega_2)}\lesssim &\|\frac{1}{\det(\nabla_y\Gamma)}\|_{L^\infty(\Omega_2)} \|\nabla_y^2\Gamma\|_{{H}^{s-2}(\Omega_2)}\\&\times\big(1+\|\frac{1}{\det(\nabla_y\Gamma)}\|_{L^\infty(\Omega_2)}^{3s}+\|\nabla_y\Gamma\|_{{L}^{\infty}(\Omega_2)}^{3s}+\|\nabla_y^2\Gamma\|_{{H}^{s-2}(\Omega_2)}^{3s}\big).
\end{split}
\end{align}
Plugging \eqref{lem-BHs} into \eqref{lem-Bphi}, we get
\begin{align}\label{lem-Bhs}
\begin{split}
\|\nabla_x(\mathbf{B}\circ \Gamma^{-1})\|_{H^{s-2}(\Omega_1)}&\lesssim\big(\|\nabla_y\Gamma\|_{L^{\infty}(\Omega_2)}+\|\nabla_y^2\Gamma\|_{H^{s-2}(\Omega_2)}\big)\\&\qquad\times\Big(1+\big(\|\nabla_y\Gamma\|_{L^{\infty}(\Omega_2)}+\|\nabla_y^2\Gamma\|_{H^{s-2}(\Omega_2)}\big)^{5s^2+s+5}\Big)\\&\qquad\times\big(1+\|\det{\nabla_y\Gamma}\|_{L^{\infty}(\Omega_2)}^{\frac{3}{2}s-2}\big) \big(1+\|\frac{1}{\det(\nabla_y\Gamma)}\|_{L^\infty(\Omega_2)}^{5s^2+s+5}\big).
\end{split}
\end{align}
Combining \eqref{lem-ghs} with \eqref{lem-Bhs}, we can finish the proof of \eqref{lemA-3}.
\end{proof}

We now estimate the quotient of two functions.
\begin{lemma}\label{lem-g1g2}
Let $g_1$ and $g_2$ be functions defined on a  smooth Sobolev extension
domain $\Omega\subseteq\mathbb R^2$, and assume that
$g_1\ge m>0$.
Then, for every integer $s\ge0$, the following estimates hold:
\begin{align}
 &\big\|\nabla\frac1{g_1}\big\|_{H^s}
 \le C\|\nabla g_1\|_{H^s}
       \big(1+\|\nabla g_1\|_{H^s}^{s}\big),
 \label{gHs}\\
 &\big\|\frac{g_2}{g_1}\big\|_{H^s}
 \le C\|g_2\|_{H^s}
       \big(1+\|\nabla g_1\|_{H^s}^{s}\big),
 \qquad s=0,1,
 \label{g1g2Hs-1}\\
 &\big\|\frac{g_2}{g_1}\big\|_{H^s}
 \le C\|g_2\|_{H^s}
       \big(1+\|\nabla g_1\|_{H^{s-1}}^{s}\big),
 \qquad s\ge2.
 \label{g1g2Hs-2}
\end{align}
The constant $C$ depends only on $s$, $m$, and $\Omega$.
\end{lemma}
\begin{proof}
We begin by establishing \eqref{gHs}.
For $j=1,2$ a direct computation gives
\begin{align*}
&\p_{j}\frac{1}{g_1}=-\frac{\p_{j}g_1}{g_1^2},\quad \nabla\p_{j}\frac{1}{g_1}=-\frac{\nabla\p_{j}g_1}{g_1^2}+\frac{2\nabla g_1\p_{j}g_1}{g_1^3},
\end{align*}
By using the H\"older inequality, this yields \eqref{gHs} for $s=0,\,1$.
For $s\ge 2$, one argues by induction, repeatedly applying
$\nabla$ to the above identities and invoking the
H\"older inequality again; the details are omitted.

Next we show \eqref{g1g2Hs-1} and \eqref{g1g2Hs-2}. Notice that
\begin{align*}
\nabla\frac{g_2}{g_1}=\frac{\nabla g_2}{g_1}+g_2\nabla \frac{1}{g_1}=\frac{\nabla g_2}{g_1}- \frac{g_2\nabla g_1}{g_1^2}.
\end{align*}
By using H\"older's inequality, we have
\begin{align*}
\|\frac{g_2}{g_1}\|_{L^2}\leq C\|g_2\|_{L^2},\quad
\|\nabla\frac{g_2}{g_1}\|_{L^2}\leq C\|g_2\|_{H^1}(1+\|\nabla g_1\|_{H^1}).
\end{align*}
This proves the case $s=0, 1$.
For $s\geq 2$, it follows from the product laws and \eqref{gHs} that
\begin{align*}
\|\frac{g_2}{g_1}\|_{{H}^s}&\leq C\|{ g_2}\|_{{H}^s}\big(\|\frac{1}{ g_1}\|_{L^\infty}+\|\nabla\frac{1}{g_1} \|_{{H}^{s-1}}\big)\leq C\|{ g_2}\|_{{H}^s}(1+\|\nabla { g_1}\|_{{H}^{s-1}}^s).
\end{align*}
This completes the proof of \eqref{g1g2Hs-2}.
\end{proof}
\begin{proof}[Proof of Lemma \ref{coordinate-lemma}]
Throughout the proof, norms of the initial data before composition
are taken in the $y$-coordinates, and $C=C(a,m,L)$, unless additional
dependence is explicitly indicated. We may assume that $\epsilon_*$
in \eqref{asmp} is at most one. By \eqref{xi-cond0},
\eqref{asmp}, and Lemma~\ref{lem-g1g2},
\begin{equation}\label{app-ratio-est}
 \big\|\frac{b_0^2}{b_0^1}\big\|_{H^k}
 \le C\|b_0^2\|_{H^k},
 \qquad 0\le k\le a+1.
\end{equation}

\medskip
\noindent\textbf{Estimates of $\|\nabla y\|_{L^\infty}$
and $\|\rme^{\pm h}\|_{L^\infty}$.}
By \eqref{pz2y2}, Sobolev embedding, and
\eqref{app-ratio-est}, we have
\[
 \|h\|_{L^\infty}
 \le C\big\|\frac{b_0^2}{b_0^1}\big\|_{H^3}
 \le C\|b_0^2\|_{H^3}\le C\epsilon_*.
\]
Taking $\epsilon_*$ sufficiently small that
$C\epsilon_*\le\log 2$, we obtain
\begin{equation}\label{pzy2-es1}
 \frac 12 \leq \rme^{\pm h} \leq 2.
\end{equation}
On the other hand,
\begin{equation}\label{pzy2-es2}
 \big\|\frac{b_0^2}{b_0^1}(y(z))\big\|_{L^\infty}
 \le \big\|\frac{b_0^2}{b_0^1}\big\|_{L^\infty}
 \le C\|b_0^2\|_{H^2}.
\end{equation}
By \eqref{nazy}, \eqref{pzy2-es1}, and \eqref{pzy2-es2},
we deduce that
\[
 \|\nabla y\|_{L^\infty}\le C.
\]
Since $\det\nabla y=\rme^{-h}$, these estimates also provide
the Jacobian bounds required in Lemma~\ref{lemA}.

\medskip
\noindent\textbf{Estimates of $\|\nabla y-I\|_{H^a}$
and $\|\rme^{\pm h}-1\|_{H^a}$.}
We shall first prove by induction that
\begin{equation}\label{naz2y-es4}
 \|\nabla^2y\|_{H^{s-2}}
 \le C\|b_0^2\|_{H^s},
 \qquad 3\le s\le a+1.
\end{equation}
By \eqref{nazy},
\begin{equation}\label{na2x-1}
 \|\nabla^2y\|_{H^{s-2}}
 \le C\Big(
 \|\nabla\rme^{-h}\|_{H^{s-2}}
 +\big\|
 \nabla\big(\frac{b_0^2}{b_0^1}(y(z))\big)
 \big\|_{H^{s-2}}
 \Big).
\end{equation}

We first consider $s=3$.
Differentiating \eqref{pz2y2}, we obtain
\begin{align}
 \partial_1\rme^{-h}
 &=\rme^{-h}
   \partial_{y_2}\big(\frac{b_0^2}{b_0^1}\big)(y(z)),
 \label{px1mu}\\
 \partial_2\rme^{-h}
 &=\rme^{-h}\int_{0}^{z_1}
   \partial_{y_2}^2\big(\frac{b_0^2}{b_0^1}\big)
       (y(z_1',z_2))\rme^{-h(z_1',z_2)}
       \,\mathrm dz_1'.
 \label{p2-1}
\end{align}
Invoking Lemma~\ref{lemA}(i), \eqref{pzy2-es1}, and
\eqref{app-ratio-est}, we deduce from \eqref{px1mu}
and \eqref{p2-1} that
\begin{equation}\label{pmux2-1}
 \|\nabla\rme^{-h}\|_{L^2}
 \le C\|b_0^2\|_{H^2}.
\end{equation}
The same change-of-variables argument in $L^4$, together
with the Sobolev embedding $H^1\hookrightarrow L^4$,
gives
\begin{equation}\label{p2L4}
 \|\nabla\rme^{-h}\|_{L^4}
 \le C\|b_0^2\|_{H^3}.
\end{equation}

By \eqref{px1mu}, Lemma~\ref{lemA}(ii),
\eqref{app-ratio-est}, and \eqref{pmux2-1}, we have
\begin{equation}\label{pmux2-2}
 \begin{aligned}
 \|\partial_1\rme^{-h}\|_{H^1}
 &\le C\Big(
 \big\|
 \partial_{y_2}\big(\frac{b_0^2}{b_0^1}\big)(y(z))
 \big\|_{H^1}
 +\|\nabla\rme^{-h}\|_{L^2}
  \big\|
  \partial_{y_2}\big(\frac{b_0^2}{b_0^1}\big)
  \big\|_{L^\infty}
 \Big)\le C\|b_0^2\|_{H^3}.
 \end{aligned}
\end{equation}
Similarly, differentiating \eqref{p2-1} and applying
H\"older's inequality, Lemma~\ref{lemA}(ii),
\eqref{pzy2-es1}, and \eqref{p2L4}, we obtain
\begin{equation}\label{pmux2-3}
 \begin{aligned}
 \|\partial_2\rme^{-h}\|_{H^1}
 &\le C\Big(
 \big\|
 \partial_{y_2}^2\big(\frac{b_0^2}{b_0^1}\big)(y(z))
 \big\|_{H^1}
 +\|\nabla\rme^{-h}\|_{L^4}
  \big\|
  \partial_{y_2}^2\big(\frac{b_0^2}{b_0^1}\big)(y(z))
  \big\|_{L^4}
 \Big)\le C\|b_0^2\|_{H^3},
 \end{aligned}
\end{equation}
where we have also used
$\|b_0^2\|_{H^3}\le\epsilon_*\le1$.
Combining \eqref{pmux2-2} with \eqref{pmux2-3},
we get
\begin{equation}\label{naeh-es1}
 \|\nabla\rme^{-h}\|_{H^1}
 \le C\|b_0^2\|_{H^3}.
\end{equation}

Next, we estimate
$\|\nabla((b_0^2/b_0^1)(y(z)))\|_{H^1}$.
Direct computation shows that
\begin{align*}
 \partial_1\Big(\frac{b_0^2}{b_0^1}(y(z))\Big)
 &=
 \partial_{y_1}\big(\frac{b_0^2}{b_0^1}\big)(y(z))
 +\partial_{y_2}\big(\frac{b_0^2}{b_0^1}\big)(y(z))
       \frac{b_0^2}{b_0^1}(y(z)),\\
 \partial_2\Big(\frac{b_0^2}{b_0^1}(y(z))\Big)
 &=
 \partial_{y_2}\big(\frac{b_0^2}{b_0^1}\big)(y(z))\rme^{-h}
 =\partial_1\rme^{-h}.
\end{align*}
Thus, by Lemma~\ref{lemA}(ii),
\eqref{app-ratio-est}, and \eqref{pmux2-2},
\begin{equation}\label{pmux2-4}
 \begin{aligned}
 \big\|
 \nabla\Big(\frac{b_0^2}{b_0^1}(y(z))\Big)
 \big\|_{H^1}
 &\le C\Big[
 \big\|\frac{b_0^2}{b_0^1}\big\|_{H^2}
 \Big(1+\big\|\frac{b_0^2}{b_0^1}\big\|_{H^3}\Big)
 +\|\partial_1\rme^{-h}\|_{H^1}
 \Big]\le C\|b_0^2\|_{H^3}.
 \end{aligned}
\end{equation}
Plugging \eqref{naeh-es1} and \eqref{pmux2-4}
into \eqref{na2x-1}, we obtain
\begin{equation}\label{xh1}
 \|\nabla^2y\|_{H^1}
 \le C\|b_0^2\|_{H^3}.
\end{equation}
This proves \eqref{naz2y-es4} for $s=3$.

Now let $4\le s\le a+1$.
Using \eqref{na2x-1}, we have
\begin{equation}\label{app-map-reduction}
 \|\nabla^2y\|_{H^{s-2}}
 \le C\Big(
 \big\|\frac{b_0^2}{b_0^1}(y(z))\big\|_{H^{s-1}}
 +\|\rme^{-h}-1\|_{H^{s-1}}
 \Big).
\end{equation}
We first estimate the first term on the right-hand side.
Applying Lemma~\ref{lemA}(iv), \eqref{app-ratio-est},
and the bounds from \eqref{nazy} and \eqref{pzy2-es1}--\eqref{pzy2-es2}, we obtain
\begin{equation}\label{nx2hs-1}
 \begin{aligned}
 \big\|\frac{b_0^2}{b_0^1}(y(z))\big\|_{H^{s-1}}
 &\le C\Big(
 1+\|\nabla y\|_{L^\infty}^{s-1}
   +\|\nabla^2y\|_{H^{s-3}}^{s-1}
 \Big)
 \|\rme^h\|_{L^\infty}^{1/2}
 \big\|\frac{b_0^2}{b_0^1}\big\|_{H^{s-1}}\\
 &\le C\Big(
 1+\|\nabla^2y\|_{H^{s-3}}^{s-1}
 \Big)\|b_0^2\|_{H^{s-1}}.
 \end{aligned}
\end{equation}

It remains to estimate $\|\rme^{-h}-1\|_{H^{s-1}}$.
For this purpose, we first estimate $\|h\|_{H^{s-1}}$.
By \eqref{pz2y2}, for $j,\ell\ge0$ with $j+\ell\le s-1$,
\[
 \partial_1^j\partial_2^\ell h(z)
 =
 \begin{cases}
 \displaystyle
 -\int_{0}^{z_1}
 \partial_2^\ell\Big[
 \partial_{y_2}\big(\frac{b_0^2}{b_0^1}\big)
       (y(z_1',z_2))
 \Big]\,\mathrm dz_1',
 &j=0,\\[8pt]
 \displaystyle
 -\partial_1^{j-1}\partial_2^\ell
 \Big[
 \partial_{y_2}\big(\frac{b_0^2}{b_0^1}\big)(y(z))
 \Big],
 &j\ge1.
 \end{cases}
\]
Hence,
\[
 \|h\|_{H^{s-1}}
 \le C\big\|
 \partial_{y_2}\big(\frac{b_0^2}{b_0^1}\big)(y(z))
 \big\|_{H^{s-1}}.
\]
By Lemma~\ref{lemA}(iv), \eqref{app-ratio-est},
and the bounds from \eqref{nazy} and \eqref{pzy2-es1}--\eqref{pzy2-es2},
\begin{equation}\label{ps-5}
 \begin{aligned}
 \big\|
 \partial_{y_2}\big(\frac{b_0^2}{b_0^1}\big)(y(z))
 \big\|_{H^{s-1}}
 &\le C\Big(
 1+\|\nabla^2y\|_{H^{s-3}}^{s-1}
 \Big)
 \big\|
 \partial_{y_2}\big(\frac{b_0^2}{b_0^1}\big)
 \big\|_{H^{s-1}}\\
 &\le C\Big(
 1+\|\nabla^2y\|_{H^{s-3}}^{s-1}
 \Big)\|b_0^2\|_{H^s}.
 \end{aligned}
\end{equation}

Moreover, iterating the Sobolev product estimate gives
\[
 \|h^n\|_{H^{s-1}}
 \le nC_s^{n-1}\|h\|_{L^\infty}^{n-1}
       \|h\|_{H^{s-1}},
 \qquad n\ge1,
\]
where $C_s\ge1$ depends only on $s$.
Expanding $\rme^{\pm h}-1$ into its power series and using
the uniform bound for $\|h\|_{L^\infty}$, we obtain
\[
 \|\rme^{\pm h}-1\|_{H^{s-1}}
 \le \sum_{n=1}^{\infty}\frac{\|h^n\|_{H^{s-1}}}{n!}
 \le \rme^{C_s\|h\|_{L^\infty}}\|h\|_{H^{s-1}}
 \le C\|h\|_{H^{s-1}}.
\]
Consequently, \eqref{ps-5} implies
\begin{equation}\label{app-h-high}
 \|h\|_{H^{s-1}}
 +\|\rme^{\pm h}-1\|_{H^{s-1}}
 \le C\Big(
 1+\|\nabla^2y\|_{H^{s-3}}^{s-1}
 \Big)\|b_0^2\|_{H^s},
 \qquad 4\le s\le a+1.
\end{equation}
Plugging \eqref{nx2hs-1} and \eqref{app-h-high}
into \eqref{app-map-reduction}, we obtain
\begin{equation}\label{nx2hs-3}
 \|\nabla^2y\|_{H^{s-2}}
 \le C_0\big(
 1+\|\nabla^2y\|_{H^{s-3}}^{s-1}
 \big)\|b_0^2\|_{H^s},
 \qquad 4\le s\le a+1,
\end{equation}
where $C_0\ge1$ depends only on $a,\,m,\,L$.
Enlarging $C_0$ if necessary, we may also assume that
\eqref{xh1} holds with $C=C_0$.

Taking
\[
 \epsilon_*\le\min\big\{1,(2C_0)^{-1}\big\},
\]
we prove by induction that
\[
 \|\nabla^2y\|_{H^{s-2}}
 \le 2C_0\|b_0^2\|_{H^s},
 \qquad 3\le s\le a+1.
\]
The case $s=3$ follows from \eqref{xh1}.
Let $4\le s\le a+1$, and assume that the estimate
holds with $s$ replaced by $s-1$. By \eqref{asmp}
and the choice of $\epsilon_*$,
\[
 \|\nabla^2y\|_{H^{s-3}}
 \le 2C_0\|b_0^2\|_{H^{s-1}}
 \le 2C_0\epsilon_*
 \le1.
\]
Substituting this bound into \eqref{nx2hs-3}, we obtain
\[
 \|\nabla^2y\|_{H^{s-2}}
 \le C_0\big(
 1+\|\nabla^2y\|_{H^{s-3}}^{s-1}
 \big)\|b_0^2\|_{H^s}
 \le 2C_0\|b_0^2\|_{H^s}.
\]
This closes the induction and proves \eqref{naz2y-es4}
with $C=2C_0$.

Taking $s=a+1$ in \eqref{nx2hs-1} and \eqref{app-h-high},
and using the bound
$\|\nabla^2y\|_{H^{a-2}}\le2C_0\epsilon_*\le1$
just established, we deduce from \eqref{nazy} that
\begin{equation}\label{app-map-high}
 \begin{aligned}
 \|\nabla y-I\|_{H^a}
 +\|\rme^{\pm h}-1\|_{H^a}
 &\le C\Big(
 \big\|\frac{b_0^2}{b_0^1}(y(z))\big\|_{H^a}
 +\|\rme^{\pm h}-1\|_{H^a}
 \Big)\le C\|b_0^2\|_{H^{a+1}}.
 \end{aligned}
\end{equation}

\medskip
\noindent\textbf{Estimates of $\|B-I\|_{H^a}$
and $\|\rho-1\|_{H^a}$.}
By \eqref{def-B},
\[
 B-I=
 \begin{pmatrix}
  0&-\displaystyle\frac{b_0^2}{b_0^1}(y(z))\rme^h\\
  0&\rme^h-1
 \end{pmatrix}.
\]
Writing $\rme^h=1+(\rme^h-1)$ and applying the Sobolev
product estimate, we deduce from \eqref{nazy} and
\eqref{app-map-high} that
\[
 \begin{aligned}
 \|B-I\|_{H^a}
 &\le C\Big[
 \big\|\frac{b_0^2}{b_0^1}(y(z))\big\|_{H^a}
 \big(1+\|\rme^h-1\|_{H^a}\big)
 +\|\rme^h-1\|_{H^a}
 \Big]\le C\|b_0^2\|_{H^{a+1}}.
 \end{aligned}
\]
Moreover, $\rho-1=\langle \rme^{-h}-1\rangle$ implies
\[
 \|\rho-1\|_{H^a}
 \le C\|\rme^{-h}-1\|_{H^a}
 \le C\|b_0^2\|_{H^{a+1}}.
\]
For sufficiently small $\epsilon_0$, the bound for $\rho-1$
and Sobolev embedding give $\rho\ge1/2$.
Applying \eqref{g1g2Hs-2} with $g_1=\rho$, $g_2=1-\rho$,
and $s=a$, we obtain
\begin{align*}
 \|\rho^{-1}-1\|_{H^a}
 &=\Big\|\frac{1-\rho}{\rho}\Big\|_{H^a}\le C\|\rho-1\|_{H^a}
 \bigl(1+\|\pa_2\rho\|_{H^{a-1}}^a\bigr)
 \le C\|b_0^2\|_{H^{a+1}},
\end{align*}
where we used
$\|\pa_2\rho\|_{H^{a-1}}
\le\|\rho-1\|_{H^a}\le C\epsilon_0$. This proves \eqref{static-est3}.
The bounds \eqref{pzy2-es1} and \eqref{app-map-high},
together with Lemma~\ref{lemA}, also yield
\begin{equation}\label{app-composition}
 \|g(y(z))\|_{H^k}\le C\|g\|_{H^k},
 \qquad 0\le k\le a+1,
\end{equation}
for every integer $k$ in this range and every $g\in H^k$.

\medskip
\noindent\textbf{Estimates of $\nabla_zb_0^1(y(z))$.}
We first estimate $b_0^1(y(z))-\xi(z_2)$.
Since $y_2(0,z_2)=z_2$ and
$\partial_1y_2=(b_0^2/b_0^1)(y(z))$, we have
\[
 b_0^1(y(z))-\xi(z_2)
 =(b_0^1-\xi)(y(z))
 +\int_{0}^{z_1}
   \Big(\xi'\frac{b_0^2}{b_0^1}\Big)
       (y(z_1',z_2))\,\mathrm dz_1'.
\]
Using
\eqref{app-composition}, \eqref{app-ratio-est},
and the Sobolev product estimate, we obtain
\begin{equation}\label{b01-xi-small}
 \begin{aligned}
 \|b_0^1(y(z))-\xi(z_2)\|_{H^{a+1}}
 &\le C\Big(
       \|(b_0^1-\xi)(y(z))\|_{H^{a+1}}
       +\big\|
         \Big(\xi'\frac{b_0^2}{b_0^1}\Big)(y(z))
        \big\|_{H^{a+1}}
       \Big)\\
 &\le C\Big(
       \|b_0^1-\xi\|_{H^{a+1}}
       +\|\xi'\|_{H^{a+1}}
        \|b_0^2\|_{H^{a+1}}
       \Big)\\
 &\le C\|b_0-(\xi,0)^\top\|_{H^{a+1}},
 \end{aligned}
\end{equation}
where the last inequality follows from \eqref{xi-cond0}.
Since $\xi(z_2)$ is independent of $z_1$,
\eqref{b01-xi-small}, \eqref{xi-cond0}, and \eqref{asmp}
give
\begin{equation}\label{app-b01-derivatives}
 \begin{aligned}
 \|\partial_1b_0^1(y(z))\|_{H^a}
 &\le C\|b_0^1(y(z))-\xi(z_2)\|_{H^{a+1}}
 \le C\|b_0-(\xi,0)^\top\|_{H^{a+1}},\\
 \|\nabla_zb_0^1(y(z))\|_{H^a}
 &\le C\Big(
       \|b_0^1(y(z))-\xi(z_2)\|_{H^{a+1}}
       +\|\xi'\|_{H^a}
       \Big)
 \le C.
 \end{aligned}
\end{equation}

\medskip
\noindent\textbf{Estimates of $\gamma$ and $\Phi$.}
By the definition of $\Phi$,
\begin{equation}\label{Phi1-exp}
 \Phi^1
 =b_0^1(y(z))-\xi(z_2)-(\gamma-\xi)(z_2).
\end{equation}
By  \eqref{b01-xi-small}, Lemma \ref{lem-g1g2}, and the derivative bounds in \eqref{app-b01-derivatives},
\begin{equation*}
\Big\|\frac{\xi(z_2)}{b_0^1(y(z))}-1\Big\|_{H^{a+1}}\leq C\|b_0^1(y(z))-\xi(z_2)\|_{H^{a+1}}\leq C\epsilon_0.
\end{equation*}
Integrating in $z_1$ preserves this bound. Thus, for
sufficiently small $\epsilon_*$,
\[
\Big\|\av{\frac{\xi(z_2)}{b_0^1(y(z))}}-1\Big\|_{H^{a+1}}\leq C\epsilon_0,\qquad \av{\frac{\xi(z_2)}{b_0^1(y(z))}}\geq\frac12.
\]
Moreover, the identities \(\frac{\xi(z_2)}{b_0^1(y(z))}-1=(\xi(z_2)-b_0^1(y(z)))/b_0^1(y(z))\) and \(\gamma=\xi(z_2)/\av{\frac{\xi(z_2)}{b_0^1(y(z))}}\) give
\begin{equation*}
\gamma-\xi(z_2)
=\xi(z_2)\frac{1-\av{\frac{\xi(z_2)}{b_0^1(y(z))}}}{\av{\frac{\xi(z_2)}{b_0^1(y(z))}}}
=\frac{\av{\big(\frac{\xi(z_2)}{b_0^1(y(z))}\big)(b_0^1(y(z))-\xi(z_2))}}{\av{\frac{\xi(z_2)}{b_0^1(y(z))}}}.
\end{equation*}
Applying \eqref{g1g2Hs-2} to this quotient and using
\eqref{b01-xi-small} and the Sobolev product estimate
in the numerator, with
\[
 \frac{\xi(z_2)}{b_0^1(y(z))}
 =1+\Big(\frac{\xi(z_2)}{b_0^1(y(z))}-1\Big),
\]
we obtain
\begin{equation}\label{gm-xi-small}
\begin{aligned}
\|\gamma-\xi(z_2)\|_{H^{a+1}}
&\leq C\Big\|\av{\big(\frac{\xi(z_2)}{b_0^1(y(z))}\big)(b_0^1(y(z))-\xi(z_2))}\Big\|_{H^{a+1}}\\
&\leq C\|b_0^1(y(z))-\xi(z_2)\|_{H^{a+1}}.
\end{aligned}
\end{equation}
\eqref{xi-cond0}, \eqref{asmp},
and \eqref{gm-xi-small} yield
\[
 \|\gamma'\|_{H^a}
 \le \|\xi'\|_{H^a}
      +\|\gamma-\xi\|_{H^{a+1}}
 \le C.
\]
Together with \eqref{app-b01-derivatives}, this proves
\eqref{static-est2}.

Combining \eqref{Phi1-exp}, \eqref{b01-xi-small},
and \eqref{gm-xi-small}, we deduce that
\begin{equation}\label{Phi1-small}
 \|\Phi^1\|_{H^{a+1}}
 \le C\|b_0-(\xi,0)^\top\|_{H^{a+1}}.
\end{equation}
Since $\Phi^2=b_0^2(y(z))$, \eqref{app-composition} also gives
\begin{equation}\label{Phi2-small}
 \|\Phi^2\|_{H^{a+1}}
 \le C\|b_0^2\|_{H^{a+1}}
 \le C\|b_0-(\xi,0)^\top\|_{H^{a+1}}.
\end{equation}
Thus the required estimate for $\Phi$ follows from
\eqref{Phi1-small} and \eqref{Phi2-small}.

\medskip
\noindent\textbf{Estimate of the corrector $\widetilde Y$.} By \eqref{DefBaY}, $\pa_1\tY=-\Phi/b_0^1(y(z))$,
$\tY^1(0,z_2)=0$, and $\langle\tY^2\rangle=0$.
These normalizations also hold for the pure $z_2$ derivatives.
Applying the Poincar\'e inequality in $z_1$ to those derivatives
and using \eqref{b01-cond}, \eqref{app-b01-derivatives},
\eqref{Phi1-small}, \eqref{Phi2-small}, and
Lemma~\ref{lem-g1g2}, we obtain
\[
 \begin{aligned}
 \|\widetilde Y\|_{H^{a+1}}
 &\le C\big\|\frac{\Phi}{b_0^1(y(z))}\big\|_{H^{a+1}}\le C\|\Phi\|_{H^{a+1}}
 \le C\|b_0-(\xi,0)^\top\|_{H^{a+1}}.
 \end{aligned}
\]
Together with \eqref{app-b01-derivatives},
\eqref{Phi1-small}, and \eqref{Phi2-small},
this proves \eqref{static-est1}.


\medskip
\noindent\textbf{Estimates for the magnetic directional derivative.}
In this part of the proof, we write $b_0^1$ for
$b_0^1(y(z))$. Since $b_0^1\ge m/2$, the quotient
estimates \eqref{g1g2Hs-1} and \eqref{g1g2Hs-2},
together with \eqref{app-b01-derivatives}, give
\[
 \|\pa_1f\|_{H^s}
 =\Big\|\frac{\Db f}{b_0^1}\Big\|_{H^s}
 \le C\|\Db f\|_{H^s},
 \qquad 0\le s\le a.
\]
This proves \eqref{equiva-1}.

For $1\le s\le a$ and $|\alpha|\le s$, Leibniz's
formula and \eqref{app-b01-derivatives} imply
\begin{equation}\label{app-single-commutator}
 \begin{aligned}
 \|[\partial^\alpha,b_0^1]\pa_1f\|_{L^2}
 &\le C\|\pa_1f\|_{H^{s-1}},\\
 \|[\partial^\alpha,b_0^1]\Db f\|_{L^2}
 &\le C\|\Db f\|_{H^{s-1}}.
 \end{aligned}
\end{equation}
Indeed, the derivatives of $b_0^1$ of orders $1$ through
$a-1$ are bounded in $L^\infty$. When all $a$ derivatives
fall on $b_0^1$, we use the $L^2$ bound for
$\partial^\alpha b_0^1$ and
$H^{a-1}\hookrightarrow L^\infty$.

We next prove \eqref{equ-b0-norm}. For $s=0$, the two
norms are equal. For $1\le s\le a$, we use
\begin{equation}\label{app-magnetic-derivative-identity}
 \partial^\alpha(\Db f)
 =\Db\partial^\alpha f
  +[\partial^\alpha,b_0^1]\pa_1f.
\end{equation}
Applying \eqref{app-single-commutator} and
\eqref{equiva-1}, then summing over $|\alpha|\le s$,
we obtain
\[
 \begin{aligned}
 \sum_{|\alpha|\le s}\|\Db\partial^\alpha f\|_{L^2}^2
 &\le C\bigl(
 \|\Db f\|_{H^s}^2+\|\pa_1f\|_{H^{s-1}}^2
 \bigr)\le C\|\Db f\|_{H^s}^2.
 \end{aligned}
\]
For the reverse inequality, $b_0^1\ge m/2$ gives
\[
 \|\pa_1f\|_{H^{s-1}}^2
 \le \frac4{m^2}
 \sum_{|\nu|\le s-1}\|\Db\partial^\nu f\|_{L^2}^2.
\]
Combining this bound with
\eqref{app-magnetic-derivative-identity} and
\eqref{app-single-commutator}, we obtain
\[
 \begin{aligned}
 \|\Db f\|_{H^s}^2
 &\le C\sum_{|\alpha|\le s}
       \|\Db\partial^\alpha f\|_{L^2}^2
       +C\|\pa_1f\|_{H^{s-1}}^2\le C\sum_{|\alpha|\le s}
       \|\Db\partial^\alpha f\|_{L^2}^2.
 \end{aligned}
\]
This proves \eqref{equ-b0-norm}.

Finally, we estimate the commutator with the squared
coefficient. For $1\le s\le a$ and $|\alpha|\le s$,
Leibniz's formula gives
\begin{equation}\label{app-squared-commutator}
 \begin{aligned}
 [\partial^\alpha,(b_0^1)^2]\pa_1f
 &=[\partial^\alpha,b_0^1]\Db f
   +b_0^1[\partial^\alpha,b_0^1]\pa_1f\\
 &=[\partial^\alpha,b_0^1]\Db f
   +\sum_{0<\nu\le\alpha}
    \binom{\alpha}{\nu}
    (\partial^\nu b_0^1)
    \Db\partial^{\alpha-\nu}f.
 \end{aligned}
\end{equation}
For $0<|\nu|\le a-1$, we use
\eqref{app-b01-derivatives} and \eqref{equ-b0-norm}
at order $s-1$ to obtain
\[
 \|(\partial^\nu b_0^1)
       \Db\partial^{\alpha-\nu}f\|_{L^2}
 \le C\|\Db\partial^{\alpha-\nu}f\|_{L^2}
 \le C\|\Db f\|_{H^{s-1}}.
\]
The case $|\nu|=a$ occurs only when
$s=|\alpha|=a$ and $\nu=\alpha$. In this case,
Sobolev embedding gives
\[
 \|(\partial^\alpha b_0^1)\Db f\|_{L^2}
 \le \|\partial^\alpha b_0^1\|_{L^2}
      \|\Db f\|_{L^\infty}
 \le C\|\Db f\|_{H^{a-1}}.
\]
Substituting these bounds into
\eqref{app-squared-commutator} and using
\eqref{app-single-commutator} and \eqref{equiva-1},
we conclude that
\[
 \begin{aligned}
 &\|[\partial^\alpha,b_0^1]\pa_1f\|_{L^2}
 +\|[\partial^\alpha,(b_0^1)^2]\pa_1f\|_{L^2}\\
 &\qquad\le C\bigl(
 \|\pa_1f\|_{H^{s-1}}+\|\Db f\|_{H^{s-1}}
 \bigr)
 \le C\|\Db f\|_{H^{s-1}}.
 \end{aligned}
\]
This proves \eqref{magnetic-commutator-bound}
and completes the proof.
\end{proof}

\end{document}